\documentclass[final]{amsart}
\usepackage[margin=1.1in]{geometry} 
\usepackage{nicefrac}
\usepackage[style=alphabetic,url=false, isbn=false, doi=false,maxbibnames=99]{biblatex}
\AtEveryBibitem{
    \clearfield{urlyear}
    \clearfield{urlmonth}
    \clearfield{month}
    \clearfield{day}
}
\usepackage{graphicx,amsmath,amsthm,amssymb,amsfonts,hyperref,physics,etoolbox,tikz,tensor,xcolor,fancybox,mathrsfs}
\usepackage{cleveref}
\usetikzlibrary{arrows}
\usepackage{showkeys}
\numberwithin{equation}{section}
\renewcommand{\theequation}{\arabic{section}.\arabic{equation}}
\usepackage[nomargin,inline]{fixme}
\newcounter{subeqn}
\renewcommand{\thesubeqn}{\theequation\alph{subeqn}}
\newcommand{\subeqn}{%
  \refstepcounter{subeqn}
  \tag{\thesubeqn}
}
\makeatletter
\@addtoreset{subeqn}{equation}
\newcommand{\newseq}{%
  \refstepcounter{equation}
}
\makeatletter
\renewcommand*\FXLayoutInline[3]{%
  {\@fxuseface{inline}
  \ignorespaces\noindent \ovalbox{\hspace{.01\textwidth} \begin{minipage}{.95\textwidth}
  	#3 \fxnotename{#1}: #2
  \end{minipage}\hspace{.01\textwidth}}}
  \newline}
\makeatother
\newcommand{\Ben}[1]{\fxnote[inline,author=Ben]{\textcolor{blue!50}{ #1}}}
\newcommand{\Aiden}[1]{\fxnote[inline,author=Aiden]{\textcolor{red!50}{ #1}}}
\newtheorem{theorem}{Theorem}[section]
\newtheorem{proposition}[theorem]{Proposition}

\newtheorem{lemma}[theorem]{Lemma}

\newtheorem{definition}[theorem]{Definition}

\newtheorem{conjecture}[theorem]{Conjecture}
\newtheorem{algorithm}[theorem]{Algorithm}

\theoremstyle{remark}
\newtheorem{remark}[theorem]{Remark}
\newtheorem{examplex}{Example}

\newcommand{\Gr}{\mathrm{Gr}}
\newcommand{\C}{\mathbb{C}}
\newcommand{\R}{\mathbb{R}}
\newcommand{\Z}{\mathbb{Z}}
\newcommand{\N}{\mathbb{N}}

\newcommand{\GL}{\mathrm{GL}}

\newcommand{\Hom}{\mathrm{Hom}}
\newcommand{\taut}{\mathcal{T}}
\newcommand{\wcoor}{e}
\newcommand{\coord}{D}

\newcommand{\coul}{\mathbf{A}}
\newcommand{\vertex}{\mathsf{V}}
\newcommand{\edge}{\mathsf{E}}
\newcommand{\Bw}{\mathbf{w}}
\newcommand{\Bv}{\mathbf{v}}
\newcommand{\fM}{\mathfrak{M}}
\newcommand{\Lotimes}{\overset{\mathbb{L}}{\otimes}}
\newcommand{\End}{\operatorname{End}}
\newcommand{\Ext}{\operatorname{Ext}}
\newcommand{\mmod}{\operatorname{-mod}}
\newcommand{\wedgetwo}{\bigwedge{}^{\!\!2}}
\newcommand{\wedgek}{\bigwedge{}^{\!\!k}}
\newcommand{\wedgenk}{\bigwedge{}^{\!\!n-k}}
\newcommand{\trans}{\gamma}
\newcommand{\Sym}{\operatorname{Sym}}
\newcommand{\exten}{\mathcal{H}}
\newcommand\singdot[1]{
    \filldraw[fill=black, draw=black] (#1) circle (1.5pt)
}
\newcommand{\tilting}{\mathcal{Z}}
\newcommand{\Bi}{\mathbf{i}}
\newcommand{\Bj}{\mathbf{j}}
\newcommand{\Loc}[1]{\operatorname{Loc}(#1)}
\newcommand{\BK}{{\reflectbox{\rm R}}}
\newcommand\dotstrand{
    \begin{tikzpicture}[baseline=-3pt]
        \draw (0,-0.2) -- (0,0.2);
        \singdot{0,0};
    \end{tikzpicture}
}
\newcommand{\Zij}[3]{Z^{(#3)}_{#2,#1}}
\newcommand{\rk}{\operatorname{rk}}
\newcommand{\Proj}{\operatorname{Proj}}
\newcommand{\Spec}{\operatorname{Spec}}
\newcommand{\cL}{\mathcal{L}}
\newcommand{\cO}{\mathcal{O}}

\usetikzlibrary{decorations.pathreplacing,backgrounds,decorations.markings,calc}
\tikzset{wei/.style={draw=red,double=red!40!white,double distance=1.5pt,thin}}
\tikzset{fringe/.style={gray,postaction={decoration=border,decorate,draw,gray, segment length=4pt,thick}}}

\title{Tilting Generator for the $T^*\Gr(2,4)$ Coulomb Branch}
\author{Aiden Suter and Ben Webster}
\address{University of Waterloo \& Perimeter Institute, Waterloo, ON, Canada}

\email{adsuter@uwaterloo.ca}
\email{ben.webster@uwaterloo.ca}
\date{\today}

\begin{document}
\begin{abstract}
    Remarkable work of Kaledin, based on earlier joint work with Bezrukavnikov, has constructed a tilting generator of the category of coherent sheaves on a very general class of symplectic resolutions of singularities.

    In this paper, we give a concrete construction of this tilting generator on the cotangent bundle of $Gr(2,4)$, the Grassmannian of 2-planes in $\mathbb{C}^4$.  This construction builds on work of the second author describing these tilting bundles in terms of KLRW algebras, but in this low-dimensional case, we are able to describe our tilting generator as a sum of geometrically natural bundles on $T^*Gr(2,4)$: line bundles and their extensions, as well as the tautological bundle and its perpendicular.  
\end{abstract}
\maketitle
\tableofcontents

\section{Introduction}

One of the cornerstone ideas of geometric representation theory is the construction of equivalences between algebraic and geometric categories.  Perhaps the most famous example of such an equivalence is the Beilinson-Bernstein theorem, which relates modules over a semi-simple Lie algebra and D-modules on the corresponding flag variety.  The functor relating these categories is simply taking sections, and so is represented by the algebra of twisted differential operators itself, thought of as a twisted D-module.  

If, instead of D-modules, we consider the coherent sheaves on a variety which is not affine, then the structure of this category conspires against the existence of a similar theorem: in particular, this category will effectively never have projective objects, so there are no representable exact functors.  On the other hand, we can fix this defect by passing to derived categories and functors.  For example, by \cite[Th. 5.3.1]{BMRR}, if we consider a base field $\Bbbk$ of large positive characteristic, the category of coherent sheaves on $T^*G/B$ set-theoretically supported on $G/B$ is equivalent to the principal block of the category of finite-dimensional representations of the Lie algebra $U(\mathfrak{g})$ over $\Bbbk$.

As with any equivalence to the category of modules over a ring, the key to the equivalence is the image of the ring as a module over itself.  This is a little more subtle, due to the finiteness conditions, but this functor is of the form $\mathbb{R}\operatorname{Hom}(T,-)$ for $T$ a vector bundle which is a {\bf tilting generator}.

More generally, let $X$ be a variety over a field $\Bbbk$ which is projective over an affine variety (that is, a closed subvariety of $\mathbb{A}^k_{\Bbbk}\times \mathbb{P}^n_{\Bbbk}$).  
For a vector bundle $T$, let $A=\End(T)^{\operatorname{op}}$ and $A\mmod$ be the abelian category of finitely presented $A$-modules.  
\begin{definition}
    A vector bundle $\tilting$ on a quasi-projective variety $X$ is called a {\bf tilting generator} if the functor 
    \[\mathbb{R}\operatorname{Hom}(\tilting,-)\colon D^b(\mathsf{Coh}(X))\to D^b(A\mmod)\]
    is an equivalence of categories.
\end{definition}
Note that this functor sends $\tilting$ to $A$ as a left module over itself, and thus implies that:
\begin{enumerate}
    \item For all $i> 0$, we have $\Ext^i(\tilting,\tilting)=0$.
    \item The object $\tilting$ generates $D^b(\mathsf{Coh}(X))$, that is,
    $\mathbb{R}\operatorname{Hom}(\tilting,M)=0$ for a coherent sheaf $M$ if and only if $M=0$.  In fact, by adjunction, we have:
    \[M\cong \mathbb{R}\operatorname{Hom}(\tilting,M)\Lotimes_A \tilting.\]
\end{enumerate}
    In more concrete terms, under this equivalence, projective $A$-modules correspond to summands of $\tilting^{\oplus q}$.  
    The module $\mathbb{R}\operatorname{Hom}(\tilting,M)$ has a finite projective resolution, and we can reconstruct $M$ by turning each of these projectives into the corresponding projective.  

    Note, there are several natural operations which preserve the set of tilting generators:
    \begin{enumerate}
    \renewcommand{\theenumi}{\roman{enumi}}
        \item The dual of a tilting generator is a tilting generator.
        \item The tensor product of a tilting generator with any line bundle is a tilting generator.
        \item Any sheaf $\tilting'$ that is {\bf equiconstituted} with a tilting generator $\tilting$ (that is, a coherent sheaf is a summand of $(\tilting')^{\oplus r}$ for some $r$ if and only if it is a summand of $\tilting^{\oplus s}$ for some $s$) is a tilting generator.  
    \end{enumerate}

In general, there is no guarantee that a given variety carries a tilting generator, but a remarkable theorem of Kaledin \cite{KalDEQ} guarantees that a conical symplectic resolution of singularities must carry one.  In fact, the same algebra $A$ appears as the endomorphisms of a tilting generator on any such resolution of a given conical symplectic singularity, giving a beautiful proof of Kawamata's conjecture that K-equivalence implies D-equivalence in this case.  The algebra $A$ is a {\bf noncommutative crepant resolution of singularities}, so we can think of this result as proving the D-equivalence of two usual resolutions of the singularities factoring through a noncommutative resolution.  

Unfortunately, Kaledin's prescription is quite abstract and passes through quantization in characteristic $p$.  In particular, our description of coherent sheaves on $T^*G/B$ in terms of $U(\mathfrak{gl}_n)$ has obvious defects---it only covers the coherent sheaves with set-theoretic support on the zero-section and only works in characteristic $p$.

In order to fix this defect, we must replace $U(\mathfrak{gl}_n)$ by the algebra $A=\End(T)^{\operatorname{op}}$ for $T$ a corresponding tilting bundle, but while \cite{KalDEQ} guarantees the existence of such a bundle, it does not provide much guidance on how to compute with it.  

In some special cases, we can compute this tilting bundle and its endomorphism algebra more explicitly.  For example, for the cotangent bundle $T^*\mathbb{CP}^{n-1}$, we can show on abstract grounds that this tilting generator is a sum of line bundles.  An easy calulation confirms that the only sum of line bundles which gives a tilting generator is the pullback of the famous tilting generator of Beilinson $\mathcal{O}(a)\oplus \dots \oplus \mathcal{O}(a-n)$ for some $a\in \Z$ \cite{beilinsonCoherentSheaves1978} (see \Cref{sec:Pn} for more details).  However, it is rare that an algebraic variety has a tilting generator which is a sum of line bundles, so we look for richer special classes that we can study.  

The one that will be relevant for us is when the symplectic singularity is a Coulomb branch in the sense of Braverman, Finkelberg, and Nakajima \cite{BFN}.  This case is studied in the work of the second author \cite{websterCoherentSheaves2019,WebcohII}, especially in the case of {\bf bow varieties}, which are realized as Coulomb branches by work of Nakajima and Takayama \cite{nakajimaCherkisBow2017}. In particular, the algebra $A$ in this case can be realized as a cylindrical version of a KLRW algebra.  

This description has its own defects, in that the resulting coherent sheaves, described in terms of modules over the projective coordinate ring of a resolution, are hard to compare with other descriptions of coherent sheaves on these varieties.

Our aim in this paper is to give a description of the resulting coherent sheaves in one of the most familiar of these singularities, the cotangent bundle $T^*X$ of the Grassmannian $X=\Gr(2,4)$.  The sheaves we will need below are: the structure sheaf $\mathcal{O}$, the unique ample line bundle $\cL$ that generates $\operatorname{Pic}(T^*X)\cong \Z$, the tautological bundle $\mathcal{T}\subset \mathcal{O}^{\oplus 4},$ and $\mathcal{H}$, the unique extension $\cL^{-2}\to \mathcal{H}\to \mathcal{O}$ such that $H^1(T^*X;\mathcal{H})=0$.  
\begin{theorem}
    Under the natural operations mentioned above, any of Kaledin's tilting bundles on $T^*X$ will be equivalent to the tilting generator $\tilting\cong \mathcal{O}\oplus \cL^{-1}\oplus \mathcal{H}\oplus (\mathcal{H}\otimes \cL)\oplus \mathcal{T}\oplus \mathcal{T}^{\perp}$.  
\end{theorem}
In fact, we don't need operation i., since $\tilting^*\cong \tilting \otimes \cL$.  
A natural question is how this construction compares with other known constructions of tilting bundles on $X$ and $T^*X$.  Kapranov has famously constructed a tilting generator on the Grassmannian $X$ \cite{kapranovDerivedCategories1988}; taking the dual of the version of this generator given in \cite[\S 2.1.2]{donovanWindowShifts2014}, one of Kapranov's generators on $\Gr(2,4)$ is given by \[\mathcal{O}\oplus \cL^{-1}\oplus \cL^{-2}\oplus \mathcal{T}\oplus \mathcal{T}\otimes \cL^{-1}\oplus \Sym^2(\mathcal{T}).\]
As noted in \cite[Rem. 3.6]{kawamataDerivedEquivalence2006}, the pullback of this vector bundle to $T^*X$ is no longer tilting: $\Ext^1(\mathcal{O},\mathcal{L}^{-2})=H^1(T^*X;\mathcal{L}^{-2})\neq 0.$  Since $\mathcal{H}$ is essentially constructed to extend $\mathcal{L}^{-2}$ so as to kill these higher extensions, there are some hints that our tilting generator might be a correction of Kapranov's, though there is another important difference---Kapranov's generator is always a sum of Schur functors applied to the tautological bundle and its dual, but $\mathcal{T}^{\perp}$ is not of this form.  

A recent preprint of Tseu \cite{tseuCanonicalTilting2024} gives a construction of a tilting bundle on $T^*\Gr(2,n)$ for all values of $n$.  We show in \Cref{tseu} that for $n=4$, up to tensor product with a line bundle, it coincides with ours and thus with Kaledin's.  

A more implicit construction of a tilting generator on this space is given by Toda--Uehara \cite{todaTiltingGenerators2010}; note that this construction is not canonical, so we cannot speak of {\it the} Toda-Uehara tilting generator, only the class of generators they construct, which are unique up to equiconstitution. That is, the indecomposable summands that appear are well defined, but different choices can lead to different multiplicities in the final tilting generator.  At the moment, due to the difficulty of making their construction explicit, we have not been able to prove or disprove the equivalence of our construction.  However, we can show that this tilting generators has a large summand in common (\Cref{prop:TU-comparison}), leading us to conjecture that:
\begin{conjecture}
    The tilting generator $\mathcal{Z}$ is equiconstituted with any Toda-Uehara tilting generator.
\end{conjecture}
An obvious question is whether these results can be applied to Grassmannians in higher ranks.  Our diagrammatic description of $T^*X$ and its tilting generator is still valid in the higher rank case, but the number of summands of $\tilting$ will explode in complexity. Thus, it is hard to extrapolate from the cases considered here which summands are likely to appear.  An {\it ad hoc} analysis as we have done here will rapidly cease to be practical, and some more systematic approach seems more likely to succeed.  

\section{Preliminaries}
\subsection{$T^*\Gr(n-k,n)$ and its vector bundles}\label{sec:TGr}
Let $X=\mathrm{Gr}(n-k,n)$ be the Grassmannian of codimension $k$ subspaces (or equivalently, dimension $k$ quotients) in $\C^n$.  
We can also express $X$ as the partial flag variety $G/P$ for $G=GL(n,\C)$ and
\begin{align*}
    &&P&=\left\lbrace \left(\begin{array}{cc}
        A & 0 \\
        C &  D
    \end{array}\right)\mid A\in GL(k,\C),D\in GL(n-k,\C)\; C\in M_{n-k,k}(\C)   \right\rbrace.&
\end{align*}
Consider the Pl\"ucker embedding: 
\begin{align*}
   && \iota: X&\hookrightarrow \mathbb{P}\left(\wedgenk\C^n\right )\cong \mathbb{P}^{\binom{n}{k}-1}&\\
   && V &\mapsto \wedgenk V.&
\end{align*}
We let $e_1,\dots, e_n$ be the usual coordinates of $\mathbb{C}^n$.  We will use the induced Pl\"ucker coordinates, which satisfy the Pl\"ucker relations.   The case of primary interest to us is $k=2$ and $n=4$, where Pl\"ucker coordinates correspond to ordered pairs $\lbrace \wcoor_{ij}:=e_i\wedge e_j\rbrace$ and there is a single Pl\"ucker relation:
\begin{equation}\label{plucker}
    \wcoor_{12}\wcoor_{34}-\wcoor_{13}\wcoor_{24}+\wcoor_{14}\wcoor_{23}=0.
\end{equation}
We'll let $\mathcal{L}$ denote the pullback of $\mathcal{O}(1)$ under this map; this is the unique ample generator of the Picard group of $X$ and its pullback to $T^*X$ is the unique ample generator of the Picard group there.  We have an isomorphism $\Gamma(X;\mathcal{L})\cong \wedgek\C^n$ as representations of $GL_n$ by Borel-Weil; more generally, $\Gamma(X;\mathcal{L}^{\otimes p})$ is an irreducible representation of $GL_n$ whose highest weight is $p\omega_k$, where $\omega_k$ is the fundamental highest weight of $\wedgek\C^n$.  

We'll also want to consider the trivial bundle $\mathcal{E}=\mathcal{O}_X^{\oplus n}$ of rank $n$.  This contains the rank $n-k$ tautological bundle $\mathcal{T}$ whose fiber over a subspace $U\subset \C^n$ is the subspace itself;  let $\mathcal{Q}=\mathcal{E}/\mathcal{T}$ be the tautological quotient bundle of rank $k$.  Note that we can think of $e_1,e_2,e_3,e_4$ as sections of $\mathcal{Q}$, and thus of $\wcoor_{ij}$ as sections of $\wedgek\mathcal{Q}\cong (\mathcal{E}\wedge\mathcal{E})/(\mathcal{E}\wedge\mathcal{T})\cong \mathcal{L}$; this map induces the isomorphism $\Gamma(X;\mathcal{L})\cong \wedgek \C^n$.  Dually, we have $\wedgenk\mathcal{T}\cong \mathcal{L}^{-1}$.

The bundle $\mathcal{E}=\mathcal{O}_X^{\oplus n}$ is self-dual via the usual dot product, so we have the dual filtration of $\mathcal{E}$ by the subbundle $\mathcal{T}^{\perp}\cong \mathcal{Q}^*$.  We have an isomorphism $\Gr(n-k,n)\cong \Gr(k,n)$ which swaps the role of the tautological bundles and their perpendiculars.  This map is not $GL_n$-equivariant; rather, it intertwines the action of $g\in GL_n$ with $(g^T)^{-1}$, its inverse-transpose.  Depending on which of these $GL_n$ actions we use, we can identify the sections of $\mathcal{L}$ with $\wedgek\C^n$ or $\wedgenk\C^n$.

For the rest of the section below, we specialize to $k=2,n=4$ for simplicity of notation, though most of the results below extend to the general case.  In this case, both $\mathcal{T}$ and $\mathcal{Q}$ have rank 2. Since any rank 2 vector bundle satisfies $V\cong V^*\otimes \wedgetwo V$, we have $\mathcal{T}^{\perp}\cong \mathcal{Q}\otimes \mathcal{L}^{-1}$.

For any very ample line bundle on a projective variety, we can obtain an open cover by taking the non-vanishing sets of a basis of the space of sections.  Applying this to $\mathcal{L}$, we find that  $X$ has an open cover $\mathcal{U}=\lbrace U_{ij}\rbrace$ where
\begin{align*}
    &&U_{ij}&=X\backslash\lbrace \wcoor_{ij}=0\rbrace.&
\end{align*}
We now consider the cotangent bundle $\pi:T^*X\to X$. This has an open cover $\mathcal{U}'=\lbrace \pi^{-1}(U_{ij})\rbrace$.

Given a vector bundle $\mathcal{Y}$ on $T^*X$, the restriction of $\mathcal{Y}$ to $U_{ij}'$ must be trivial.  Thus, if we choose trivializations $\tau_{ij}\colon \mathcal{O}^{\rk \mathcal{Y}}\to \mathcal{Y}$ on $U_{ij}'$, we obtain the usual cocycle description of this bundle in terms of the transition functions $\trans^{ij}_{rs}=\tau_{ij}^{-1}\circ \tau_{rs}$.  
We want to trivialize the bundles $\mathcal{L}, \mathcal{T},\mathcal{T}^{\perp}$ on each of these open sets and compute the transition functions.  

Of course, for a line bundle, this is simple: the section $\wcoor_{ij}$ defines the desired trivialization, and $\trans^{ij}_{rs}=\wcoor_{rs}/\wcoor_{ij}$.  More generally, for a power $\mathcal{L}^{m}$, we have transition functions $\trans^{ij}_{rs}=(\wcoor_{rs}/\wcoor_{ij})^m.$

For the tautological bundles, it's useful to think about the following approach: given $(x_1,\dots, x_4)\in \mathcal{E}_V$, we can test whether $(x_1,\dots, x_4)\in V$ by whether the wedge product of $(x_1,\dots, x_4)$ with a basis of $V$ is 0.  The coordinate of $e_i\wedge e_j\wedge e_r$ in this triple wedge product is $x_i\wcoor_{jr}+x_{j}\wcoor_{ri}+x_{r}\wcoor_{ij}=0.$
That is, $(x_1,\dots, x_4)\in V$ if it is perpendicular to the sections of $\mathcal{E}(1)$ defined by:
\begin{align*}
    &&A_1&=\left(\begin{array}{c}
            0 \\
         \wcoor_{34}  \\
         -\wcoor_{24} \\
         \wcoor_{23}
    \end{array}\right),& 
    A_2&=\left(\begin{array}{c}
            -\wcoor_{34} \\
          0  \\
         \wcoor_{14} \\
         -\wcoor_{13}
    \end{array}\right), & 
    A_3&=\left(\begin{array}{c}
            \wcoor_{24} \\
         -\wcoor_{14}  \\
         0 \\
         \wcoor_{12}
    \end{array}\right), & 
    A_4&=\left(\begin{array}{c}
            -\wcoor_{23} \\
         \wcoor_{13}  \\
          -\wcoor_{12} \\
         0
    \end{array}\right).&
\end{align*}
\Ben{We might want to choose the signs differently here to better match the signs we get from the diagrams later.  }
That is, the image of these sections span $\mathcal{T}^{\perp}$.  One can easily work out from the Pl\" ucker relations \eqref{plucker} that we have a similar set of maps $\mathcal{L}^{-1}\to \mathcal{T}$, defined by the vectors $B_{i}=(B_{1i},\dots, B_{4i})$ where:
\begin{align*}
    &&B_{ij}&=\begin{cases}
    \wcoor_{ij} & \text{if }i\neq j\\
    0 & \text{if }i=j.
    \end{cases}&
\end{align*}

On $U_{ij}'$, we have that $\wcoor_{ij}$ is invertible, and we can always trivialize $\mathcal{T}$ by using the vectors $B_i/\wcoor_{ij},B_{j}/\wcoor_{ij}$, and similarly for $\mathcal{T}^{\perp}$ using the vectors $A_k/\wcoor_{ij},A_{\ell}/\wcoor_{ij}$ where $\{r,s\}=\{i,j\}^c$.  
We can always compute the transition matrices using the fact that any 3 $A_i$'s or $B_i$'s are linearly dependent.  Note that while there are many pairs of $\{i,j\}$ in $\{1,2,3,4\}$, we only need the transition function for a single pair of open sets: the union $U_{ij}\cup U_{rs}$ has complement of codimension $\geq 2$, so a rational map between vector bundles defined on this open subset must extend to an isomorphism on $T^*X$.

To give an example that we will use later, let us consider $i=1,j=3,r=2,s=4$. The transition functions in this case are defined on $U_{13}'\cap U_{24}'$. On this intersection, the Pl\"ucker relation (\ref{plucker}) tells us:
\begin{align*}
    &&\wcoor_{24}A_2&=\left(\begin{array}{c}
         \wcoor_{24}\wcoor_{34} \\
          0 \\
          -\wcoor_{24}\wcoor_{14} \\
          \wcoor_{12}\wcoor_{34}+\wcoor_{14}\wcoor_{23}
    \end{array}\right)=\wcoor_{14}A_1+\wcoor_{34}A_3,&\\
    &&\wcoor_{24}A_4&=\left(\begin{array}{c}
         -\wcoor_{24}\wcoor_{23}  \\
         \wcoor_{14}\wcoor_{23}+\wcoor_{12}\wcoor_{34}\\
         -\wcoor_{24}\wcoor_{12}\\
         0
    \end{array}\right)=\wcoor_{12}A_1-\wcoor_{23}A_3,&\\
    &&\wcoor_{24}B_1&=\left(\begin{array}{c}
         0  \\
         \wcoor_{24}\wcoor_{12} \\
         \wcoor_{12}\wcoor_{34}+\wcoor_{14}\wcoor_{23} \\
         \wcoor_{24}\wcoor_{14}
    \end{array}\right)=\wcoor_{14}B_2+\wcoor_{12}B_4,&\\
    &&D_{24}B_3&=\left(\begin{array}{c}
         -\wcoor_{12}\wcoor_{34}-\wcoor_{14}\wcoor_{23}  \\
         -\wcoor_{24}\wcoor_{23} \\
         0\\
         \wcoor_{24}\wcoor_{34}
    \end{array}\right)=\wcoor_{34}B_2-\wcoor_{23}B_4.&
\end{align*}
Dividing by $\wcoor_{13}\wcoor_{24}$, we find that:
\begin{lemma}\label{transition-lemma-2}
    The transition functions $\trans^{13}_{24}$ are given by:
    \begin{enumerate}
        \item For $\mathcal{L}^{k}$, we have $\trans^{24}_{13}=\wcoor_{13}^k/\wcoor_{24}^k,\trans^{13}_{24}=\wcoor_{24}^k/\wcoor_{13}^k$.
        \item For $\mathcal{T}$, we have $\trans^{24}_{13}=\begin{bmatrix} \wcoor_{14}/\wcoor_{13}& \wcoor_{34}/\wcoor_{13}\\ \wcoor_{12}/\wcoor_{13}& -\wcoor_{23}/\wcoor_{13}\end{bmatrix},\trans^{13}_{24}=\begin{bmatrix} \wcoor_{23}/\wcoor_{24}& \wcoor_{12}/\wcoor_{24}\\ \wcoor_{34}/\wcoor_{24}& -\wcoor_{14}/\wcoor_{24}\end{bmatrix}$.
        \item For $\mathcal{T}^{\perp}$, we have $\trans^{24}_{13}=\begin{bmatrix} \wcoor_{14}/\wcoor_{13}& \wcoor_{12}/\wcoor_{13}\\ \wcoor_{34}/\wcoor_{13}& -\wcoor_{23}/\wcoor_{13}\end{bmatrix},\trans^{13}_{24}=\begin{bmatrix} \wcoor_{23}/\wcoor_{24}& \wcoor_{34}/\wcoor_{24}\\ \wcoor_{12}/\wcoor_{24}& -\wcoor_{14}/\wcoor_{24}\end{bmatrix}$.
    \end{enumerate}
\end{lemma}

In \cite[\S 7.2]{todaTiltingGenerators2010}, it's shown that 
\begin{equation}\label{eq:ext1}
     \Ext^1_{T^*X}(\mathcal{O},\mathcal{L}^{-2})\cong \bigoplus_{j=1}^{\infty} \Gamma(X;\operatorname{Sym}^{j-1}(\mathcal{T}))\cong \Gamma(T^*X;\mathcal{O})
\end{equation}  
where this decomposition gives the weight decomposition under the scaling $\C^*$. The summand for $j$ has weight $2j$; in particular, $j=1$ gives a unique line of nontrivial extensions of minimal weight which freely generate $ \Ext^1_{T^*X}(\mathcal{O},\mathcal{L}^{-2})$ as a module over $\Gamma(T^*X;\mathcal{O})$.  We'll use later that:
\begin{lemma}\label{indecomp-ext}
	There is a unique indecomposable extension $0\to \mathcal{L}^{-2} \to \exten \to \mathcal{O}\to 0$ which is equivariant for the squared scaling action, where $\mathcal{O}$ has the usual equivariant structure (where its lowest degree sections have weight 0) and $\mathcal{L}^{-2}$ is tensored with a space of weight 2 under the scaling action (where its lowest degree sections have weight 0).  That is, this extension has transition matrix
 \begin{equation}\label{eq:H-transition}
 \trans^{13}_{24}=\begin{bmatrix} \wcoor_{24}^2/\wcoor_{13}^2& f\\ 0& 1\end{bmatrix}
 \end{equation}
 where $f$ is a function of weight 2 under the squared scaling action. We'll also use the tensor product $\exten\otimes \mathcal{L}$, which has transition functions
  \begin{equation}\label{eq:HL-transition}
 \trans^{13}_{24}=\begin{bmatrix} \wcoor_{24}/\wcoor_{13}& g\\ 0& \wcoor_{13}/\wcoor_{24}\end{bmatrix}.
 \end{equation}
 The extension $\mathcal{H}$ is uniquely characterized by the property $H^1(T^*X;\mathcal{H})=0$.
 \end{lemma}
 \begin{proof}
     The uniqueness is determined by the fact that the $j=1$ component of $\Ext^1_{T^*X}(\mathcal{O},\mathcal{L}^{-2})$ is $\Gamma(X;\mathcal{O})$ which is 1-dimensional.  The description of transition functions follows by the standard form for an extension of vector bundles.

    For a given $\gamma\in  \Ext^1_{T^*X}(\mathcal{O},\mathcal{L}^{-2})$, let $\mathcal{H}_{\gamma}$ be the corresponding extension.  Then, by the usual long exact sequence, $H^1(T^*X;\mathcal{H}_{\gamma})$ is the cokernel of the boundary map $\Gamma(T^*X;\mathcal{O})\to H^1(T^*X;\mathcal{L}^{-2})$.  That is, it is trivial if and only if $\gamma$ generates $H^1(T^*X;\mathcal{L}^{-2})$ as a $\Gamma(T^*X;\mathcal{O})$-module.  Since $\Gamma(T^*X;\mathcal{O})$ has no units, this module has a unique line of such generators, which exactly gives us $\mathcal{H}_{\gamma}\cong \mathcal{H}$.
 \end{proof}
 We will determine the functions $f,g$ later, but by the uniqueness of this nontrivial extension, we do not need to determine them now.  This extension is also constructed in \cite[Lem. 3.2]{tseuCanonicalTilting2024}, since this lowest-weight line is $PGL(4)$-invariant.  This observation is needed to show that our tilting generator coincides with Tseu's.

Since $T^*\Gr(n-k,n)$ is a $GL_n$-equivariant vector bundle, we can write it as the associated bundle for a representation of $P$, which is most naturally written as the perpendicular to $\mathfrak{p}$ in the dual $ \mathfrak{g}^*$.  Using the trace pairing to identify $\mathfrak{g}^*\cong \mathfrak{g}$, this perpendicular is identified with the unipotent radical of $\mathfrak{p}$, the strictly upper triangular block matrices.  We thus have a Springer map $\sigma\colon T^*\Gr(n-k,n)\to \mathfrak{g}^*\cong \mathfrak{g}$, which we can also think of as the moment map for the usual $G$-action.  The image of this map is the closure $\bar{\mathcal{N}}(\lambda)$ of the orbit of nilpotent matrices whose Jordan type is $\lambda=(2^k,1^{n-2k})$, and this map is a symplectic resolution of singularities.  The cotangent bundle $T^*\Gr(n-k,n)$ is the only variety that appears as a symplectic resolution of $\bar{\mathcal{N}}(\mu)$, although there are two non-isomorphic symplectic resolution maps if $n>2$---the map $\sigma$ and its composition with the transposition map on $\mathfrak{g}$.

\subsection{Quiver gauge theories and cylindrical KLRW algebras}


Let $\Gamma$ be a finite 
quiver with vertices $\vertex(\Gamma)$, edges $\edge(\Gamma)$ and dimension vectors $\mathbf{v}$,$\mathbf{w}:\vertex\to\N$ with the $\mathbf{w}$ corresponding to the framing and denote the number of vertices by $n=|\vertex(\Gamma)|$. The quiver gauge theory associated to $\Gamma$ has gauge group
\begin{align*}
    &&G(\Gamma)&=\prod\GL(\C^{v_i})&
\end{align*}
and matter representation
\begin{align*}
    &&V(\Gamma)&=\left(\bigoplus_{i\to j}\Hom(\C^{v_i},\C^{v_j}) \right)\oplus\left(\bigoplus_{i\in\vertex}\Hom(\C^{w_i},\C^{v_i}) \right).&
\end{align*}
The theory of interest in this paper is determined by the data of the following framed quiver diagram:
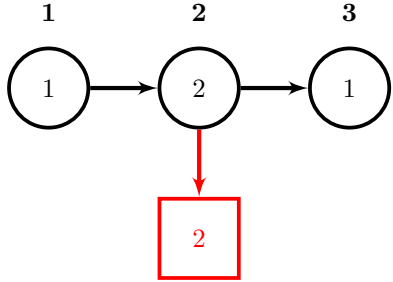
\tikzstyle{int}=[draw, line width = 0.5mm, minimum size=3em]
\begin{center}
\begin{tikzpicture}[node distance=3cm,auto,>=latex']
    \node[int, circle] (a) [] {$1$};

    \node [above of=a, node distance=1cm] {$\mathbf{1}$};

    \node [int,circle] (b) [right of=a, node distance=2cm] {$2$};

     \node [above of=b, node distance=1cm] {$\mathbf{2}$};

     \node [int,circle] (c) [right of=b, node distance=2cm] {$1$};

      \node [above of=c, node distance=1cm] {$\mathbf{3}$};

    \node [int,red] (d) [below of=b, node distance=2cm] {$2$};

    \draw[->,ultra thick,black] (a) edge node {} (b);

    \draw[->,ultra thick,black] (b) edge node {} (c);

    \draw[->,ultra thick,red] (b) edge node {} (d);
  
\end{tikzpicture}
\end{center}
It was shown in \cite[Th. E]{WebcohII} that the Coulomb branch of the corresponding $3d$ $\mathcal{N}=4$ gauge theory is given by a cylindrical KLRW algebra defined by the quiver data. These diagrammatic algebras are generated by string diagrams on the cylinder $\R/\Z\times[0,1]$ with product given by vertical composition. We can define these diagrams from the quiver data more precisely as follows.

Throughout the rest of this paper, a {\bf strand} will mean a  curve in $ \R/\Z\times [0,1]$ 
of the form
  $\{(\bar{\pi}(t),t)\mid t\in [0,1]\}$ for some path $\bar{\pi}\colon
  [0,1]\to \R/\Z$.  
  
\begin{definition} An \textbf{(unflavored) cylindrical KLRW diagram} associated to the framed quiver $\Gamma$ with dimension vectors $\mathbf{v}$ and $\mathbf{w}$ is a diagram on the cylinder $\R/\Z\times[0,1]$ consisting of
strands beginning at the bottom edge $\R/\Z\times 0$ and terminating at the top edge $\R/\Z\times 1$.
For each $i$, there will be an unordered $v_i$-tuple of strands, which we call {\bf black} and draw in that color, and an unordered $w_i$-tuple of strands, which we call {\bf red} and draw with that color. We can also add dots at arbitrary positions on black strands (but not on red) that avoid crossings and the lines $y=0$ and $y=1$.  In a cylindrical diagram, we require the red lines to be vertical (the corresponding $\bar{\pi}(t)$ is constant)\footnote{We'll consider diagrams where these red strands cross later.}.

These must satisfy the usual genericity property of avoiding tangencies and triple points between any set of strands, as well as strands meeting at $y=0$ or $y=1$.  We identify any diagrams that differ by isotopies that preserve these genericity conditions.  For simplicity of drawing, we'll also assume genericity with respect to the line $x=0$.  
\end{definition}

    Let us now describe our conventions for drawing cylindrical diagrams.  
We'll draw these on the page in the rectangle $[0,1]\times
[0,1]$ with seams on the left and right sides of the diagram where we
should glue to obtain the cylindrical diagram. 

Note that the order in which strands meet the circles $y=0$ and and $y=1$ induce two distinct cyclic orders on the set of strands, which we'll call the bottom and top orders.  By reading the labels of the strands, red or black, in a positive direction starting at $x=0$ (left to right in the convention discussed above), we obtain two words whose entries are $v_i$ instances of $i$ which are black and $w_i$ which are red, ranging over $i\in \vertex(\Gamma)$.  We'll call these words the {\bf bottom} and {\bf top} of our diagram.

Note that the top and bottom of a KLRW diagram will have the same order on red strands, so we will typically take this order as fixed, and only consider diagrams with this order.





\begin{definition}\label{def:cKLRW}  The {\bf cylindrical (unflavored) KLRW algebra} $ {\mathring{R}}$ attached to
  the data $\Gamma,\Bv,\Bw$  and an order on the red strands is the quotient of
  the formal span over a commutative ring $\Bbbk$ of cylindrical KLRW 
  diagrams for these data by the local relations below:
  \newseq
\begin{equation*}\subeqn\label{first-QH}
    \begin{tikzpicture}[scale=.7,baseline]
      \draw[very thick](-4,0) +(-1,-1) -- +(1,1) node[below,at start]
      {$i$}; \draw[very thick](-4,0) +(1,-1) -- +(-1,1) node[below,at
      start] {$j$}; \fill (-4.5,.5) circle (4pt);
      \node at (-2,0){=}; \draw[very thick](0,0) +(-1,-1) -- +(1,1)
      node[below,at start] {$i$}; \draw[very thick](0,0) +(1,-1) --
      +(-1,1) node[below,at start] {$j$}; \fill (.5,-.5) circle (4pt);
    \end{tikzpicture}\hspace{1.5cm}
    \begin{tikzpicture}[scale=.7,baseline]
      \draw[very thick](-4,0) +(-1,-1) -- +(1,1) node[below,at start]
      {$i$}; \draw[very thick](-4,0) +(1,-1) -- +(-1,1) node[below,at
      start] {$j$}; \fill (-3.5,.5) circle (4pt);
      \node at (-2,0){=}; \draw[very thick](0,0) +(-1,-1) -- +(1,1)
      node[below,at start] {$i$}; \draw[very thick](0,0) +(1,-1) --
      +(-1,1) node[below,at start] {$j$}; \fill (-.5,-.5) circle (4pt);
      \node at (3.8,0){unless $i=j$};
    \end{tikzpicture}
  \end{equation*}
\begin{equation*}\subeqn\label{nilHecke-1}
    \begin{tikzpicture}[scale=.7,baseline]
      \draw[very thick](-4,0) +(-1,-1) -- +(1,1) node[below,at start]
      {$i$}; \draw[very thick](-4,0) +(1,-1) -- +(-1,1) node[below,at
      start] {$i$}; \fill (-4.5,.5) circle (4pt);
      \node at (-2,0){$-$}; \draw[very thick](0,0) +(-1,-1) -- +(1,1)
      node[below,at start] {$i$}; \draw[very thick](0,0) +(1,-1) --
      +(-1,1) node[below,at start] {$i$}; \fill (.5,-.5) circle (4pt);
      \node at (2,0){$=$}; 
    \end{tikzpicture}
    \begin{tikzpicture}[scale=.7,baseline]
      \draw[very thick](-4,0) +(-1,-1) -- +(1,1) node[below,at start]
      {$i$}; \draw[very thick](-4,0) +(1,-1) -- +(-1,1) node[below,at
      start] {$i$}; \fill (-4.5,-.5) circle (4pt);
      \node at (-2,0){$-$}; \draw[very thick](0,0) +(-1,-1) -- +(1,1)
      node[below,at start] {$i$}; \draw[very thick](0,0) +(1,-1) --
      +(-1,1) node[below,at start] {$i$}; \fill (.5,.5) circle (4pt);
      \node at (2,0){$=$}; \draw[very thick](4,0) +(-1,-1) -- +(-1,1)
      node[below,at start] {$i$}; \draw[very thick](4,0) +(0,-1) --
      +(0,1) node[below,at start] {$i$};
    \end{tikzpicture}
  \end{equation*}
  \begin{equation*}\subeqn\label{black-bigon}
    \begin{tikzpicture}[very thick,scale=.8,baseline]

      \draw (-2.8,0) +(0,-1) .. controls (-1.2,0) ..  +(0,1)
      node[below,at start]{$i$}; \draw (-1.2,0) +(0,-1) .. controls
      (-2.8,0) ..  +(0,1) node[below,at start]{$j$}; 
   \end{tikzpicture}=\quad
   \begin{cases}
0 & i=j\\
   \begin{tikzpicture}[very thick,scale=.6,baseline=-3pt]
       \draw (2,0) +(0,-1) -- +(0,1) node[below,at start]{$j$};
       \draw (1,0) +(0,-1) -- +(0,1) node[below,at start]{$i$};\fill (2,0) circle (4pt);
     \end{tikzpicture}-\begin{tikzpicture}[very thick,scale=.6,baseline=-3pt]
       \draw (2,0) +(0,-1) -- +(0,1) node[below,at start]{$j$};
       \draw (1,0) +(0,-1) -- +(0,1) node[below,at start]{$i$};\fill (1,0) circle (4pt);
     \end{tikzpicture}& i\to j\\
  \begin{tikzpicture}[very thick,baseline=-3pt,scale=.6]
       \draw (2,0) +(0,-1) -- +(0,1) node[below,at start]{$j$};
       \draw (1,0) +(0,-1) -- +(0,1) node[below,at start]{$i$};\fill (1,0) circle (4pt);
     \end{tikzpicture}-\begin{tikzpicture}[very thick,scale=.6,baseline=-3pt]
       \draw (2,0) +(0,-1) -- +(0,1) node[below,at start]{$j$};
       \draw (1,0) +(0,-1) -- +(0,1) node[below,at start]{$i$};\fill (2,0) circle (4pt);
     \end{tikzpicture}& i\leftarrow j\\     \begin{tikzpicture}[very thick,scale=.6,baseline=-3pt]
       \draw (2,0) +(0,-1) -- +(0,1) node[below,at start]{$j$};
       \draw (1,0) +(0,-1) -- +(0,1) node[below,at start]{$i$};
     \end{tikzpicture} & \text{otherwise}\\
   \end{cases}
  \end{equation*}
 \begin{equation*}\subeqn\label{triple-dumb}
    \begin{tikzpicture}[very thick,scale=.8,baseline=-3pt]
      \draw (-2,0) +(1,-1) -- +(-1,1) node[below,at start]{$k$}; \draw
      (-2,0) +(-1,-1) -- +(1,1) node[below,at start]{$i$}; \draw
      (-2,0) +(0,-1) .. controls (-3,0) ..  +(0,1) node[below,at
      start]{$j$}; \node at (-.5,0) {$-$}; \draw (1,0) +(1,-1) -- +(-1,1)
      node[below,at start]{$k$}; \draw (1,0) +(-1,-1) -- +(1,1)
      node[below,at start]{$i$}; \draw (1,0) +(0,-1) .. controls
      (2,0) ..  +(0,1) node[below,at start]{$j$}; \end{tikzpicture}=\quad
      \begin{cases} 
    \begin{tikzpicture}[very thick,scale=.6,baseline=-3pt]
     \draw (6.2,0)
      +(1,-1) -- +(1,1) node[below,at start]{$k$}; \draw (6.2,0)
      +(-1,-1) -- +(-1,1) node[below,at start]{$i$}; \draw (6.2,0)
      +(0,-1) -- +(0,1) node[below,at
      start]{$j$};     \end{tikzpicture}& j\leftarrow i=k\\
    -\begin{tikzpicture}[very thick,scale=.6,baseline]
     \draw (6.2,0)
      +(1,-1) -- +(1,1) node[below,at start]{$k$}; \draw (6.2,0)
      +(-1,-1) -- +(-1,1) node[below,at start]{$i$}; \draw (6.2,0)
      +(0,-1) -- +(0,1) node[below,at
      start]{$j$};     \end{tikzpicture}& j\to i=k\\
      0 & \text{otherwise}
      \end{cases}
  \end{equation*}
      \begin{equation*}\label{cost}\subeqn
      \begin{tikzpicture}[very thick,scale=.8,baseline]
        \draw (-2.8,0) +(0,-1) .. controls (-1.2,0) ..  +(0,1)
        node[below,at start]{$i$}; \draw[wei] (-1.2,0) +(0,-1)
        .. controls (-2.8,0) .. node[below, at start]{$j$} +(0,1);
        \end{tikzpicture}=
\begin{cases}
     \begin{tikzpicture}[very thick,scale=.8,baseline]
      \draw[wei] (2.8,0) +(0,-1) -- node[below, at start]{$j$} +(0,1); \draw (1.2,0) +(0,-1) -- +(0,1)
        node[below,at start]{$i$}; \fill (1.2,0) circle (3pt);
        \end{tikzpicture}& i=j\\
         \begin{tikzpicture}[very thick,scale=.8,baseline]
      \draw[wei] (2.8,0) +(0,-1) -- node[below, at start]{$j$} +(0,1); \draw (1.2,0) +(0,-1) -- +(0,1)
        node[below,at start]{$i$};
        \end{tikzpicture}& i\neq j
\end{cases}     
\hspace{2cm}
      \begin{tikzpicture}[very thick,scale=.8,baseline]
        \draw[wei] (-2.8,0) +(0,-1) .. controls (-1.2,0) ..  node[below, at start]{$i$} +(0,1); \draw (-1.2,0) +(0,-1)
        .. controls (-2.8,0) ..  +(0,1) node[below,at start]{$j$};  \end{tikzpicture}=\begin{cases}
                  \begin{tikzpicture}[very thick,scale=.8,baseline]
         \draw(2.8,0) +(0,-1) -- +(0,1) node[below,at start]{$j$};
        \draw[wei] (1.2,0) +(0,-1) -- node[below, at start]{$i$} +(0,1); \fill (2.8,0) circle (3pt);
      \end{tikzpicture}& i=j\\\\
            \begin{tikzpicture}[very thick,scale=.8,baseline]
         \draw(2.8,0) +(0,-1) -- +(0,1) node[below,at start]{$j$};
        \draw[wei] (1.2,0) +(0,-1) -- node[below, at start]{$i$} +(0,1);
      \end{tikzpicture}& i\neq j
        \end{cases}
    \end{equation*}
    \begin{equation*}\label{red-triple}\subeqn
      \begin{tikzpicture}[very thick,baseline=-2pt,scale=.8]
        \draw [wei] (0,-1) -- node[below, at start]{$k$}(0,1);
        \draw(.5,-1) to[out=90,in=-30] node[below,at start]{$i$}
        (-.5,1); \draw(-.5,-1) to[out=30,in=-90] node[below,at
        start]{$j$} (.5,1);
      \end{tikzpicture}- \begin{tikzpicture}[very thick,baseline=-2pt,scale=.8]
        \draw [wei] (0,-1) -- node[below, at start]{$k$}(0,1);
        \draw(.5,-1) to[out=150,in=-90] node[below,at start]{$i$}
        (-.5,1); \draw(-.5,-1) to[out=90,in=-150] node[below,at
        start]{$j$} (.5,1);
      \end{tikzpicture}
      =\begin{cases}
          \begin{tikzpicture}[very thick,baseline=-2pt,scale=.8]
        \draw [wei]  (0,-1) -- node[below, at start]{$k$}(0,1);
        \draw(.5,-1) to[out=90,in=-90]  node[below,at start]{$i$}(.5,1);
        \draw(-.5,-1) to[out=90,in=-90] node[below,at start]{$j$} (-.5,1);
      \end{tikzpicture} & i=j=k\\
      0 & \text{otherwise}
      \end{cases} 
    \end{equation*}
  For all other triple points, we set the two sides of the isotopy
through it equal.
\end{definition}
Given a word in red and black copies of the symbols $\vertex(\Gamma)$, we can define a corresponding idempotent by having straight vertical strands on the cylinder, with no crossings or dots, where the strands we encounter are labeled by the vertices in the word, with the corresponding color.

Note that the relations above are homogeneous under the {\bf scaling grading} defined by:
\begin{equation}
    \label{eq:KLRW-grading} \deg    \begin{tikzpicture}[very thick,scale=.6,baseline=-3pt]
       \draw (2,0) +(0,-.7) -- +(0,.7) node[below,at start]{$j$};\fill (2,0) circle (4pt);
     \end{tikzpicture}=2\qquad\qquad \deg    \begin{tikzpicture}[very thick,scale=.6,baseline=-3pt]
       \draw (2,0) +(.7,-.7) -- +(-.7,.7) node[below,at start]{$j$}; \draw (2,0) +(-.7,-.7) -- +(.7,.7) node[below,at start]{$i$};
     \end{tikzpicture}=\begin{cases}
     	-2 & i=j\\
     	1 & i\leftrightarrow j\\
     	0 & i\not\leftrightarrow j
     \end{cases}
\end{equation}
\begin{equation}
    \label{eq:KLRW-grading2} 
   \deg    \begin{tikzpicture}[very thick,scale=.6,baseline=-3pt]
       \draw[wei] (2,0) +(.7,-.7) -- +(-.7,.7) node[below,at start]{$j$}; \draw (2,0) +(-.7,-.7) -- +(.7,.7) node[below,at start]{$i$};
     \end{tikzpicture}=     \deg    \begin{tikzpicture}[very thick,scale=.6,baseline=-3pt]
       \draw (2,0) +(.7,-.7) -- +(-.7,.7) node[below,at start]{$j$}; \draw [wei] (2,0) +(-.7,-.7) -- +(.7,.7) node[below,at start]{$i$};
     \end{tikzpicture}=\begin{cases}
     	1 & i=j\\
     	0 & i\neq j\\
     \end{cases}\qquad \qquad   \deg    \begin{tikzpicture}[very thick,scale=.6,baseline=-3pt]
       \draw[wei] (2,0) +(.7,-.7) -- +(-.7,.7) node[below,at start]{$j$}; \draw[wei] (2,0) +(-.7,-.7) -- +(.7,.7) node[below,at start]{$i$};
     \end{tikzpicture}= -k
\end{equation}
We also have a {\bf winding grading} by the group $\Z^{\Gamma}$ whose $i$-component counts with sign the number of times that all strands with label $i$ cross the vertical line $x=0$;  this is left unchanged by the relations \crefrange{first-QH}{red-triple} (though it can change which strand a crossing sits on) and isotopy can only cancel crossings with opposite sign.

When there are $k$ consecutive strands with the same label, then crossings between these strands and the dots on them form a copy of the nilHecke algebra of rank $k$.  It is well known that this algebra is isomorphic to a rank $k!$ matrix algebra over symmetric polynomials in the dots (see, for example, \cite[Prop. 3.5]{laudaCategorificationQuantum2010}) and so any object in an abelian (or more generally Karoubian) category with an action of this algebra is the direct sum of $k!$ isomorphic objects.  We can make this object easier to access by replacing the cylindrical KLRW algebra with the Morita equivalent algebra where we add this divided power formally.  

A graphical rule for doing this is already given in \cite{khovanovExtendedGraphical2012}. We now allow strands with the same label to meet at the top and bottom of diagrams. These should be generic in the sense that at these meeting points, all the strands which meet should have different slopes, so there is a clear ordering on the strands as they leave the meeting point.  When we compose diagrams with such meetings, the composition will be taken to be zero unless the labels and groupings of strands match---if $k$ strands meet at the bottom of one diagram, $k$ strands with the same label must meet at the top of the other.

We need a rule for interpreting a composition where we join the bottom of one diagram and the top of another where $k$ points with the same label meet.  This is done using the convention of \cite[(2.69)]{khovanovExtendedGraphical2012}: if we join $k$ strands and split them, then we can replace this join and split with the diagram $D_k$ where we perform a half-twist on the strands which joined.  The only case that we will use in this paper is that joining and then splitting two strands can be interpreted as a crossing of the strands.  
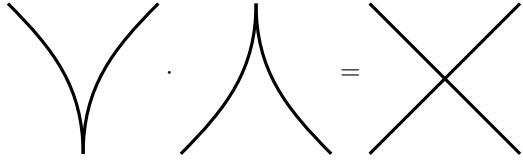
\[ \tikz[very thick,baseline]{ \draw (1,1) to[in=90,out=-135] (0,-1);
            \draw (-1,1) to[in=90,out=-45] (0,-1);
          }\cdot \tikz[very thick,baseline]{
            \draw (1,-1) to[in=-90,out=135] (0,1);
            \draw (-1,-1) to[in=-90,out=45] (0,1);
          }=\tikz[very thick,baseline]{
            \draw (-1,1)-- (1,-1);
            \draw (-1,-1)-- (1,1);
          }\]

We also add in idempotents with vertical strands, where we allow each strand to have a thickness $>1$;  describing these as a word, we write $i^{(k)}$ for a strand with label $i$ and thickness $k$.  
As in \cite[(2.63)]{khovanovExtendedGraphical2012}, we can put symmetric functions in an alphabet of $k$ dots on strands of thickness $k$; these can be constructed from dots on thickness 1 strands by \cite[(2.72)]{khovanovExtendedGraphical2012}.

Now, let $e$ be any idempotent where for every label $i$, all $v_i$ black strands with this label are joined.  For example, we can take $e(\Bi)$ for the word $\Bi=\textcolor{red}{1^{w_1}\cdots n^{w_n}}1^{(v_1)}2^{(v_2)}3^{(v_3)}\cdots (n-1)^{(v_{n-1})} $.  
We let $\bullet_i\in e\mathring{R}e$ denote the idempotent $e$ with the elementary symmetric function of degree 1 applied to the single thick strand with label $i$ (and no other dots added).

\begin{theorem}[\mbox{\cite[Th. 6.15]{WebcohII}}]
    The subalgebra $e\mathring{R}e$ is isomorphic to the Coulomb branch algebra for the quiver gauge theory defined by $\Gamma, \Bv,\Bw$.  In particular, this subring is abelian and $\fM=\operatorname{Spec}(e\mathring{R}e)$ is the BFN Coulomb branch of this theory.  
\end{theorem}

\rm
In this paper, we will only ever consider the vectors $\Bw$ of ``level two'' where there are two red strands in total.  Thus, we can distinguish these as the ``left'' and ``right'' red strands.  
\Ben{I've now removed any reference to ``first'' and ``second'' red strands.}
Let us specialize to this case from now on, to simplify our discussion.  In order to discuss line bundles on the Grassmannian, we will need to consider twisted KLRW diagrams as well. 
\begin{definition}
    A $k$-twisted (cylindrical) KLRW diagram for $k\in \Z$ is a diagram that satisfies all the same local rules as a KLRW diagram, except that instead of being vertical, the left red strand wraps around the cylinder $k$ times (that is, if $k<0$, it wraps $|k|$ times in the negative direction), while the right strand remains vertical.  We assume that the red strands never make a bigon, that is, there are $|k|$ intersection points of the red strands, all with the same sign, and that the top and bottom of the left red strand is the same point of $S^1$.

    The quotient of the formal span of the $k$-twisted diagrams by the relations (\ref{first-QH}--\ref{red-triple}) gives a bimodule $\mathring{T}^k$ over $\mathring{R}$, and the sum $\mathbf{R}=\oplus_{k\geq 0}\mathring{T}^k$ forms a graded ring with composition given by stacking.
\end{definition}

To avoid confusion with other gradings, we say that elements of $\mathring{T}^k $ have {\bf twist degree} $k$. Of course, since $\mathring{R}$ is the twist degree 0 subalgebra of $\mathbf{R}$, this means that we can think of $e$ as an element of $\mathbf{R}$. 

\begin{examplex} The following is an example of a diagram in $e\mathring{T}^1$:
\begin{align*}
    \begin{array}{c}  
 \tikz[very thick,xscale=2,yscale=2]{
            \draw (-1,-1)-- (1,-1);
            \draw (-1,1)-- (1,1);
          \draw[dashed] (-1,-1)-- (-1,1);
           \draw[red] (-.8,-1) .. controls (-.8,-1) and (-.8,-.25) ..node[below, at start]{$r_L$} (1,.5);
            \draw (-.48,-1) .. controls (-.48,-1) and (0,0) ..node[below, at start]{$2$} (0,1);
            \draw (-0.16,-1) .. controls (-.16,-1) and (-.16,-0.5) ..node[below, at start]{$1$} 
          (1,0);
          \draw[red] (.48 ,-1)--node[below, at start]{$r_R$} (.48,1);
           \draw (.64,-1) .. controls (.64,-1) and (.64,-.75) ..node[below, at start]{$2$} (1,-.5);
           \draw (.8,-1).. controls (.8,-1) and (.48,-1) .. node[below, at start]{$3$}node[above, at end]{$3$} (.16,1);
          \draw[red] (-1,0.5) .. controls (-1,0.5) and (-0.8,0.75) .. (-0.8,1);
          \draw (-1,0) .. controls (-1,0) and (-.16,0.5) ..node[above, at end]{$1$} (-.16,1);
          \draw (-1,-.5) .. controls (-1,-.5) and (0,0.25) ..node[above, at end]{$2$} (0,1);
          \draw[dashed] (1,1)-- (1,-1);
          }
   \end{array}&
\end{align*}
\end{examplex}

\begin{theorem}[\mbox{\cite[Th. 6.19]{WebcohII}}]
    The ring $\coul=e\mathbf{R}e$ is commutative and $\tilde{\fM}=\operatorname{Proj}(\coul)$ is a symplectic resolution of $\operatorname{Spec}(e\mathring{R}e)$.  
\end{theorem}

\subsection{Grassmannians from quiver gauge theories}
Fix integers $n$ and $0<k\leq n/2$, as before.  Let $\Gamma=A_{n-1}$, with $\vertex=\{1,\dots, n-1\}$ in the usual identification; since orientation will occasionally be important, we choose the arrows $i\to i+1$.  Consider the dimension vectors $\Bw=\mathbf{e}_k+\mathbf{e}_{n-k}$, and $\Bv=\sum_{i=1}^{n-1}\max(i,n-i-1,k)$.  For example, if $n=4,k=2$, then $\Bw=(0,2,0)$ and $\Bv=(1,2,1)$.  On the other hand, if $n=7,k=2$, then $\Bw=(0,1,0,0,1,0)$ and $\Bv=(1,2,2,2,2,1)$.  If we consider a $k\times (n-k)$ grid of boxes indexed by $\{1,\dots, k\}\times \{1,\dots, n-k\}$, then $v_m$ is the number of boxes $(i,j)$ that satisfy $m=i+j-1$.  
Let $e$ be the idempotent $\textcolor{red}{(n-k)}1^{(v_1)}2^{(v_2)}3^{(v_3)}\cdots (n-1)^{(v_{n-1})} \textcolor{red}{k}$.  In the case of $n=4$ and $k=2$, this idempotent is given by
$\textcolor{red}{2}12^{(2)}3\textcolor{red}{2}$.  

The choice of $\Bw,\Bv$ above might seem slightly random, but it is related to $T^*\Gr(n-k,n)$ by a result of Nakajima and Takayama, building on work of Maffei, Mirkovi\'c-Vilonen and earlier work of Nakajima. 
In the notation of \cite[\S 7.4]{nakajimaCherkisBow2017}, we see that $T^*\Gr(n-k,n)$ is the unique symplectic variety which is a resolution of $\overline{\mathcal{N}}(\lambda)=\overline{\mathcal{N}}(\lambda)\cap \mathcal{S}(\mu)$ where \[\lambda=(\underbrace{2,\cdots, 2}_{k \text{ times}},\underbrace{1,\dots, 1}_{n-2k\text{ times}}, 0,\dots,0)\qquad \mu=(1,\dots, 1).\]
Our choice of $\Bw$ and $\Bv$ arises from this choice of weights exactly as in \cite[(7.7)]{nakajimaCherkisBow2017}. 

Thus, we have that:
\begin{theorem}[\mbox{\cite[Th. 7.11]{nakajimaCherkisBow2017}}]\label{th:NT-iso}
    The Coulomb branch $\fM$ is isomorphic as a Poisson variety to $\overline{\mathcal{N}}(\lambda)$ and its symplectic resolution $\tilde{\fM}$ is isomorphic to $T^*\Gr(n-k,n)$ as a symplectic variety.  

    This isomorphism intertwines the weight 2 scaling $\C^*$-action on cotangent fibers with the scaling grading on KLRW algebras given in \cref{eq:KLRW-grading} and the action of diagonal matrices with the winding grading by $\Z^{n-1}$, identifying $(a_1,\dots, a_{n-1})$ with $\sum a_i\alpha_i$ in the root lattice of $\mathfrak{sl}_n$. 
\end{theorem}

This isomorphism depends on a long chain of so-called  ``Hanany-Witten moves'' and thus one does not immediately see explicit formulas for it.  However, this isomorphism is essentially uniquely fixed by compatibility with scaling and winding gradings.

In particular, the embedding $\overline{\mathcal{N}}(\lambda) \subset \mathfrak{sl}_n^*$ induces a surjective map $\Sym(\mathfrak{sl}_n)\to \C[\overline{\mathcal{N}}(\lambda)]$, which is an isomorphism in degree 2.  That is, the elements of $\C[\overline{\mathcal{N}}(\lambda)] $ of scaling degree 2 form a copy of $\mathfrak{sl}_n$ as a Lie algebra under Poisson bracket.  

Thus, we can identify each of the root spaces of $\mathfrak{sl}_n$ by their scaling degree, and the fact that the root space for $\pm(\alpha_i+\cdots +\alpha_{j-1})$ for $i<j$ must have winding grading $(0,\dots,0, 1,\dots, 1,0,\dots, 0)$.

For computing degrees, it is helpful to represent monopole operators with tableaux on $k\times (n-k)$ rectangles, as discussed above.  For each $m$, we have an unordered $v_m$-tuple of winding numbers that describe how many times the strands with label $m$ wrap around the cylinder.  Let $a_1^{(m)}\leq a_2^{(m)}\leq \cdots \leq a_{v_m}^{(m)} $ be these winding numbers in increasing order.  Then we can represent these as a tableau (with no assumption on any inequalities) where we fill the boxes $(i,j)$ satisfying $m=i+j-1$ (which we refer to as a {\bf diagonal}) with the integers $a_{j-s}^{(m)}$ where $s=\min(0, m-k)$ is the minimal column number of a box satisfying this equality $m=i+j-1$ minus 1; that is, we fill these boxes with the entries $a_{*}^{(m)}$ increasing from left to right.  The degree of the monopole operator for such a tableau is given by the sum of the contributions:
\begin{enumerate}
	\item $-2|a_{p}^{(m)}-a_{q}^{(m)}|$ for $p<q$.
	\item $|a_{p}^{(m)}-a_{q}^{(m\pm 1)}|$ for all $p,q$.
	\item $|a_{p}^{(k)}|+|a_{p}^{(n-k)}|$ for all $p$.  
\end{enumerate}
\begin{lemma}
	The only such tableau with degree 0 is the constant 0 tableau.  
\end{lemma}
\begin{proof}
	Let $r$ be the largest value appearing in the tableau.  Assume $r>0$.  There is at least one box $(i,j)$ with entry $r$ such that neither $(i,j+1)$ nor $(i-1,j)$ has entry $r$.  Let $a$ and $b$ be the number of other boxes in the same diagonal with entries $<r$ or $=r$ respectively, and $a^{\pm},b^{\pm}$ be the number of boxes in adjacent diagonals with entries $<r$ or $=r$, respectively.  In this case, changing this entry from $r$ to $r-1$ changes the contribution from the same diagonal by $2a-2b$ and from adjacent diagonals by $b^++b^--a^+-a^--\delta_{i+j-1,k}-\delta_{i+j-1,n-k}$.  This is 0 if every box below or right of $(i,j)$ in adjacent diagonals has label $\geq r$, and negative otherwise.  Thus, there is a chain of moves which weakly decrease degree that we can use to turn every positive label into a 0.  Note that if there is only one entry with label $r$, then decreasing it decreases the degree, so this chain will have strictly decreased degree if there were any positive entries.
	
	A symmetric argument allows us to get rid of all negative entries and arrive at the tableau with all zeros.  This has degree 0, so any tableau with non-zero entries has positive degree. 
\end{proof}

If instead we consider diagrams with twisting degree $\ell$, then the rules for calculating degrees are almost the same, except that the contribution from the red strands is $|a_{p}^{(k)}|+|a_{p}^{(n-k)}-\ell|$ for all $p$ minus a global factor of $k\ell$.  The same argument above shows that: 
\begin{lemma}\label{lem:twisting-minimal}
	The diagrams of scaling degree 0 and twisting degree $\ell$ correspond to tableaux with entries in $[0,\ell]$ which satisfy $a(i-1,j)\leq a(i,j)\leq a(i,j-1)$.  
\end{lemma}
Note that such tableau are in bijection with basis vectors of the representation with highest weight $\ell\omega_k$; for example, if we extend our diagram by adding all boxes $(i,j)$ satisfying $i+j\leq n+1$ and fill all entries with $j>k$ with $0$ and with $i>n-k$ with $\ell$, then reading along the diagonals gives a Gelfand-Tsetlin tableau for the corresponding representation.  The usual Gelfand-Tsetlin formulas thus even provide explicit matrices for this representation.

\subsubsection{Quantization}
To better understand this isomorphism, it is helpful to note that there is a quantum version of it, a special case of the ``quantum Mirkovi\'c-Vybornov isomorphism'' of \cite{WWY}.  In particular, this is the cleanest way of seeing that we have an isomorphism not just of varieties, but of Poisson varieties.

We first consider a noncommutative deformation $A_{\hbar}^{b_L,b_R}$ over $\C[\hbar]$ of $\C[\fM]$, where one deforms the relations so that sliding a dot past the ``seam'' $x=0$ on the cylinder shifts the value of the dot by $\pm \hbar$, and the relations involving red strands are shifted by scalars $b_L,b_R$ corresponding to the left and right strands:
     \begin{equation}\label{a-dot-slide}
    \begin{tikzpicture}[very thick,baseline,scale=.7]
  \draw(-3,0) +(-1,-1) -- +(1,1);
  \draw[fringe](-3,0) +(0,-1) --  +(0,1);
\fill (-3.5,-.5) circle (3pt); \end{tikzpicture}
=
 \begin{tikzpicture}[very thick,baseline,scale=.7] \draw(1,0) +(-1,-1) -- +(1,1);
  \draw[fringe](1,0) +(0,-1) --  +(0,1);
\fill (1.5,.5) circle (3pt);
    \end{tikzpicture} +h  \begin{tikzpicture}[very thick,baseline,scale=.7] \draw(1,0) +(-1,-1) -- +(1,1);
  \draw[fringe](1,0) +(0,-1) --  +(0,1);
    \end{tikzpicture}
\qquad \qquad     \begin{tikzpicture}[very thick,baseline,scale=.7]
  \draw(-3,0) +(1,-1) -- +(-1,1);
  \draw[fringe](-3,0) +(0,-1) --  +(0,1);
\fill (-2.5,-.5) circle (3pt); \end{tikzpicture}
=
 \begin{tikzpicture}[very thick,baseline,scale=.7] \draw(1,0) +(1,-1) -- +(-1,1);
  \draw[fringe](1,0) +(0,-1) --  +(0,1);
\fill (.5,.5) circle (3pt);
    \end{tikzpicture} -h  \begin{tikzpicture}[very thick,baseline,scale=.7] \draw(1,0) +(1,-1) -- +(-1,1);
  \draw[fringe](1,0) +(0,-1) --  +(0,1);
    \end{tikzpicture}
  \end{equation}
        \begin{equation*}\subeqn\label{x-cost-1}
   \begin{tikzpicture}[very thick,baseline,scale=.8]
     \draw (-2.8,0)  +(0,-1) .. controls (-1.2,0) ..  +(0,1) node[below,at start]{$i$};
        \draw[red] (-2,0)  +(0,-1)--node[below,at start ]{$r_L$}  +(0,1);
   \end{tikzpicture}
 =\begin{cases}
  \begin{tikzpicture}[very thick,baseline,scale=.8]
  \draw[red] (2.3,0)  +(0,-.5) -- node[below,at start ]{$r_L$} +(0,.5);
        \draw (1.5,0)  +(0,-.5) -- +(0,.5) node[below,at start]{$i$};
        \fill (1.5,0) circle (3pt);
 \end{tikzpicture} +hb_{L}  \begin{tikzpicture}[very thick,baseline,scale=.8]
  \draw[red] (2.3,0)  +(0,-.5) -- node[below,at start ]{$r_L$} +(0,.5);
        \draw (1.5,0)  +(0,-.5) -- +(0,.5) node[below,at start]{$i$};
 \end{tikzpicture} & i=n-k\\
 \begin{tikzpicture}[very thick,baseline,scale=.8]
     \draw (-2.8,0)  +(0,-.5) --  +(0,.5) node[below,at start]{$i$};
        \draw[red] (-2,0)  +(0,-.5)--node[below,at start ]{$r_L$}  +(0,.5);
   \end{tikzpicture}          & i\neq n-k
 \end{cases} \qquad   \begin{tikzpicture}[very thick,baseline,scale=.8]
     \draw (-2.8,0)  +(0,-1) .. controls (-1.2,0) ..  +(0,1) node[below,at start]{$i$};
        \draw[red] (-2,0)  +(0,-1)--node[below,at start ]{$r_R$}  +(0,1);
   \end{tikzpicture}
 =\begin{cases}
  \begin{tikzpicture}[very thick,baseline,scale=.8]
  \draw[red] (2.3,0)  +(0,-.5) -- node[below,at start ]{$r_R$} +(0,.5);
        \draw (1.5,0)  +(0,-.5) -- +(0,.5) node[below,at start]{$i$};
        \fill (1.5,0) circle (3pt);
 \end{tikzpicture} +hb_{R}  \begin{tikzpicture}[very thick,baseline,scale=.8]
  \draw[red] (2.3,0)  +(0,-.5) -- node[below,at start ]{$r_R$} +(0,.5);
        \draw (1.5,0)  +(0,-.5) -- +(0,.5) node[below,at start]{$i$};
 \end{tikzpicture} & i=k\\
 \begin{tikzpicture}[very thick,baseline,scale=.8]
     \draw (-2.8,0)  +(0,-.5) --  +(0,.5) node[below,at start]{$i$};
        \draw[red] (-2,0)  +(0,-.5)--node[below,at start ]{$r_R$}  +(0,.5);
   \end{tikzpicture}          & i\neq k
 \end{cases}
 \end{equation*}
Note that this is still a graded algebra with $\deg \hbar=2$
As usual, we can define the Poisson bracket on $\C[\fM]=A_{\hbar}/\hbar A_{\hbar}$ by the formula $\{a,b\}=\frac 1{\hbar}[a,b]\pmod \hbar $.  In this non-commutative algebra, the degree 2 elements form a Lie algebra under the bracket $\frac 1{\hbar}[a,b]$ which is isomorphic to $\mathfrak{pgl}_n\oplus \C$.  

Note that $A_{\hbar}^{b_L,b_R}$ only depends as an algebra on the difference $b_L-b_R$, since we can shift these parameters simultaneously by sending $\dotstrand\mapsto \dotstrand +a$ for any $a\in \C$. 

By \cite[Lem. 8.8]{WebcohII}, the algebra $A_{\hbar}^{b_L,b_R}$ is the quantum Coulomb branch as constructed by Braverman--Finkelberg--Nakajima \cite{BFN}.  The parameter $b_L-b_R$ is a ``flavor parameter'' in this context.  
It will be easier to compare this with other papers where these flavor parameters appear by comparing with the ``KLR Yangian'' algebra $\BK$ defined in \cite{KTWWYO}.  This algebra is very similar to the deformed KLRW algebra discussed above, but with some key differences.
Let 
\begin{equation}
    p_j(u)=\begin{cases}
    1 & j\notin \{k,n-k\}\\
    u+k-2(b_R+1) & j=k\neq n-k\\
    u+(n-k)-2b_L  & j=n-k\neq k\\
    (u+k-2(b_R+1))(u+k -2b_L) & j=k=n-k.\\
\end{cases}
\end{equation} 
\begin{lemma}
    We can define a homomorphism of the KLR Yangian $\BK$ with the polynomials $p_j$ as above into $A_{\hbar}^{b_L,b_R}$ specialized at $\hbar=-2$ by the rule:
    \begin{enumerate}
        \item a dotless diagram is sent to the same diagram with the red strand $r_L$ placed as $x=\epsilon$, and $r_R$ at $x=-\epsilon$ for $\epsilon>0$ very small.  
        \item a dot $\dotstrand$ on a strand with the label $i$ is mapped to $\dotstrand+i$.
    \end{enumerate}
\end{lemma}
\begin{proof}
    The KLRW relations \cite[(3.1-3.3)]{KTWWYO} are almost identical to \crefrange{first-QH}{triple-dumb};  the only difference is that we use the polynomial $\bar{Q}_{ij}(u,v)$ which means that in $\BK$ we have the relation
 \begin{equation*}\subeqn\label{black-bigonYa}
    \begin{tikzpicture}[very thick,scale=.8,baseline]

      \draw (-2.8,0) +(0,-1) .. controls (-1.2,0) ..  +(0,1)
      node[below,at start]{$i$}; \draw (-1.2,0) +(0,-1) .. controls
      (-2.8,0) ..  +(0,1) node[below,at start]{$j$}; 
   \end{tikzpicture}=\quad
   \begin{cases}
   \begin{tikzpicture}[very thick,scale=.6,baseline=-3pt]
       \draw (2,0) +(0,-1) -- +(0,1) node[below,at start]{$j$};
       \draw (1,0) +(0,-1) -- +(0,1) node[below,at start]{$i$};\fill (2,0) circle (4pt);
     \end{tikzpicture}-\begin{tikzpicture}[very thick,scale=.6,baseline=-3pt]
       \draw (2,0) +(0,-1) -- +(0,1) node[below,at start]{$j$};
       \draw (1,0) +(0,-1) -- +(0,1) node[below,at start]{$i$};\fill (1,0) circle (4pt);
     \end{tikzpicture}- \begin{tikzpicture}[very thick,scale=.6,baseline=-3pt]
       \draw (2,0) +(0,-1) -- +(0,1) node[below,at start]{$j$};
       \draw (1,0) +(0,-1) -- +(0,1) node[below,at start]{$i$};
     \end{tikzpicture}& i\to j\\
  \begin{tikzpicture}[very thick,baseline=-3pt,scale=.6]
       \draw (2,0) +(0,-1) -- +(0,1) node[below,at start]{$j$};
       \draw (1,0) +(0,-1) -- +(0,1) node[below,at start]{$i$};\fill (1,0) circle (4pt);
     \end{tikzpicture}-\begin{tikzpicture}[very thick,scale=.6,baseline=-3pt]
       \draw (2,0) +(0,-1) -- +(0,1) node[below,at start]{$j$};
       \draw (1,0) +(0,-1) -- +(0,1) node[below,at start]{$i$};\fill (2,0) circle (4pt);
     \end{tikzpicture}- \begin{tikzpicture}[very thick,scale=.6,baseline=-3pt]
       \draw (2,0) +(0,-1) -- +(0,1) node[below,at start]{$j$};
       \draw (1,0) +(0,-1) -- +(0,1) node[below,at start]{$i$};
     \end{tikzpicture}& i\leftarrow j\\     
   \end{cases}
  \end{equation*}  Since our quiver is oriented with $i\mapsto i+1$, the image of the RHS of this equation will agree with  \cref{black-bigon}; all other KLRW relations are the same.

  The relation \cite[(4.7a)]{KTWWYO} agrees with \cref{a-dot-slide} after the specialization $\hbar=-2$.  In order for the relations \cite[(4.7b--d)]{KTWWYO} to follow from \eqref{red-triple} and (\ref{a-dot-slide}--\ref{x-cost-1}) we must have \begin{equation}
 \label{x-cost-2}
  \begin{tikzpicture}[very thick,baseline,scale=.7]
       \draw[wei] (-2.3,0)  +(0,-1)--  +(0,1);
       \draw[dashed] (-2,0)  +(0,-1)--  +(0,1);
       \draw[wei] (-1.7,0)  +(0,-1)--  +(0,1);
  \draw (-1,0)  +(0,-1) .. controls (-3.5,0) ..  +(0,1) node[below,at start]{$i$};\end{tikzpicture}
           =p_i\Bigg(\hspace{5mm}
  \begin{tikzpicture}[very thick,baseline,scale=.7]
    \draw (2.8,0)  +(0,-1) -- +(0,1) node[below,at start]{$i$};
          \draw[wei] (2.3,0)  +(0,-1)--  +(0,1);
       \draw[dashed] (2,0)  +(0,-1)--  +(0,1);
       \draw[wei] (1.7,0)  +(0,-1)--  +(0,1);
       \fill (2.8,0) circle (3pt);\end{tikzpicture}+i\Bigg)
       \qquad \qquad
  \end{equation}   
\end{proof}
The appearance of $\hbar=-2$ might seem strange, but this is an artifact of notational choices in the literature on Yangians;  all specializations of $\hbar$ to non-zero scalars are isomorphic, via their natural isomorphism to the degree 0 subspace of $A_{\hbar}^{b_L,b_R}[\hbar^{-1}]$.

Thus, in the notation of \cite{KTWWYO,WWY}, we should take $\mathbf{R}_k=\{-k+2(b_R+1)\}$ and $\mathbf{R}_{n-k}=\{(k-n+2b_L \} $ if $n>2k$ and $\mathbf{R}_k=\{-k+2b_L,-k+2(b_R+1)\}$ if $n=2k$.  By \cite[Th. 2.13]{WWY}, this means that $A_{\hbar}^{b_L,b_R}$ is a quotient of $U(\mathfrak{sl}_n)$ by the maximal ideal in its center that corresponds to the set 
\begin{multline}
\tilde{\mathbf R}=\{-n+3+2b_R, -n+5+2b_R,\dots, n-2k+1+2b_R,\\ -n+1 +2b_L,-n+3+2b_L,\dots, -n-1+2k+2b_L\}.    
\end{multline}
Explicitly, this is the maximal ideal of the center that vanishes on any Verma module for a weight $\lambda$ of $\mathfrak{gl}_n$ such that the entries of $\lambda+\rho$ in the coordinates $\epsilon_i(\lambda+\rho)$ as an unordered multiset agree with $ \tilde{\mathbf R}$; 
this maximal ideal is unchanged under the dot action of $S_n$ and thus only depends on these coordinates as an unordered set.  Furthermore, we are only interested in the intersection of this ideal with $U(\mathfrak{sl}_n)$, which is unchanged by simultaneous translation of the set $ \tilde{\mathbf R}$.  Examples of weights of this form which will be dominant for some values of $b_L-b_R$ are $(-n+k-1+b_L-b_R)\omega_k$ and $(-k+1+b_R-b_L)\omega_{n-k}$.

By \cite[Th. 4.3]{WWY}, the algebra $A_{\hbar}$ is also the algebra of sections of a sheaf $\mathscr{D}_{b_L,b_R}$ of twisted differential operators on $X$, where the twist depends on $b_L-b_R$ as determined in \cite[\S 3.3.2]{WWY}.   It is easy to clarify this dependence by noting that the parabolic Verma module with highest weight $\ell \omega_k$ is the space of sections of $\mathcal{L}^{\ell}$ on the open $P$ orbit in $X$ (for $\ell\in \Z$), and thus carries an action of twisted differential operators for this twist.  On the other hand, if we twist the $GL_n$-action on $X$ by the automorphism of taking inverse-transpose, then this space of sections will be the parabolic Verma module with highest weight $\ell \omega_{n-k}$ This shows that:
\begin{lemma}\label{lem:twist-fix}
    The algebra $A_{\hbar}^{b_L,b_R}$ is isomorphic to differential operators twisted by $\mathcal{L}^{\ell}$  when $b_L-b_R=\ell+n-k+1$ or, using the inverse-transpose automorphism, when $b_L-b_R=-\ell -k+1$.  
\end{lemma}
This quantization will make it easier to see the connection to line bundles on the resolution.  

\subsubsection{The isomorphism}
Armed with the Poisson structure induced by this quantization, we can relatively easily give a more explicit description of the isomorphism between $e\mathring{R}e$ and $\C[T^*X]$.  As discussed above, the algebra $\C[T^*X]$ is generated by the matrix entries $e_{i,j}$, which span a copy of $\mathfrak{pgl}_n$ as a Lie algebra under Poisson bracket; these are precisely the elements of scaling degree 2.  Any isomorphism of Lie algebras to the scaling degree 2 elements of $e\mathring{R}e$ (also considered under Poisson bracket) will induce an isomorphism $\C[T^*X]\to e\mathring{R}e$.

Since the space of elements of the correct winding and scaling degree is 1-dimensional, we must have:
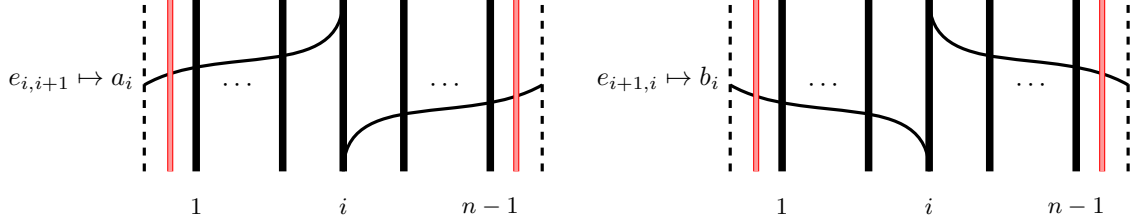
\begin{equation}\label{homomorphism}
e_{i,i+1}\mapsto a_i\,\, 
 \begin{tikzpicture}[baseline,scale=2.3]
      \draw[very thick, dashed] (4,0) +(0,-.5) -- +(0,.5);
 \draw[line width =1mm] (4.3,0) +(0,-.5) -- +(0,.5);
  \draw[line width =1mm] (6,0) +(0,-.5) -- +(0,.5);
  \draw[very thick, dashed] (6.3,0) +(0,-.5) -- +(0,.5);
  \draw (4.3,0) +(0,-.7) node {\small$1$};
  \draw (6,0) +(0,-.7) node {\small$n-1$};
  \draw[line width =1mm] (4.8,0) +(0,-.5) -- +(0,.5);
  \draw[line width =1mm] (5.5,0) +(0,-.5) -- +(0,.5);
  \draw[line width =1mm] (5.15,0) +(0,-.5) -- +(0,.5);
\draw (5.15,0) +(0,-.7) node {\small$i$};
\draw (4.55,0) node {$\cdots$};
\draw (5.75,0) node {$\cdots$};
\draw[very thick] (5.15,-.5) to[out=90,in=-150] (6.3,0);
\draw[very thick] (5.15,.5) to[out=-90,in=30] (4,0);
        \draw[wei] (6.15,0) +(0,-.5) -- +(0,.5);
      \draw[wei] (4.15,0) +(0,-.5) -- +(0,.5);
\end{tikzpicture}\qquad e_{i+1,i}\mapsto b_i\,\,
 \begin{tikzpicture}[baseline,scale=2.3]
      \draw[very thick, dashed] (4,0) +(0,-.5) -- +(0,.5);
 \draw[line width =1mm] (4.3,0) +(0,-.5) -- +(0,.5);
  \draw[line width =1mm] (6,0) +(0,-.5) -- +(0,.5);
  \draw[very thick, dashed] (6.3,0) +(0,-.5) -- +(0,.5);
  \draw (4.3,0) +(0,-.7) node {\small$1$};
  \draw (6,0) +(0,-.7) node {\small$n-1$};
  \draw[line width =1mm] (4.8,0) +(0,-.5) -- +(0,.5);
  \draw[line width =1mm] (5.5,0) +(0,-.5) -- +(0,.5);
  \draw[line width =1mm] (5.15,0) +(0,-.5) -- +(0,.5);
\draw (5.15,0) +(0,-.7) node {\small$i$};
\draw (4.55,0) node {$\cdots$};
\draw (5.75,0) node {$\cdots$};
\draw[very thick] (5.15,.5) to[out=-90,in=150] (6.3,0);
\draw[very thick] (5.15,-.5) to[out=90,in=-30] (4,0);
        \draw[wei] (6.15,0) +(0,-.5) -- +(0,.5);
      \draw[wei] (4.15,0) +(0,-.5) -- +(0,.5);
\end{tikzpicture}
\end{equation}
These elements automatically satisfy the Serre relations, since the space of diagrams of the correct winding degree has no elements of scaling degree 2.  Thus, in order to show that we have a homomorphism of Lie algebras, we only need to choose the factors $a_i,b_i$ so that these elements satisfy the relations of $\mathfrak{sl}_2$, which we know on abstract grounds to be possible.  In particular, we can always fix $b_i=1$ without loss of generality.

In the case of primary interest to us, $k=2,n=4$, calculations too long to include here confirm that:  
\begin{theorem}  When $k=2,n=4$ and $a_1=a_2=a_3=-1$, 
    there is a unique isomorphism of Poisson algebras $\C[T^*X]\to e\mathring{R}e$ satisfying \eqref{homomorphism}.  
\end{theorem}
This seems to suggest that we can take $a_i=-1$ in all cases, but the confirmation of this sign is a very delicate computation, so we will leave more general confirmation of this pattern to another time.  

\subsubsection{Resolutions}
Since thinking about resolutions is also important, we wish to more explicitly identify the resolution $\tilde{\fM}$.  We know that it must be $T^*\Gr(n-k,n)$, since that is the only symplectic resolution of $\overline{\mathcal{N}}(\lambda)$.  However, it is not immediately obvious, for example, which ample line bundle we recover.  Let $\mathcal{O}(1)$ denote the canonical ample line bundle on $\tilde{\fM}$ induced by its construction as a $\Proj$.

We can do this by defining a deformed version $e \mathring{T}^1_{\hbar} e$ of the module $e \mathring{T}^1 e$, which is a bimodule over $A_{\hbar}^{b_L+1,b_R}\operatorname{-}A_{\hbar}^{b_L,b_R}$.  We can define this by applying the local relations (\ref{first-QH}--\ref{red-triple},\ref{a-dot-slide}), with some extra care to the red/black bigon relations:  instead of \eqref{cost}, we apply \eqref{x-cost-1} with the values $b_L$ and $b_R$ below the $y$-value where the red strand $r_L$ crosses the seam at $x=0$, and $b_L+1$ and $b_R$ above this point.  We can tensor together $m$ these bimodules for suitable choices of parameters to obtain a quantization $e \mathring{T}^m_{\hbar} e$ of the module $e \mathring{T}^m e$.

By \cite[Prop. 8.12]{WebcohII}, this is the natural deformation of $\Gamma(T^*X;\mathcal{O}(m))$  to a bimodule $\Gamma(X;\mathcal{O}(m)\otimes_{\mathcal{O}}\mathscr{D}_{b_L,b_R})$ over the corresponding rings of twisted differential operators.  Using \Cref{lem:twist-fix}, we can see that if $\mathscr{D}_{b_L,b_R}$ is the ring of differential operators twisted by $\mathcal{L}^{\ell}$, then $\mathscr{D}_{b_L+m,b_R}$ is twisted by $\mathcal{L}^{\ell+m}$.  This is only possible if:
\begin{lemma}
	There is an isomorphism $\tilde{\fM}\cong T^*X$ which induces an isomorphism $\mathcal{O}(1)\cong \mathcal{L}$.  
\end{lemma}  
Thus, we have obtained the usual Pl\"ucker embedding of $T^*X$.  We can recover the $\mathfrak{sl}_n$-equivariant structure on $\mathcal{L}$ by generalizing the adjoint action of $\mathfrak{sl}_n$ on $U(\mathfrak{sl}_n)$.  
This defines a $\mathfrak{sl}_n$ action on $e \mathring{T}^1_{\hbar} e$ by the action $X\cdot t=\frac{1}{\hbar}(Xt-tX)$.  
This action preserves the scaling degree and thus preserves the subspace $V$ of elements of scaling degree 0 in $e \mathring{T}^1_{\hbar} e$.  
\begin{lemma}
	As a $\mathfrak{sl}_n$-module, we have an isomorphism $V\cong \bigwedge{}^{k}\C^n\cong \Gamma(X;\mathcal{L})$. 
\end{lemma}
\begin{proof}
	The diagram where all strands have winding number 1 around the cylinder (in the case of $k=2, n=4$, this is $\wcoor_{12}$) is a highest weight vector since there is no element of $V$ where a strand with winding number $>1$ appears.  Let $\lambda$ be its weight.  There is a non-zero element of the weight $\lambda-\alpha_i$ only when $i=k$, given by the diagram where a single strand with label $k$ has winding number 0 and all the others are 1 (in the case of $k=2, n=4$, this is $\wcoor_{13}$).  Furthermore, there is also no non-zero element of weight $\lambda-2\alpha_i$ (no diagram in $V$ where a strand with label $k$ has winding number $-1$).  This shows that the weight $\lambda=\omega_k$ and we have an equivariant map $\wedgek \C^n\to V$ that must be injective by the simplicity of $\wedgek \C^n$.  On the other hand, by \Cref{lem:twisting-minimal}, we have that $\dim V=\binom{n}{k}$, so this map is an isomorphism.  
\end{proof}

\begin{lemma}
    Under any $\C^*\times SL_n$ equivariant isomorphism $\tilde \fM$, the subspace $V$ maps isomorphically to $\Gamma(X;\mathcal{L})$.  In particular, the cover of $T^*X$ and $X$ obtained by inverting non-zero weight vectors in $V$ is the usual affine cover of discussed in Section \ref{sec:TGr}.
\end{lemma}
\begin{proof}
    We have a map $\Sym^{r}(V)\to \Gamma(T^*X;\mathcal{L}^r)$, which for all $r$ is surjective onto the elements of scaling degree 0.  In particular, the $r$th power of an element of $V$ never vanishes.  This shows that the elements of $V$ define non-zero sections of $\Gamma(T^*X;\mathcal{L})$.  Since they have scaling degree 0, they are constant on the cotangent fibers, and thus are non-zero after restriction to $X$.  The induced map to $\Gamma(X;\mathcal{L})$ is $SL_n$ equivariant and not zero, so it must be injective.  We already know that these vector spaces have the same dimension, so the induced map must be an isomorphism.  
\end{proof}
In particular, in the case $k=2$ and $n=4$,  we can explicitly find diagrams $\coord_{ij}$ which must agree up to scalar with the wedges $e_{ij}=e_i\wedge e_j$, since they are the only elements in $V$ with the correct winding grading.  These diagrams are:
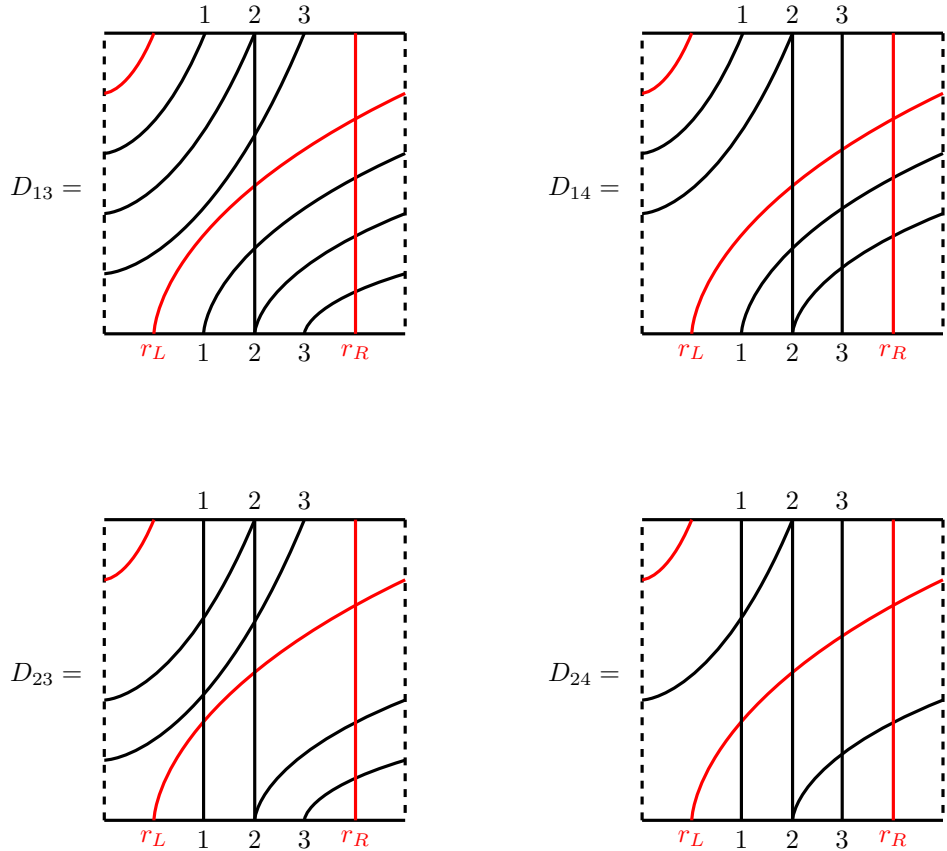
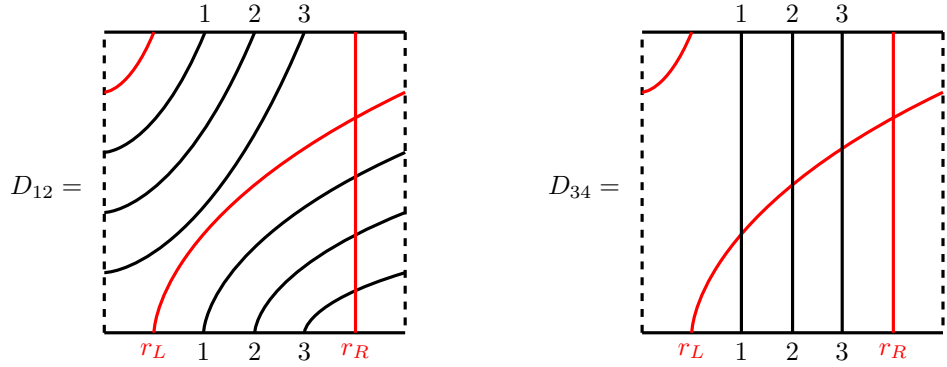
\begin{align*}
&&\coord_{13}&=\begin{array}{c}  
 \tikz[very thick,xscale=2,yscale=2]{
            \draw (-1,-1)-- (1,-1);
            \draw (-1,1)-- (1,1);
            \draw[dashed] (1,1)-- (1,-1);
          \draw[dashed] (-1,-1)-- (-1,1);
           \draw[red] (-.67,-1).. controls (-.67,-1) and (-.67,-.2).. node[below, at start]{$r_L$} (1,.6);
            \draw (-.34,-1) .. controls (-.34,-1) and (-.34,-0.4) ..node[below, at start]{$1$} (1,0.2);
            \draw (0,-1) .. controls (0,-1) and (0,-.6) ..node[below, at start]{$2$} (1,-.2);
            \draw (.33 ,-1).. controls (.33,-1) and (.33,-.8).. node[below, at start]{$3$} (1,-0.6);
           \draw[red] (.67,-1)--node[below, at start]{$r_R$} 
          (.67,1);
          \draw[red] (-1,.6).. controls (-1,.6) and (-.835,.6).. (-.67,1);
          \draw (-1,0.2).. controls (-1,0.2) and (-.66,.2).. node[above, at end]{$1$} (-.33,1);
          \draw (-1,-0.2).. controls (-1,-0.2) and (-0.5,-0.2).. node[above, at end]{$2$} (0,1);
          \draw (-1,-.6).. controls (-1,-.6) and (-.33,-.6).. node[above, at end]{$3$} (.33,1);
          \draw (0,1)-- (0,-1);
          }
   \end{array}&
    &&\coord_{14}&=\begin{array}{c}  
 \tikz[very thick,xscale=2,yscale=2]{
            \draw (-1,-1)-- (1,-1);
            \draw (-1,1)-- (1,1);
            \draw[dashed] (1,1)-- (1,-1);
          \draw[dashed] (-1,-1)-- (-1,1);
           \draw[red] (-.67,-1).. controls (-.67,-1) and (-.67,-.2).. node[below, at start]{$r_L$} (1,.6);
            \draw (-.34,-1) .. controls (-.34,-1) and (-.34,-0.4) ..node[below, at start]{$1$} (1,0.2);
            \draw (0,-1) .. controls (0,-1) and (0,-.6) ..node[below, at start]{$2$} (1,-.2);
            \draw (.33 ,-1)--node[above, at end]{$3$} node[below, at start]{$3$} (0.33,1);
           \draw[red] (.67,-1)--node[below, at start]{$r_R$} 
          (.67,1);
          \draw[red] (-1,.6).. controls (-1,.6) and (-.835,.6).. (-.67,1);
          \draw (-1,0.2).. controls (-1,0.2) and (-.66,.2).. node[above, at end]{$1$} (-.33,1);
          \draw (-1,-0.2).. controls (-1,-0.2) and (-0.5,-0.2).. node[above, at end]{$2$} (0,1);
          \draw (0,1)-- (0,-1);
          }
   \end{array}&
\end{align*}
\begin{align*}
&&\coord_{23}&=\begin{array}{c}  
 \tikz[very thick,xscale=2,yscale=2]{
            \draw (-1,-1)-- (1,-1);
            \draw (-1,1)-- (1,1);
            \draw[dashed] (1,1)-- (1,-1);
          \draw[dashed] (-1,-1)-- (-1,1);
           \draw[red] (-.67,-1).. controls (-.67,-1) and (-.67,-.2).. node[below, at start]{$r_L$} (1,.6);
            \draw (-.34,-1)--node[below, at start]{$1$} node[above, at end]{$1$} (-.34,1);
            \draw (0,-1) .. controls (0,-1) and (0,-.6) ..node[below, at start]{$2$} (1,-.2);
            \draw (.33 ,-1).. controls (.33,-1) and (.33,-.8).. node[below, at start]{$3$} (1,-0.6);
           \draw[red] (.67,-1)--node[below, at start]{$r_R$} 
          (.67,1);
          \draw[red] (-1,.6).. controls (-1,.6) and (-.835,.6).. (-.67,1);
          \draw (-1,-0.2).. controls (-1,-0.2) and (-0.5,-0.2).. node[above, at end]{$2$} (0,1);
          \draw (-1,-.6).. controls (-1,-.6) and (-.33,-.6).. node[above, at end]{$3$} (.33,1);
          \draw (0,1)-- (0,-1);
          }
   \end{array}&
    &&\coord_{24}&=\begin{array}{c}  
 \tikz[very thick,xscale=2,yscale=2]{
            \draw (-1,-1)-- (1,-1);
            \draw (-1,1)-- (1,1);
            \draw[dashed] (1,1)-- (1,-1);
          \draw[dashed] (-1,-1)-- (-1,1);
           \draw[red] (-.67,-1).. controls (-.67,-1) and (-.67,-.2).. node[below, at start]{$r_L$} (1,.6);
            \draw (-.34,-1)--node[below, at start]{$1$}node[above, at end]{$1$} (-.34,1);
            \draw (0,-1) .. controls (0,-1) and (0,-.6) ..node[below, at start]{$2$} (1,-.2);
            \draw (.33 ,-1)--node[above, at end]{$3$} node[below, at start]{$3$} (0.33,1);
           \draw[red] (.67,-1)--node[below, at start]{$r_R$} 
          (.67,1);
          \draw[red] (-1,.6).. controls (-1,.6) and (-.835,.6).. (-.67,1);
          \draw (-1,-0.2).. controls (-1,-0.2) and (-0.5,-0.2).. node[above, at end]{$2$} (0,1);
          \draw (0,1)-- (0,-1);
          }
   \end{array}&
\end{align*}
\begin{align*}
&&\coord_{12}&=\begin{array}{c}  
 \tikz[very thick,xscale=2,yscale=2]{
            \draw (-1,-1)-- (1,-1);
            \draw (-1,1)-- (1,1);
            \draw[dashed] (1,1)-- (1,-1);
          \draw[dashed] (-1,-1)-- (-1,1);
           \draw[red] (-.67,-1).. controls (-.67,-1) and (-.67,-.2).. node[below, at start]{$r_L$} (1,.6);
            \draw (-.34,-1) .. controls (-.34,-1) and (-.34,-0.4) ..node[below, at start]{$1$} (1,0.2);
            \draw (0,-1) .. controls (0,-1) and (0,-.6) ..node[below, at start]{$2$} (1,-.2);
            \draw (.33 ,-1).. controls (.33,-1) and (.33,-.8).. node[below, at start]{$3$} (1,-0.6);
           \draw[red] (.67,-1)--node[below, at start]{$r_R$} 
          (.67,1);
          \draw[red] (-1,.6).. controls (-1,.6) and (-.835,.6).. (-.67,1);
          \draw (-1,0.2).. controls (-1,0.2) and (-.66,.2).. node[above, at end]{$1$} (-.33,1);
          \draw (-1,-0.2).. controls (-1,-0.2) and (-0.5,-0.2).. node[above, at end]{$2$} (0,1);
          \draw (-1,-.6).. controls (-1,-.6) and (-.33,-.6).. node[above, at end]{$3$} (.33,1);
          }
   \end{array}&
    &&\coord_{34}&=\begin{array}{c}  
 \tikz[very thick,xscale=2,yscale=2]{
            \draw (-1,-1)-- (1,-1);
            \draw (-1,1)-- (1,1);
            \draw[dashed] (1,1)-- (1,-1);
          \draw[dashed] (-1,-1)-- (-1,1);
           \draw[red] (-.67,-1).. controls (-.67,-1) and (-.67,-.2).. node[below, at start]{$r_L$} (1,.6);
            \draw (-.34,-1)--node[below, at start]{$1$}node[above, at end]{$1$} (-.34,1);
            \draw (.33 ,-1)--node[above, at end]{$3$} node[below, at start]{$3$} (0.33,1);
           \draw[red] (.67,-1)--node[below, at start]{$r_R$} 
          (.67,1);
          \draw[red] (-1,.6).. controls (-1,.6) and (-.835,.6).. (-.67,1);
          \draw (0,1)--node[above, at start]{$2$} node[below, at end]{$2$} (0,-1);
          }
   \end{array}&
\end{align*}
In fact, scalars are not needed;  in order to check this, we only confirm that the Pl\"ucker relation 
\begin{equation*}
    D_{12}D_{34}-D_{13}D_{24}+D_{14}D_{23}=0
\end{equation*}
holds, rather than some other relation between scalar multiples of these diagrams. This calculation proceeds as follows:
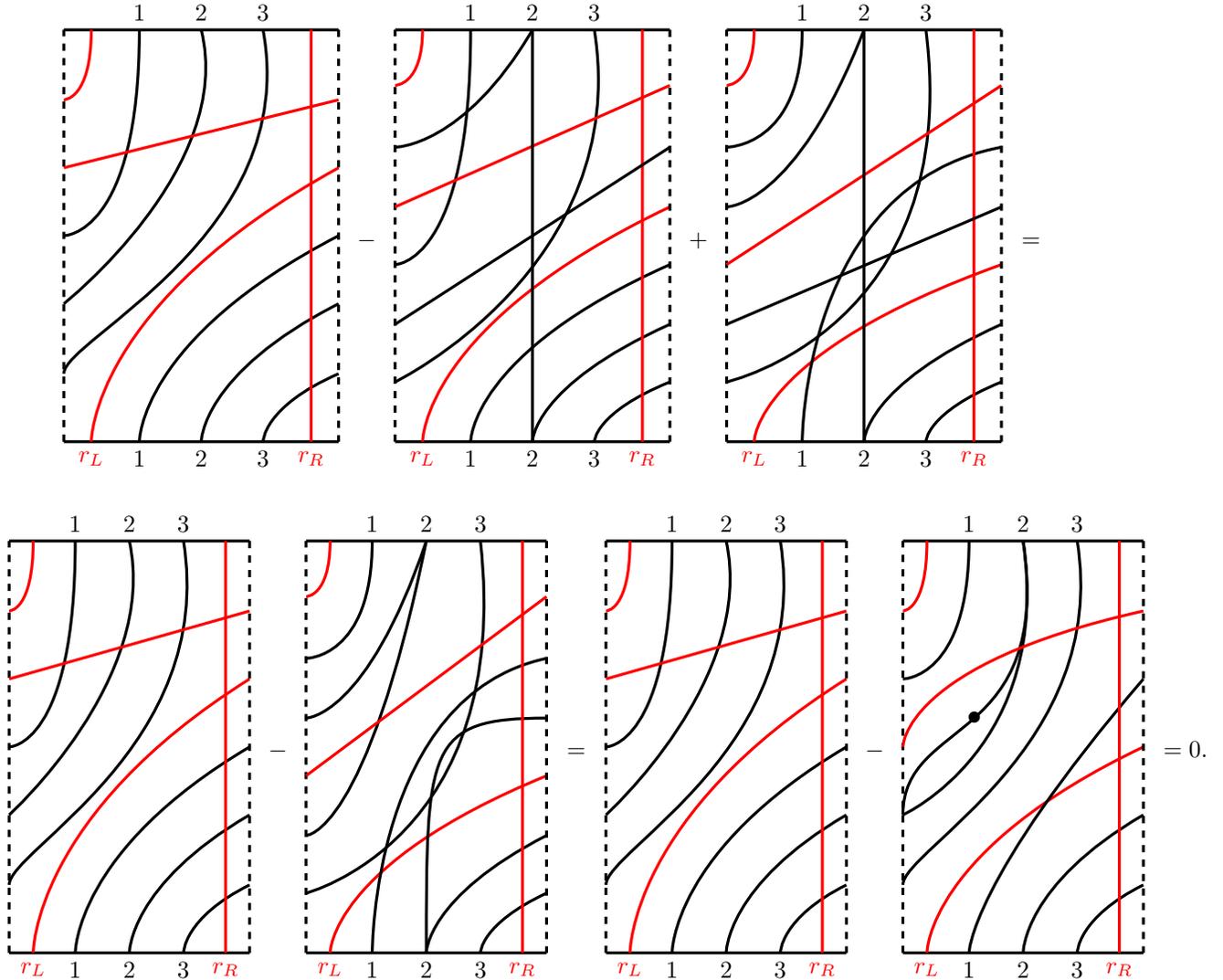
\begin{align*}
  \begin{array}{c}  
 \tikz[very thick,xscale=2,yscale=3]{
            \draw (-1,-1)-- (1,-1);
            \draw (-1,1)-- (1,1);
          \draw[dashed] (-1,-1)-- (-1,1);
           \draw[red] (-.8,-1) .. controls (-.8,-1) and (-.8,-.33) ..node[below, at start]{$r_L$} (1,.33);
            \draw (0,-1) .. controls (0,-1) and (0,-0.67) ..node[below, at start]{$2$} (1,-.33);
            \draw (-0.45,-1) .. controls (-.45,-1) and (-.45,-0.5) ..node[below, at start]{$1$} 
          (1,0);
           \draw (.45,-1) .. controls (.45,-1) and (.45,-.83) ..node[below, at start]{$3$}
          (1,-.67);
          \draw[red] (.8 ,-1)--node[below, at start]{$r_R$} (.8,1);
          \draw[red] (-1,0.66) .. controls (-1,0.66) and (-0.8,0.66) .. (-0.8,1);
          \draw (-1,0) .. controls (-1,0) and (-.45,0) ..node[above, at end]{$1$} (-.45,1);
          \draw (-1,-.33) .. controls (-1,-.33) and (.25,.33) ..node[above, at end]{$2$} (0,1);
          \draw (-1,-.67) .. controls (-1,-0.5) and (.725,0) ..node[above, at end]{$3$} (0.45,1);
          \draw[dashed] (1,1)-- (1,-1);
          \draw[red] (-1,.33)-- (1,.66); 
          }
   \end{array}-  \begin{array}{c}  
 \tikz[very thick,xscale=2,yscale=3]{
            \draw (-1,-1)-- (1,-1);
            \draw (-1,1)-- (1,1);
          \draw[dashed] (-1,-1)-- (-1,1);
           \draw[red] (-.8,-1) .. controls (-.8,-1) and (-.8,-.43) ..node[below, at start]{$r_L$} (1,.14);
            \draw (-0.45,-1) .. controls (-.45,-1) and (-.45,-0.57) ..node[below, at start]{$1$} (1,-.14);
            \draw (0,-1) .. controls (0,-1) and (0,-0.73) ..node[below, at start]{$2$} (1,-.43);
           \draw (.45,-1) .. controls (.45,-1) and (.45,-.86) ..node[below, at start]{$3$}
          (1,-.71);
          \draw[red] (.8 ,-1)--node[below, at start]{$r_R$} (.8,1);
          \draw[red] (-1,0.73) .. controls (-1,0.73) and (-0.8,0.73) .. (-0.8,1);
          \draw (-1,-.14) .. controls (-1,-.14) and (-.45,-.14) ..node[above, at end]{$1$} (-.45,1);
          \draw (-1,.43) .. controls (-1,.43) and (-.5,.43) ..node[above, at end]{$2$} (0,1);
          \draw (-1,-.71) .. controls (-1,-.71) and (.725,-.15) ..node[above, at end]{$3$} (0.45,1);
          \draw[dashed] (1,1)-- (1,-1);
          \draw (-1,-0.43)-- (1,0.43); 
          \draw[red] (-1,.14)-- (1,.73); 
          \draw (0,-1)-- (0,1) 
          }
   \end{array}+\begin{array}{c}  
 \tikz[very thick,xscale=2,yscale=3]{
            \draw (-1,-1)-- (1,-1);
            \draw (-1,1)-- (1,1);
          \draw[dashed] (-1,-1)-- (-1,1);
           \draw[red] (-.8,-1) .. controls (-.8,-1) and (-.8,-.57) ..node[below, at start]{$r_L$} (1,-.14);
            \draw (-0.45,-1) .. controls (-.45,-1) and (-.45,.25) ..node[below, at start]{$1$} (1,.43);
            \draw (0,-1) .. controls (0,-1) and (0,-0.73) ..node[below, at start]{$2$} (1,-.43);
           \draw (.45,-1) .. controls (.45,-1) and (.45,-.86) ..node[below, at start]{$3$}
          (1,-.71);
          \draw[red] (.8 ,-1)--node[below, at start]{$r_R$} (.8,1);
          \draw[red] (-1,0.73) .. controls (-1,0.73) and (-0.8,0.73) .. (-0.8,1);
          \draw (-1,.43) .. controls (-1,.43) and (-.45,.43) ..node[above, at end]{$1$} (-.45,1);
          \draw (-1,.14) .. controls (-1,.14) and (-.5,.14) ..node[above, at end]{$2$} (0,1);
          \draw (-1,-.71) .. controls (-1,-.71) and (.725,-.45) ..node[above, at end]{$3$} (0.45,1);
          \draw[dashed] (1,1)-- (1,-1);
          \draw (-1,-0.43)-- (1,0.14); 
          \draw[red] (-1,-.14)-- (1,.73); 
          \draw (0,-1)-- (0,1) 
          }
   \end{array} =
      \end{align*}
   \begin{align*}
      \begin{array}{c}  
 \tikz[very thick,xscale=1.75,yscale=3]{
            \draw (-1,-1)-- (1,-1);
            \draw (-1,1)-- (1,1);
          \draw[dashed] (-1,-1)-- (-1,1);
           \draw[red] (-.8,-1) .. controls (-.8,-1) and (-.8,-.33) ..node[below, at start]{$r_L$} (1,.33);
            \draw (0,-1) .. controls (0,-1) and (0,-0.67) ..node[below, at start]{$2$} (1,-.33);
            \draw (-0.45,-1) .. controls (-.45,-1) and (-.45,-0.5) ..node[below, at start]{$1$} 
          (1,0);
           \draw (.45,-1) .. controls (.45,-1) and (.45,-.83) ..node[below, at start]{$3$}
          (1,-.67);
          \draw[red] (.8 ,-1)--node[below, at start]{$r_R$} (.8,1);
          \draw[red] (-1,0.66) .. controls (-1,0.66) and (-0.8,0.66) .. (-0.8,1);
          \draw (-1,0) .. controls (-1,0) and (-.45,0) ..node[above, at end]{$1$} (-.45,1);
          \draw (-1,-.33) .. controls (-1,-.33) and (.25,.33) ..node[above, at end]{$2$} (0,1);
          \draw (-1,-.67) .. controls (-1,-0.5) and (.725,0) ..node[above, at end]{$3$} (0.45,1);
          \draw[dashed] (1,1)-- (1,-1);
          \draw[red] (-1,.33)-- (1,.66); 
          }
   \end{array}-\begin{array}{c}  
 \tikz[very thick,xscale=1.75,yscale=3]{
            \draw (-1,-1)-- (1,-1);
            \draw (-1,1)-- (1,1);
          \draw[dashed] (-1,-1)-- (-1,1);
           \draw[red] (-.8,-1) .. controls (-.8,-1) and (-.8,-.57) ..node[below, at start]{$r_L$} (1,-.14);
            \draw (-0.45,-1) .. controls (-.45,-1) and (-.45,.25) ..node[below, at start]{$1$} (1,.43);
            \draw (0,-1) .. controls (0,-1) and (0,-0.73) ..node[below, at start]{$2$} (1,-.43);
            \draw (0,-1) .. controls (0,0) and (0,.14) .. (1,.14); 
           \draw (.45,-1) .. controls (.45,-1) and (.45,-.86) ..node[below, at start]{$3$} (1,-.71);
          \draw[red] (.8 ,-1)--node[below, at start]{$r_R$} (.8,1);
          \draw[red] (-1,0.73) .. controls (-1,0.73) and (-0.8,0.73) .. (-0.8,1);
          \draw (-1,.43) .. controls (-1,.43) and (-.45,.43) ..node[above, at end]{$1$} (-.45,1);
          \draw (-1,.14) .. controls (-1,.14) and (-.5,.14) ..node[above, at end]{$2$} (0,1);
          \draw (-1,.-.43) .. controls (-1,-.43) and (-.5,-.43) .. (0,1); 
          \draw (-1,-.71) .. controls (-1,-.71) and (.725,-.45) ..node[above, at end]{$3$} (0.45,1);
          \draw[dashed] (1,1)-- (1,-1);
          \draw[red] (-1,-.14)-- (1,.73); 
          }
   \end{array}
   =\begin{array}{c}  
 \tikz[very thick,xscale=1.75,yscale=3]{
            \draw (-1,-1)-- (1,-1);
            \draw (-1,1)-- (1,1);
          \draw[dashed] (-1,-1)-- (-1,1);
           \draw[red] (-.8,-1) .. controls (-.8,-1) and (-.8,-.33) ..node[below, at start]{$r_L$} (1,.33);
            \draw (0,-1) .. controls (0,-1) and (0,-0.67) ..node[below, at start]{$2$} (1,-.33);
            \draw (-0.45,-1) .. controls (-.45,-1) and (-.45,-0.5) ..node[below, at start]{$1$} 
          (1,0);
           \draw (.45,-1) .. controls (.45,-1) and (.45,-.83) ..node[below, at start]{$3$}
          (1,-.67);
          \draw[red] (.8 ,-1)--node[below, at start]{$r_R$} (.8,1);
          \draw[red] (-1,0.66) .. controls (-1,0.66) and (-0.8,0.66) .. (-0.8,1);
          \draw (-1,0) .. controls (-1,0) and (-.45,0) ..node[above, at end]{$1$} (-.45,1);
          \draw (-1,-.33) .. controls (-1,-.33) and (.25,.33) ..node[above, at end]{$2$} (0,1);
          \draw (-1,-.67) .. controls (-1,-0.5) and (.725,0) ..node[above, at end]{$3$} (0.45,1);
          \draw[dashed] (1,1)-- (1,-1);
          \draw[red] (-1,.33)-- (1,.66); 
          }
   \end{array}-\begin{array}{c}  
 \tikz[very thick,xscale=1.75,yscale=3]{
            \draw (-1,-1)-- (1,-1);
            \draw (-1,1)-- (1,1);
          \draw[dashed] (-1,-1)-- (-1,1);
           \draw[red] (-.8,-1) .. controls (-.8,-1) and (-.8,-.5) ..node[below, at start]{$r_L$} (1,0);
            \draw (0,-1) .. controls (0,-1) and (0,-0.67) ..node[below, at start]{$2$} (1,-.33);
            \draw (-0.45,-1) .. controls (-.45,-1) and (-.45,-0.66) ..node[below, at start]{$1$} (1,0.33);
           \draw (.45,-1) .. controls (.45,-1) and (.45,-.83) ..node[below, at start]{$3$}
          (1,-.67);
          \draw[red] (.8 ,-1)--node[below, at start]{$r_R$} (.8,1);
          \draw[red] (-1,0.66) .. controls (-1,0.66) and (-0.8,0.66) .. (-0.8,1);
          \draw (-1,0.33) .. controls (-1,0.33) and (-.45,0.33) ..node[above, at end]{$1$} (-.45,1);
          \draw (-1,-.33) .. controls (-1,-.33) and (.25,0) ..node[above, at end]{$2$}(0,1);
          \draw (-1,-.33) .. controls (-1,.15) and (.25,0) ..node[midway]{\Large$\bullet$} (0,1);
          \draw (-1,-.67) .. controls (-1,-0.5) and (.725,0) ..node[above, at end]{$3$} (0.45,1);
          \draw[dashed] (1,1)-- (1,-1);
          \draw[red] (-1,0).. controls (-1,0) and (-1,.4).. (1,.66); 
          }
   \end{array}=0.
\end{align*}
This shows that in the case of $k=2,n=4$:
\begin{lemma}
    We can choose the isomorphism $T^*X\cong \tilde{\fM}$ so that the pullback of $e_{ij}$ is $D_{ij}$.
\end{lemma}
 
\section{Vector bundles on $T^*\Gr(2,4)$}
\subsection{General procedure}
We can use the isomorphism above to construct a very interesting collection of vector bundles on $T^*X$ via transport of structure from $\tilde{\fM}$.  For any graded module $M$ over the projective coordinate ring $\coul$, we can define its localization $\Loc{M}$.  We can construct this explicitly as follows.  

Given $f\in \coul^m=e\mathring{T}^me$, the scheme $\Proj(\coul)$ has an open subset $U_f=\Spec(\coul_f^0)$ given by the spectrum of the degree 0 part of the localization (in the sense of commutative algebra);  by definition,  $\Proj(\coul)$ is the colimit of these open subschemes over the natural inclusions $U_{fg}\subset U_f$.  Given a module $M$, we can define $\Loc{M}$ by describing its pullback to $U_f$ for each $f$, together with pullback maps under the inclusions above.  We do this by localizing the $\coul_f^0$-module $\mathbf{M}_f^0$.  Note that this means that we have a natural map $M^m\to \Gamma(\tilde\fM; \Loc{M}\otimes \mathcal{L}^m)$; if $M$ is finitely generated, this map will be an isomorphism for $m\gg 0$.

The most natural choice of $\coul$-module (beyond $\coul$ itself) is the module $e\mathbf{R}\cong \bigoplus_{m=0}^{\infty}e\mathring{T}^m$.  
\begin{theorem}
    [\mbox{\cite[Thm. 8.14]{WebcohII}}]\label{th:Z-tilting}
    The vector bundle $\tilting=\Loc{e\mathbf{R}}$ is a tilting generator and agrees with the tilting generator constructed in \cite{KalDEQ}.    
\end{theorem}

Note that Kaledin only canonically constructs the Azumaya algebra $\tilting\otimes \tilting^{\vee}$; for general symplectic resolutions, different choices of quantization parameters in characteristic $p$ will give different Azumaya algebras, but there is only one such algebra in the case of $T^*X$.  Choosing idempotents other than $e$ in $\mathbf{R}$ can give other tilting generators, as does tensoring with line bundles.

We can write $1\in \mathbf{R}$ as a sum of idempotents corresponding to words, so $\tilting$ decomposes into the sum of the images of these idempotents, which are $\tilting^{\Bi}=\Loc{e\mathbf{R}e(\mathbf{i})}$. By \cite[Lem. 2.9]{websterCoherentSheaves2019}, we have:
\begin{lemma}\label{lem:tilting-rank}
    The sheaf $\tilting^{\Bi}$ for $i_1^{(a_1)}\cdot i_r^{(a_r)}$ is a vector bundle of rank $\frac{v_1!\cdots v_{n-1}!}{a_1!\cdots a_r!}$.  In particular, if $n=4$ and $k=2$, then this rank is $2$ if all $a_i$'s are 1, and $1$ otherwise.
\end{lemma}

\subsection{The example of $k=1$}
\label{sec:Pn}

If $k=1$, and $n$ is arbitrary, then $X=T^*\mathbb{P}^{n-1}$ and $\Bw=(1,0,\dots,0,1), \Bv=(1,1,\dots, 1,1)$.  By \Cref{lem:tilting-rank}, the summands $\tilting^{\Bi}$ are all line bundles, and thus of the form $\mathcal{O}(a(\Bi))$ for some integer $a$.  Note that if there existed $\Bi,\Bj$ such that $a(\Bi)-a(\Bj)\geq n$, then we would have $\Ext^1(\tilting^{\Bi},\tilting^{\Bj})\cong H^1(T^*X; \cO(a(\Bj)-a(\Bi)))\neq 0$, so the bundle would not be tilting.  Thus, there must exist some integer $r$ such that $a(\Bi)\in \{r,r-1,\dots, r-n+1\}$ for all $\Bi$.  One can also see that every value in this interval must be realized as $a(\Bi)$, since otherwise the resulting bundle would not be a generator---the K-theory of $T^*\mathbb{P}^n$ is free of rank $n$ as an abelian group, and thus can't be generated by the classes of $<n$ line bundles.  

Thus, we need only establish how to compute $a$ and thus find $r$.  Note that $\tilting^{{\color{red}1}12\cdots n-1 {\color{red}n-1}}$ gives the structure sheaf, so we must have $r\geq 0$.  
\begin{definition}
Start at ${\color{red}1}$ and jump to the numbers $1,2,\dots,n-1$ and then to ${\color{red}n-1}$ in the word $\Bi$.
    Let $a'(\Bi)$ be the number of times this process jumps to the left.    
\end{definition}
We can also visualize this with the idempotent written on the cylinder as shown above, thinking of jumps to the left as going around the back of the cylinder so we always move in a positive direction around the circle---the statistic $a'(\Bi)$ will be the winding number of our path around the circle.

For example, if $\Bi={\color{red}1}12\cdots n-1 {\color{red}n-1}$, then $a'(\Bi)=0$.  On the other hand, if $\Bj={\color{red}1}{\color{red}n-1}n-1 \cdots 21$, then every jump after the first is leftward, so $a'(\Bj)=n-1$.

\begin{theorem}
    $a(\Bi)=-a'(\Bi)$.
\end{theorem}
\begin{proof}
    There is a unique diagram (up to isotopy) $D$ of twisting degree $a'(\Bi)$ which connects $\Bi$ to the Coulomb idempotent $e$ without any degree 1 crossing of strands.   We construct this by adding the strands with label $1,2,\dots$ in order so that the strand $A$ with label $s$ avoids crossing the strand $B$ with label $s-1$.  If $s$ comes before $s-1$ in $\Bi$, then this is achieved by sending this $A$ around the back of the cylinder, so it has winding number one lower than $B$.  This shows that the difference between the winding numbers of the red strands must be exactly $a'(\Bi)$. 

    Now consider a diagram $D'$ with twisting degree $\geq a'(\Bi)$.    For each strand, consider its winding number minus the winding number of the strand with the same label in $D$.  This ``corrected winding number'' has the property that for two strands with consecutive labels, the difference of corrected winding numbers tells us how many times that pair of strands intersects.
    
    Now we wish to show that $D'=D''D$ for some other twisted diagram $D''$---we prove this by induction on the difference between the maximum and minimum corrected winding numbers.  For simplicity, we can assume that this diagram is tightened, that is, no pair of strands makes a bigon and $D'$ has no dots, since these can easily be pulled into a factor like $D''$.  With these assumptions, if this difference between max and min is 0, then $D'=D$.  
    
    Let $b$ be the maximum of the corrected winding numbers.  If $b>0$, every strand with corrected winding number $b$ must have a positive crossing over every other strand, so we can isotope our diagram so that at $y=\nicefrac{1}{2}$, we have the Coulomb idempotent $e$, and the portion of the diagram above $y=\nicefrac{1}{2}$ has each strand with corrected winding number $b$ wrapping 1 time around the cylinder in the positive direction and each other strand wrapping 0 times.  Thus, if $b>0$, we can write $D'=D_1D_2$ with $D_2$ having lower difference between max and min corrected winding number; by induction $D_2=D''_2D$, so $D'=D_1D''_2D$.

    A symmetric argument shows that if the minimum corrected wrapping number is $<0$, we can pull off a similar factor.  This completes the proof that $D'=D''D$.

    This shows that $D$ is a non-vanishing section of $\tilting^{\Bi}\otimes \mathcal{O}(a'(\Bi))\cong \mathcal{O}(a(\Bi)+a'(\Bi))$, that is, it trivializes this bundle.  This is only possible if $a(\Bi)+a'(\Bi)=0$, completing the proof.  
\end{proof}
This allows us to directly compute what every summand of $\tilting$ is as a line bundle. Below we show examples of the diagram $D$ discussed in the above proof, whose twisting degree is $a'(\Bi)$, and the corresponding line bundle we obtain.
For $n=2$ we have:
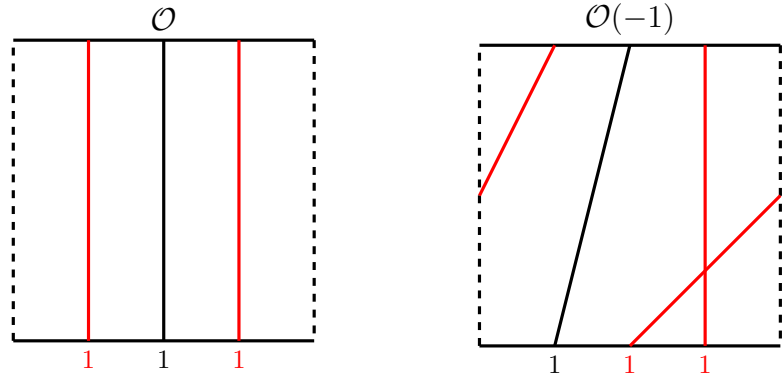
\begin{align*}
   &&\begin{array}{c}  
 \tikz[very thick,xscale=2,yscale=2]{
            \draw (-1,-1)-- (1,-1);
            \draw (-1,1)-- (1,1);
            \draw[dashed] (1,1)-- (1,-1);
          \draw[dashed] (-1,-1)-- (-1,1);
           \draw[red] (-.5,-1)--node[below, at start]{$1$} (-.5,1);
            \draw (-0,-1)--node[below, at start]{$1$}node[above, at end]{\Large$\mathcal{O}$} (0,1);
           \draw[red] (.5,-1)--node[below, at start]{$1$} 
          (.5,1);
          }
   \end{array} && \begin{array}{c}  
 \tikz[very thick,xscale=2,yscale=2]{
            \draw (-1,-1)-- (1,-1);
            \draw (-1,1)-- (1,1);
            \draw[dashed] (1,1)-- (1,-1);
          \draw[dashed] (-1,-1)-- (-1,1);
           \draw[red] (0,-1)-- node[below, at start]{$1$} (1,0);
           \draw[red] (-1,0)-- (-.5,1);
            \draw (-0.5,-1)--node[below, at start]{$1$}node[above, at end]{\Large$\mathcal{O}(-1)$} (0,1);
           \draw[red] (.5,-1)--node[below, at start]{$1$} 
          (.5,1);
          }
   \end{array} 
\end{align*}
and for $n=3$ we have:
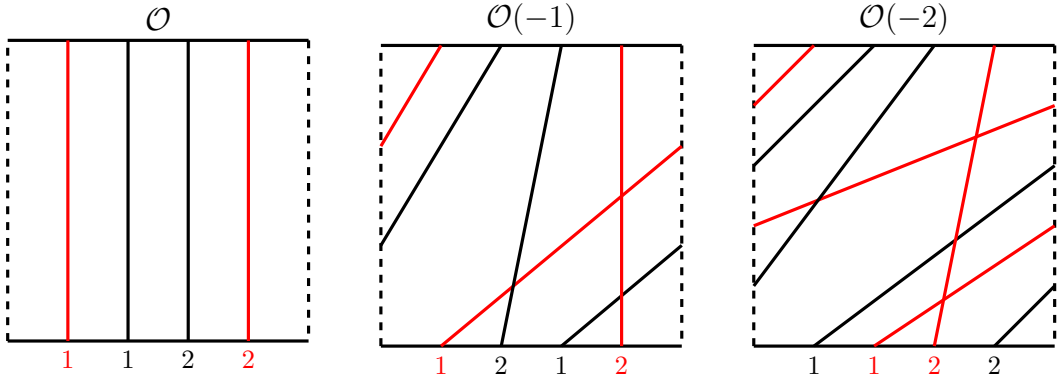
\begin{align*}
   &&\begin{array}{c}  
 \tikz[very thick,xscale=2,yscale=2]{
            \draw (-1,-1)-- (1,-1);
            \draw (-1,1)-- (1,1);
            \draw[dashed] (1,1)-- (1,-1);
          \draw[dashed] (-1,-1)-- (-1,1);
           \draw[red] (-.6,-1)--node[below, at start]{$1$} (-.6,1);
            \draw (-.2,-1)--node[below, at start]{$1$} (-.2,1);
            \path (0,-1)-- node[above, at end]{\Large$\mathcal{O}$} (0,1);
               \draw (.2,-1)--node[below, at start]{$2$} (.2,1);
           \draw[red] (.6,-1)--node[below, at start]{$2$} 
          (.6,1);
          }
   \end{array} && \begin{array}{c}  
 \tikz[very thick,xscale=2,yscale=2]{
            \draw (-1,-1)-- (1,-1);
            \draw (-1,1)-- (1,1);
            \draw[dashed] (1,1)-- (1,-1);
          \draw[dashed] (-1,-1)-- (-1,1);
           \draw[red] (-.6,-1)--node[below, at start]{$1$} (1,.33);
            \draw (.2,-1)--node[below, at start]{$1$} (1,-.33);
            \path (0,-1)-- node[above, at end]{\Large$\mathcal{O}(-1)$} (0,1);
            \draw (-.2,-1)--node[below, at start]{$2$} (.2,1);
            \draw[red] (-1,.33)-- (-.6,1);
            \draw (-1,-.33)-- (-.2,1);
           \draw[red] (.6,-1)--node[below, at start]{$2$} 
          (.6,1);
          }
   \end{array} &&
   \begin{array}{c}  
 \tikz[very thick,xscale=2,yscale=2]{
            \draw (-1,-1)-- (1,-1);
            \draw (-1,1)-- (1,1);
            \draw[dashed] (1,1)-- (1,-1);
          \draw[dashed] (-1,-1)-- (-1,1);
           \draw[red] (-.2,-1)--node[below, at start]{$1$} (1,.-.2);
            \draw (-.6,-1)--node[below, at start]{$1$} (1,.2);
            \path (0,-1)-- node[above, at end]{\Large$\mathcal{O}(-2)$} (0,1);
            \draw (.6,-1)--node[below, at start]{$2$} (1,-.6);
            \draw[red] (-1,.6)-- (-.6,1);
            \draw[red] (-1,-.2)-- (1,.6);
            \draw (-1,.2)-- (-.2,1);
            \draw (-1,-.6)-- (.2,1);
           \draw[red] (.2,-1)--node[below, at start]{$2$} (.6,1);
          }
   \end{array} 
\end{align*}
Since we know that $0$ and $n-1$ are realized as values of $a'(\Bi)$, we find immediately that $r=0$ and
\[\tilting\cong \cO\oplus \cO(-1)^{\oplus m_1}\oplus \cdots \cO(-n+1)^{\oplus m_{n-1}}\] where $m_s$ is the number of words with $s=a'(\Bi)$. \\
\Aiden{Should I Tex up any of the diagrams you mention for $n=3$ later this week? If so, which ones would be helpful?}

\subsection{Calculations}
Our main purpose in this paper is to compute the representations that appear as $\tilting^{\Bi}$ in the case of $n=4,k=2$.   Let us state one useful prepatory result.  Due to topological constraints, we have written our words linearly; when we discuss ``reordering'' a word, that implicitly includes moving the last entry to be the first and vice versa.  For example, we can obtain \textcolor{red}{2}1322\textcolor{red}{2} from \textcolor{red}{2}322\textcolor{red}{2}1 by ``reordering'' the 1 and the first \textcolor{red}{2}.  Standard results in KLRW algebras show that:
\begin{lemma}\label{wordlemma}
    \hfill
    \begin{enumerate}
        \item If we reorder two letters both of which are not black 2's, then $\tilting^{\Bi}$ is unchanged.
        \item $\tilting^{\cdots 22\cdots}\cong (\tilting^{\cdots 2^{(2)}\cdots})^{\oplus 2}$
        \item If we have a subword $2i2$ for $i\in \{1,3,{\color{red} 2}\}$, then 
        \[\tilting^{\cdots 2i2\cdots}\cong \tilting^{\cdots i2^{(2)}\cdots}\oplus \tilting^{\cdots 2^{(2)}i\cdots}\]
    \end{enumerate}
\end{lemma}
Thus, we will obtain a tilting generator equiconstituted to $\tilting$ if we take the sum of $\tilting^{\Bi}$ where $\Bi$ ranges over those where we have a $2^{(2)}$ (in which case we have a line bundle) or when the pair of black $2$'s is separated on both sides of the circle by two elements of the set $\{1,3,{\color{red} 2}\}$.  
Up to reordering using \Cref{wordlemma}(1), there are only two cases where we have $2^{(2)}$, determined by which of the gaps between red ${\color{red} 2}$'s the single appearance of $2^{(2)}$ lies in:
\begin{equation}\label{eq:line-words}
{\color{red} 2}12^{(2)}3{\color{red} 2}\qquad \qquad {\color{red} 2}{\color{red} 2}12^{(2)}3.
\end{equation}
There are four possibilities where the pair of black $2$'s are separated:
\begin{equation}\label{eq:rank-2-words}
    {\color{red}2}2132 {\color{red}2}\qquad\qquad {\color{red}22}2312 \qquad \qquad {\color{red}2}23{\color{red}2}21\qquad\qquad  {\color{red}2}21{\color{red}2}23
\end{equation} 
From this point on, we need only consider these cases.
\begin{lemma}
    In the cases of (\ref{eq:line-words}--\ref{eq:rank-2-words}), the corresponding vector bundle $\tilting^{\Bi}$ is indecomposable.  
\end{lemma}
\begin{proof}
    We can compute $\End(\tilting^{\Bi})=e(\Bi)Re(\Bi)$, so if $\tilting^{\Bi}$ is decomposable, there is an endomorphism of scaling degree 0 given by projection to one of the summands.  But one can readily check that the identity spans degree 0 diagrams in these cases.   
\end{proof}

To identify which vector bundles these $\mathcal{A}$-modules correspond to, we will consider the localization of these modules with respect to the open subsets where $\coord_{13}\neq0$ and where $\coord_{24}\neq0$. We do this for two reasons: the first is that it will assist us in constructing free generators for the modules and the second is that it will allow us to compute transition matrices of the corresponding locally free coherent sheaf which we can then compare with the results of 
Section \ref{sec:TGr}. 

We will outline here the argument used to find generators for these $\coul$-modules and this argument will essentially be the proof that the pairs of diagrams given below in Sections \ref{sec:trivial} and \ref{sec:non-trivial} below that
for each idempotent and patch do indeed generate the corresponding module.

By definition, for a given word $\Bi$ consisting of one black 1, two black 2's, one black 3 and two red \textcolor{red}{2}'s, we have an idempotent for $\coul$. We are interested in the $\coul$-module $M^{\Bi} =e\mathbf{R}e(\Bi)$ whose elements are KLRW diagrams with the word $\Bi$ labeling the bottom of the diagram and the word for the Coulomb branch idempotent $\mathbf{j}=$\textcolor{red}{2}123\textcolor{red}{2} labeling the top. 

This is not a free $\coul$-module, but we know by \Cref{lem:tilting-rank} that the module localized to the patch where $\coord_{ij}\neq0$, which we denote $M^\Bi_{ij}$, is free and of rank 2 or 1 for any $i,j$. To calculate the transition functions, we find generators of $M^\Bi_{ij}$ for the patches with $(i,j)=(1,3),(2,4)$.

By applying Nakayama's lemma to the localizations at all different prime ideals, a surjective map between rank $k$ free modules over $\coul$ is an isomorphism, so a generating set whose size matches the rank of the vector bundle must necessarily be a free generating set.

To construct these generating sets, we will extensively use the following fact:
\begin{lemma}
Given a diagram in $M^\Bi_{ij}$, let $(a_0,0)$, $(b_0,0)$ be a pair of points at the bottom of the diagram, $(a_1,1),(b_1,1)$ a pair of points at the top of the diagram, and $\alpha$ be a continuous path from $(a_0,0)$ to $(a_1,1)$. Then there exists a unique path up to isotopy $\beta$ from $(b_0,0)$ to $(b_1,1)$ such that $\alpha$ and $\beta$ do not intersect.
\end{lemma}
We will refer in the discussion below to the ``left'' and ``right'' red strands and black strands with label $2$ based on their position in the unrolled diagram at the bottom of the diagram.  
\begin{algorithm}\label{algorithm13}
Assume that $k\geq 2$.  The generating diagrams of $M^\Bi_{13}$ with winding grading $k\in\Z$, which we denote $\Zij{s}{13}{k}$ for $s=1,2$, can be constructed as follows:
\begin{enumerate}
    \item Begin by wrapping the left red strand around the cylinder $k$ times in the positive direction. The right red strand is vertical.  
    \item \begin{enumerate}
        \item If $s=1$, then we connect the left black terminal with label $2$ to the unique such terminal at the top of the diagram, avoiding intersecting the left red strand, and similarly with the right black terminal, avoiding intersection with the right red strand and minimizing the number of crossings with other strands (i.e. not creating any bigons with other strands).
        \item  If $s=2$, we do the same but with the left black strand avoiding the right red and {\it vice versa}.  For simplicity of the discussion below, we say that the black and red strand that avoid intersection above are {\bf tied.}
    \end{enumerate} 
    \item The unique terminals for 1 and 3 on the top and bottom are connected so that they avoid crossing the black strand tied to the left red strand.  That is, if $s=1$ (resp.\ $s=2$), they avoid crossing the left (resp.\ right) black strand.
\end{enumerate}

\begin{examplex}
	For example, when making $\Zij{2}{13}{1}$ $\Zij{2}{13}{2}$   
we begin both diagrams with the same diagram with only red strands, but first add a black strand tied to the left red strand connecting to the left and right black terminals at the bottom of the diagram, respectively:  
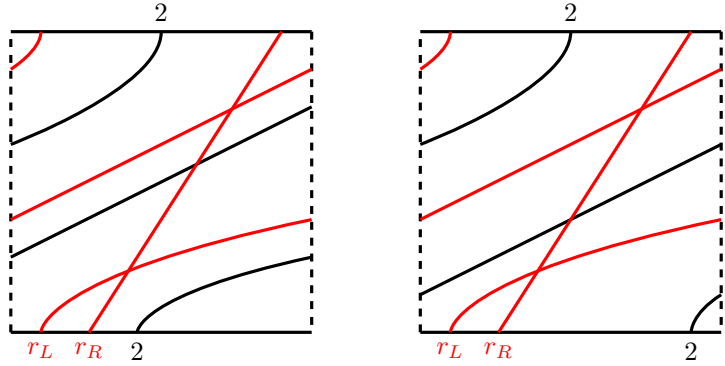
\[\tikz[very thick,xscale=2,yscale=2]{
            \draw (-1,-1)-- (1,-1);
            \draw (-1,1)-- (1,1);
          \draw[dashed] (-1,-1)-- (-1,1);
          \draw[dashed] (1,1)-- (1,-1);
          \draw[red] (-1,.75).. controls (-1,.75) and (-0.8,.875).. (-0.8,1);
          \draw (-1,.25).. controls (-1,.25) and (0,.625).. node[above, at end]{$2$} (0,1);
          \draw[red] (-1,-0.25)-- (1,0.75);
          \draw (-1,-0.5)-- (1,0.5);
          \draw[red] (-0.8,-1).. controls (-0.8,-1) and (-0.8,-.625).. node[below, at start]{$r_L$} (1,-0.25);
          \draw[red] (-.48,-1)-- node[below, at start]{$r_R$} (0.8,1);
          \draw (-.16,-1).. controls (-.16,-1) and (-.16,-.75).. node[below, at start]{$2$} (1,-0.5);
          }\qquad\qquad \tikz[very thick,xscale=2,yscale=2]{
            \draw (-1,-1)-- (1,-1);
            \draw (-1,1)-- (1,1);
          \draw[dashed] (-1,-1)-- (-1,1);
          \draw[dashed] (1,1)-- (1,-1);
          \draw[red] (-1,.75).. controls (-1,.75) and (-0.8,.875).. (-0.8,1);
          \draw (-1,.25).. controls (-1,.25) and (0,.625).. node[above, at end]{$2$} (0,1);
          \draw[red] (-1,-0.25)-- (1,0.75);
          \draw (-1,-0.75)-- (1,0.25);
          \draw[red] (-0.8,-1).. controls (-0.8,-1) and (-0.8,-.625).. node[below, at start]{$r_L$} (1,-0.25);
          \draw[red] (-.48,-1)-- node[below, at start]{$r_R$} (0.8,1);
          \draw (.8,-1).. controls (.8,-1) and (.8,-.875).. node[below, at start]{$2$} (1,-.75);
          }\]
 We then add a second red strand, tied to the right red strand, and eventually merging with the first black strand.
 
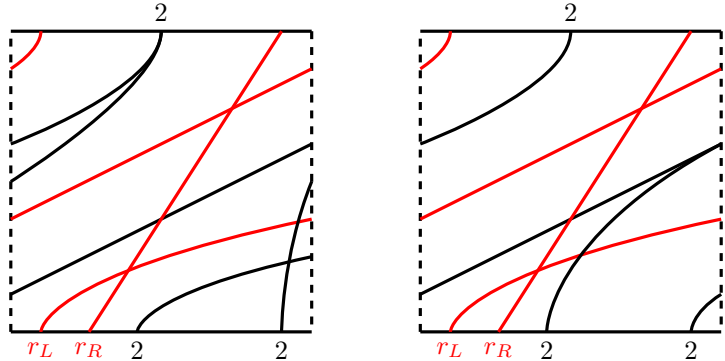
\[\tikz[very thick,xscale=2,yscale=2]{
            \draw (-1,-1)-- (1,-1);
            \draw (-1,1)-- (1,1);
          \draw[dashed] (-1,-1)-- (-1,1);
          \draw[dashed] (1,1)-- (1,-1);
          \draw[red] (-1,.75).. controls (-1,.75) and (-0.8,.875).. (-0.8,1);
          \draw (-1,.25).. controls (-1,.25) and (0,.625).. node[above, at end]{$2$} (0,1);
                    \draw (-1,0).. controls (-1,0) and (0,.625)..  (0,1);
          \draw[red] (-1,-0.25)-- (1,0.75);
          \draw (-1,-0.75)-- (1,0.25);
          \draw[red] (-0.8,-1).. controls (-0.8,-1) and (-0.8,-.625).. node[below, at start]{$r_L$} (1,-0.25);
          \draw[red] (-.48,-1)-- node[below, at start]{$r_R$} (0.8,1);
                    \draw (.8,-1).. controls (.8,-1) and (.8,-.5).. node[below, at start]{$2$} (1,0);
          \draw (-.16,-1).. controls (-.16,-1) and (-.16,-.75).. node[below, at start]{$2$} (1,-0.5);
          }\qquad\qquad \tikz[very thick,xscale=2,yscale=2]{
            \draw (-1,-1)-- (1,-1);
            \draw (-1,1)-- (1,1);
          \draw[dashed] (-1,-1)-- (-1,1);
          \draw[dashed] (1,1)-- (1,-1);
          \draw[red] (-1,.75).. controls (-1,.75) and (-0.8,.875).. (-0.8,1);
          \draw (-1,.25).. controls (-1,.25) and (0,.625).. node[above, at end]{$2$} (0,1);
          \draw[red] (-1,-0.25)-- (1,0.75);
          \draw (-1,-0.75)-- (1,0.25);
          \draw[red] (-0.8,-1).. controls (-0.8,-1) and (-0.8,-.625).. node[below, at start]{$r_L$} (1,-0.25);
          \draw[red] (-.48,-1)-- node[below, at start]{$r_R$} (0.8,1);
          \draw (.8,-1).. controls (.8,-1) and (.8,-.875).. node[below, at start]{$2$} (1,-.75);
                    \draw (-.16,-1).. controls (-.16,-1) and (-.16,-.375).. node[below, at start]{$2$} (1,0.25);
          }\]
Finally, we perform step 3., adding in the strands with labels 1 and 3:
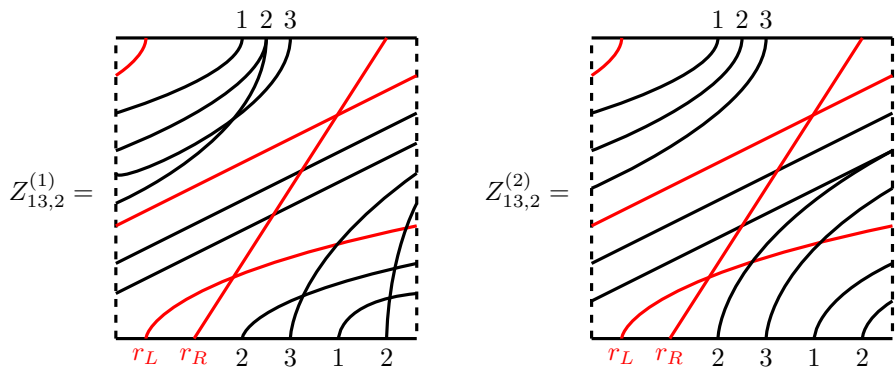
\[\Zij{2}{13}{1}=\begin{array}{c}  
 \tikz[very thick,xscale=2,yscale=2]{
           \draw (-1,-1)-- (1,-1);
            \draw (-1,1)-- (1,1);
          \draw[dashed] (-1,-1)-- (-1,1);
          \draw[dashed] (1,1)-- (1,-1);
          \draw[red] (-1,.75).. controls (-1,.75) and (-0.8,.875).. (-0.8,1);
          \draw (-1,.25).. controls (-1,.25) and (0,.625).. node[above, at end]{$2$} (0,1);
                    \draw (-1,-.1).. controls (-1,-.1) and (0,.325)..  (0,1);
          \draw[red] (-1,-0.25)-- (1,0.75);
          \draw (-1,-0.7)-- (1,0.3);
          \draw (-1,-0.5)-- (1,0.5);
          \draw[red] (-0.8,-1).. controls (-0.8,-1) and (-0.8,-.625).. node[below, at start]{$r_L$} (1,-0.25);
          \draw[red] (-.48,-1)-- node[below, at start]{$r_R$} (0.8,1);
                    \draw (.8,-1).. controls (.8,-1) and (.8,-.5).. node[below, at start]{$2$} (1,-.1);
          \draw (-.16,-1).. controls (-.16,-1) and (-.16,-.75).. node[below, at start]{$2$} (1,-0.5);
          \draw (-1,.5).. controls (-1,.5) and (-.16,.75).. node[above, at end]{$1$} (-.16,1);
          \draw (-1,.1).. controls (-1,0) and (.16,0.5).. node[above, at end]{$3$} (0.16,1);
                    \draw (0.16,-1).. controls (0.16,-1) and (0.16,-0.5).. node[below, at start]{$3$} (1,.1);
          \draw (.48,-1).. controls (.48,-1) and (.48,-.75).. node[below, at start]{$1$} (1,-.7);
          }
   \end{array}\qquad 
\Zij{2}{13}{2}=\begin{array}{c}  
 \tikz[very thick,xscale=2,yscale=2]{
            \draw (-1,-1)-- (1,-1);
            \draw (-1,1)-- (1,1);
          \draw[dashed] (-1,-1)-- (-1,1);
          \draw[dashed] (1,1)-- (1,-1);
          \draw[red] (-1,.75).. controls (-1,.75) and (-0.8,.875).. (-0.8,1);
          \draw (-1,.5).. controls (-1,.5) and (-.16,.75).. node[above, at end]{$1$} (-.16,1);
          \draw (-1,.25).. controls (-1,.25) and (0,.625).. node[above, at end]{$2$} (0,1);
          \draw (-1,0).. controls (-1,0) and (.16,0.5).. node[above, at end]{$3$} (0.16,1);
          \draw[red] (-1,-0.25)-- (1,0.75);
          \draw (-1,-0.5)-- (1,0.5);
          \draw (-1,-0.75)-- (1,0.25);
          \draw[red] (-0.8,-1).. controls (-0.8,-1) and (-0.8,-.625).. node[below, at start]{$r_L$} (1,-0.25);
          \draw[red] (-.48,-1)-- node[below, at start]{$r_R$} (0.8,1);
          \draw (-.16,-1).. controls (-.16,-1) and (-.16,-.375).. node[below, at start]{$2$} (1,0.25);
          \draw (0.16,-1).. controls (0.16,-1) and (0.16,-0.5).. node[below, at start]{$3$} (1,0);
          \draw (.48,-1).. controls (.48,-1) and (.48,-.75).. node[below, at start]{$1$} (1,-.5);
          \draw (.8,-1).. controls (.8,-1) and (.8,-.875).. node[below, at start]{$2$} (1,-.75);
          }
   \end{array}\]
\end{examplex}
If $k<2$, we let $\Zij{s}{13}{k}=\coord_{13}^{k-2}\Zij{s}{13}{2}$.  
\end{algorithm}
It might be hard to see from this description how $(i,j)=(1,3)$ is relevant; it is perhaps easiest to see when we consider why:
\begin{lemma}\label{lem:k-same}
For all $k,k'$, we have $\Zij{s}{13}{k}=\coord_{13}^{k-k'}\Zij{s}{13}{k'}$.  
\end{lemma}
\begin{proof}
It's enough to show that if $k\geq 2$, then $\Zij{s}{13}{k+1}=\coord_{13}\Zij{s}{13}{k}$.
    Since $k\geq 2$, the black strand $R$ tied to the left red will meet the other black strand $R'$ one from the left---the strand $R$ must have winding number $k+1$, while the other $R'$ must have winding number $\leq 1$ (depending on which side of the right red strand it lies on).

    Thus, in $\coord_{13}\Zij{s}{13}{k}$, the strand $R$ connects to the strand which winds around in $\coord_{13}$, and $R'$ to the one which remains vertical.  Thus, in $\coord_{13}$, the left red strand and strands with labels $1$ and $3$ do not intersect $R$, while the right red strand does not intersect $R'$.  Thus, if the product diagram $\coord_{13}\Zij{s}{13}{k}$ has no bigons, then we have $\Zij{s}{13}{k+1}$. Such a bigon is impossible for $k\geq 2$, since it would have to involve the strand $R$, and for $k\geq 2$, all the other black strands have winding number $\geq 1$ while $R'$ must have winding number $\leq 1$.
\end{proof}
Similar arguments apply when we consider $(i,j)=(1,4),(2,3),(2,4)$, suitably modifying step (3) so that strands that wrap in $\coord_{ij}$ avoid crossing the black strand tied to the left red and those that are vertical in $\coord_{ij}$.  In particular, for $(i,j)=(2,4)$, we have:
\begin{itemize}
    \item[3'.] Connect the unique terminals for 1 and 3 on the top and bottom so they avoid intersections with the black strand tied to the right red strand. 
\end{itemize}
\begin{remark}
    This approach does not work for $\coord_{12},\coord_{34}$; it is an interesting question how we modify this approach to work in these cases.  
\end{remark}
\begin{theorem}
    For $k\gg 0$, the diagrams $\Zij{1}{13}{k},\Zij{2}{13}{k}$ are a free basis for $M^{\Bi}_{13}$.
\end{theorem}
\begin{proof}
    As discussed above, it is enough to show that these are generators, and since we have inverted $\coord_{13}$, 
    by \Cref{lem:k-same}, proving this for one large $k$ establishes it for all $k$ greater than that.

    Let $m_1,\dots, m_p$ be a set of generators of this module.  By multiplying by powers of $\coord_{13}$, we can assume that all of these have the same twisting degree $k$, 
    and for some fixed $a\geq 0$, we have the winding number of the strands with labels $1,3$ and the sum of the winding numbers for the black strands with label $2$ are in the interval $[k-a,k+a]$.  Note that multiplying by $\coord_{13}$ increases all these quantities by 1, so this condition is unchanged by increasing $k$;  in particular, we can assume that $k\geq 3a$.  
    
    We can also assume without loss of generality that our diagram is ``pulled taut,'' i.e. no pair of strands make a bigon.   In this case, there must be a first point where the black strands meet.  Let $h$ be the height of this intersection. 
 Since the sum of the winding numbers of the two black strands is $\leq k+a$, at least one of these strands must have winding number $\leq \frac{k+a}{2}\leq k-a$.  Thus, this intersection must come at the top of a triangle made by the two black strands, and by the tautness, the right side of the triangle is the black strand $R$ of label $2$ with lower winding number and the left side is the black strand $L$ of label $2$ with greater winding number.  If any red strand or strand with label 1 or 3 crosses both sides of the triangle,  then we can shrink the triangle to remove this crossing, so each other strand crosses at most one side of this triangle. Since the left red strand and both strands with label $1$ and $3$ have greater winding number than $R$, they must cross it at least once, and thus, after isotopy, we can assume that these strands only intersect a black with label $2$ below $y=h$ by crossing $R$ at most once and never crossing $L$.  On the other hand, the right red strand has winding number $0$, and thus we can arrange so that its only crossing with a black strand with label 2 below $y=h$ is crossing $L$ at most once, and never crossing $R$.  Thus, by isotoping the strands with labels other than 2, we can arrange that the cross section at height $y=h$ is $e$ and the portion of the diagram below $y=h$ is $\Zij{1}{13}{k}$ or $\Zij{2}{13}{k}$, depending on whether $L$ is the left black strand or the right one.  In fact, from this argument one can see that $k=2$ suffices, since if $k>2$, there are multiple intersection points between the black strands in $\Zij{1}{13}{k}$ and $\Zij{2}{13}{k}$.
 \end{proof}

 Given this basis, we now turn to computing the transition matrices with respect to it.

\subsection{Rank 1 cases}
\label{sec:trivial}
In this section, we identify the line bundles corresponding to the words in \eqref{eq:line-words}:
\begin{equation*}
    \tilting^{\textcolor{red}{2}12^{(2)}3\textcolor{red}{2}} \hspace{1in}\tilting^{\textcolor{red}{22}12^{(2)}3}.
\end{equation*}
For the former, simply note that $e(\textcolor{red}{2}12^{(2)}3\textcolor{red}{2})=e$,  we have
    $\tilting^{\textcolor{red}{2}12^{(2)}3\textcolor{red}{2}}\cong\Loc{eRe}\cong\mathcal{O}$. 
For the latter, we note that $\tilting^{\textcolor{red}{22}12^{(2)}3}\otimes\mathcal{L}$ is generated by
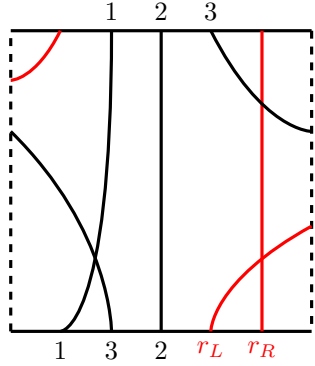
\begin{align*}
    \begin{array}{c}  
 \tikz[very thick,xscale=2,yscale=2]{
            \draw (-1,-1)-- (1,-1);
            \draw (-1,1)-- (1,1);
            \draw[dashed] (1,1)-- (1,-1);
          \draw[dashed] (-1,-1)-- (-1,1);
           \draw[red] (.33,-1).. controls (.33,-1) and (.33,-.67).. node[below, at start]{$r_L$} (1,-.3);
            \draw (-.67,-1).. controls (-.67,-1) and (-.33,-1).. node[below, at start]{$1$}node[above, at end]{$1$} (-.33,1);
           \draw[red] (.67,-1)--node[below, at start]{$r_R$} 
          (.67,1);
          \draw[red] (-1,.67).. controls (-1,.67) and (-.835,.67).. (-.67,1);
          \draw (0,1)--node[above, at start]{$2$} node[below, at end]{$2$} (0,-1);
            \draw (-.33,-1).. controls (-.33,-1) and (-.33,-.33)..n node[below, at start]{$3$} (-1,0.33);
        \draw (1,.33).. controls (1,.33) and (.67,.33).. node[above, at end]{$3$} (.33,1);
          }
   \end{array}.
\end{align*}
This diagram contains no relevant intersections and hence trivializes $\tilting^{\textcolor{red}{22}12^{(2)}3}\otimes\mathcal{L}$ on all coordinate patches $U_{ij}$. This implies
\begin{align*}
    &&\tilting^{\textcolor{red}{22}12^{(2)}3}\otimes\mathcal{L}&\cong\mathcal{O}&\\
    &\implies& \tilting^{\textcolor{red}{22}12^{(2)}3}&\cong \mathcal{L}^{-1}.&
\end{align*}
To summarize:
\begin{lemma}\label{lem:line-isos}  We have isomorphisms
$\tilting^{\textcolor{red}{2}12^{(2)}3\textcolor{red}{2}}\cong \mathcal{O}$ and  $\tilting^{\textcolor{red}{2}\textcolor{red}{2}12^{(2)}3}\cong \mathcal{L}^{-1}.$
\end{lemma} 
\subsection{Rank 2 cases}
\label{sec:non-trivial}

In the following four subsections, we apply the algorithm \ref{algorithm13} and its modification to the patch $U_{24}$ to each of the remaining idempotents identified in \eqref{eq:rank-2-words} to find explicit diagrams for each generator $\Zij{s}{ij}{k}$. The choice of generators produced is not unique.

Once sets of generators for $M^\Bi_{13}$ and $M^\Bi_{24}$ have been found, we use the KLRW relations to find the transition matrices $\gamma^{13}_{24}[\Bi]$ for the vector bundles $\tilting^\Bi$. As explained in Section 2, this is sufficient to identify each vector bundle up to isomorphism.

To do this, we formulate an ansatz for the equations that determine the transition functions. This is achieved by considering the winding and twisting degrees of the possible diagrams allowed in each equation. Since the winding and twisting gradings must agree on each side of the equation, one can deduce which coordinate diagrams may appear as coefficients to the $\Zij{s}{24}{k}$ diagrams by considering the winding degrees appearing in a given $\Zij{s}{13}{k}$ diagram. Since $\coord_{13}$ and $\coord_{24}$ are invertible on $U_{13}\cap U_{24}$, we may freely multiply $\Zij{s}{13}{k}$ on the LHS by powers of these coordinate diagrams to increase or decrease the twisting degree as needed.
\subsubsection{Sheaf with idempotent \textcolor{red}{2}2132\textcolor{red}{2}}
Applying algorithm \ref{algorithm13} and its variant to the case where  $\Bi$=\textcolor{red}{2}2132\textcolor{red}{2} produces the following generating diagrams:
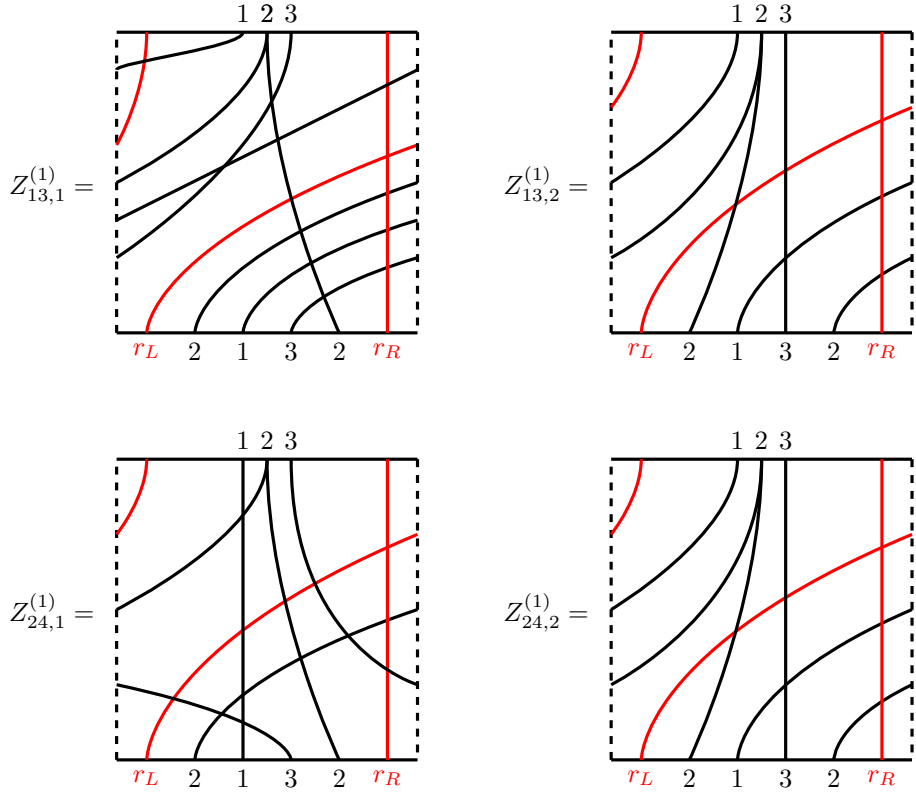
\begin{align*}
   &&\Zij{1}{13}{1}&=\begin{array}{c}  
 \tikz[very thick,xscale=2,yscale=2]{
            \draw (-1,-1)-- (1,-1);
            \draw (-1,1)-- (1,1);
          \draw[dashed] (-1,-1)-- (-1,1);
           \draw[red] (-.8,-1) .. controls (-.8,-1) and (-.8,-.4) ..node[below, at start]{$r_L$} (1,0.25);
            \draw (-.48,-1) .. controls (-.48,-1) and (-.48,-0.5) ..node[below, at start]{$2$} (1,0);
            \draw (-0.16,-1) .. controls (-.16,-1) and (-.16,-0.6) ..node[below, at start]{$1$} 
          (1,-.25);
           \draw (.16,-1) .. controls (.16,-1) and (.16,-.75) ..node[below, at start]{$3$}
          (1,-.5);
           \draw (.48,-1).. controls (.48,-1) and (0,0) .. node[below, at start]{$2$}node[above, at end]{$2$}
          (0,1);
          \draw[red] (.8 ,-1)--node[below, at start]{$r_R$} (.8,1);
          \draw[red] (-1,0.25) .. controls (-1,0.25) and (-0.8,0.6) .. (-0.8,1);
          \draw (-1,0.75) .. controls (-1,0.8) and (-.16,0.9) ..node[above, at end]{$1$} (-.16,1);
          \draw (-1,0) .. controls (-1,0) and (0,0.5) ..node[above, at end]{$2$} (0,1);
          \draw (-1,-0.25)-- (1,.75);
          \draw (-1,-.5) .. controls (-1,-.5) and (0.16,.25) ..node[above, at end]{$3$} (0.16,1);
          \draw[dashed] (1,1)-- (1,-1);
          }
   \end{array}& \Zij{2}{13}{1}&=\begin{array}{c}  
 \tikz[very thick,xscale=2,yscale=2]{
            \draw (-1,-1)-- (1,-1);
            \draw (-1,1)-- (1,1);
          \draw[dashed] (-1,-1)-- (-1,1);
           \draw[red] (-.8,-1) .. controls (-.8,-1) and (-.8,-0.25) ..node[below, at start]{$r_L$} (1,0.5);
            \draw (-.48,-1) .. controls (-.48,-1) and (0,0) ..node[below, at start]{$2$}node[above, at end]{$2$} (0,1);
            \draw (-0.16,-1) .. controls (-.16,-1) and (-.16,-0.5) ..node[below, at start]{$1$} 
          (1,0);
           \draw (.16,-1)--node[below, at start]{$3$}node[above, at end]{$3$} 
          (.16,1);
           \draw (.48,-1) .. controls (.48,-1) and (.48, -.75) ..node[below, at start]{$2$}
          (1,-0.5);
          \draw[red] (.8 ,-1)--node[below, at start]{$r_R$} (.8,1);
          \draw[red] (-1,0.5) .. controls (-1,0.5) and (-0.8,0.75) .. (-0.8,1);
          \draw (-1,0.0) .. controls (-1,0.0) and (-.16,0.5) ..node[above, at end]{$1$} (-.16,1);
          \draw (-1,-0.5) .. controls (-1,-.5) and (0,0) .. (0,1);
          \draw[dashed] (1,1)-- (1,-1);
          }
   \end{array}&
\end{align*}
\begin{align*}
   &&\Zij{1}{24}{1}&=\begin{array}{c}  
 \tikz[very thick,xscale=2,yscale=2]{
            \draw (-1,-1)-- (1,-1);
            \draw (-1,1)-- (1,1);
          \draw[dashed] (-1,-1)-- (-1,1);
           \draw[red] (-.8,-1) .. controls (-.8,-1) and (-.8,-.25) ..node[below, at start]{$r_L$} (1,0.5);
            \draw (-.48,-1) .. controls (-.48,-1) and (-.48,-.5,0) ..node[below, at start]{$2$} (1,0);
            \draw (-0.16,-1)-- node[below, at start]{$1$}node[above, at end]{$1$} (-.16,1);
           \draw (.16,-1).. controls (.16,-1) and (.16,-.75).. node[below, at start]{$3$} (-1,-.5);
           \draw (.48,-1) .. controls (.48,-1) and (0,0) ..node[below, at start]{$2$}
          (0,1);
          \draw[red] (.8 ,-1)--node[below, at start]{$r_R$} (.8,1);
          \draw[red] (-1,0.5) .. controls (-1,0.5) and (-0.8,0.75) .. (-0.8,1);
          \draw (-1,0.0) .. controls (-1,0.0) and (0,0.5) ..node[above, at end]{$2$} (0,1);
          \draw (1,-0.5) .. controls (1,-.5) and (.16,.-.25) ..node[above, at end]{$3$} (.16,1);
          \draw[dashed] (1,1)-- (1,-1);
          }
   \end{array}&\Zij{2}{24}{1}&=\begin{array}{c}  
 \tikz[very thick,xscale=2,yscale=2]{
            \draw (-1,-1)-- (1,-1);
            \draw (-1,1)-- (1,1);
          \draw[dashed] (-1,-1)-- (-1,1);
           \draw[red] (-.8,-1) .. controls (-.8,-1) and (-.8,-.25) ..node[below, at start]{$r_L$} (1,0.5);
            \draw (-.48,-1) .. controls (-.48,-1) and (0,0) ..node[below, at start]{$2$}node[above, at end]{$2$} (0,1);
            \draw (-0.16,-1) .. controls (-.16,-1) and (-.16,-.5) ..node[below, at start]{$1$} 
          (1,0);
           \draw (.16,-1)--node[below, at start]{$3$}node[above, at end]{$3$} 
          (.16,1);
           \draw (.48,-1) .. controls (.48,-1) and (.48, -.75) ..node[below, at start]{$2$}
          (1,-0.5);
          \draw[red] (.8 ,-1)--node[below, at start]{$r_R$} (.8,1);
          \draw[red] (-1,0.5) .. controls (-1,0.5) and (-0.8,0.75) .. (-0.8,1);
          \draw (-1,0.0) .. controls (-1,0.0) and (-.16,0.5) ..node[above, at end]{$1$} (-.16,1);
          \draw (-1,-0.5) .. controls (-1,-.5) and (0,0) .. (0,1);
          \draw[dashed] (1,1)-- (1,-1);
          }
   \end{array}&
\end{align*}
\begin{multline*}
    \coord_{24}^2\Zij{1}{13}{1}= \\
       \begin{array}{c}  
 \tikz[very thick,xscale=2,yscale=3]{
            \draw (-1,-1)-- (1,-1);
            \draw (-1,1)-- (1,1);
          \draw[dashed] (-1,-1)-- (-1,1);
           \draw[red] (-.8,-1) .. controls (-.8,-1) and (-.8,-.6) ..node[below, at start]{$r_L$} (1,-.2);
            \draw (-.48,-1) .. controls (-.48,-1) and (-.48,-0.7) ..node[below, at start]{$2$} (1,-.4);
            \draw (-0.16,-1) .. controls (-.16,-1) and (-.16,-0.8) ..node[below, at start]{$1$} 
          (1,-.6);
           \draw (.16,-1) .. controls (.16,-1) and (.16,-.9) ..node[below, at start]{$3$}
          (1,-.8);
           \draw (.48,-1).. controls (.48,-1) and (0,0) .. node[below, at start]{$2$}node[above, at end]{$2$}
          (0,1);
          \draw[red] (.8 ,-1)--node[below, at start]{$r_R$} (.8,1);
          \draw[red] (-1,0.8) .. controls (-1,0.8) and (-0.8,0.8) .. (-0.8,1);
          \draw (-1,0) .. controls (-1,0) and (-.45,0) ..node[above, at end]{$1$} (-.45,1);
          \draw (-1,0.6) .. controls (-1,0.6) and (0,0.6) ..node[above, at end]{$2$} (0,1);
          \draw (-1,-.8) .. controls (-1,-.6) and (.9,-.4) ..node[above, at end]{$3$} (0.45,1);
          \draw[dashed] (1,1)-- (1,-1);
          \draw (-1,-.6)-- (1,0); 
          \draw (-1,-.4)-- (1,.2); 
          \draw[red] (-1,-.2)-- (1,.4); 
          \draw (-1,.2)-- (1,.6); 
          \draw[red] (-1,.4)-- (1,.8); 
          }
   \end{array}=  \begin{array}{c}  
 \tikz[very thick,xscale=2,yscale=3]{
            \draw (-1,-1)-- (1,-1);
            \draw (-1,1)-- (1,1);
          \draw[dashed] (-1,-1)-- (-1,1);
           \draw[red] (-.8,-1) .. controls (-.8,-1) and (-.8,-.6) ..node[below, at start]{$r_L$} (1,-.2);
            \draw (-.48,-1) .. controls (-.48,-1) and (-.48,-0.7) ..node[below, at start]{$2$} (1,-.4);
            \draw (-0.16,-1) .. controls (-.16,-1) and (-.16,-0.8) ..node[below, at start]{$1$} 
          (1,-.6);
           \draw (.16,-1) .. controls (.16,-1) and (.16,-.9) ..node[below, at start]{$3$}
          (1,-.8);
           \draw (.48,-1).. controls (.48,-1) and (0,0) .. node[below, at start]{$2$}node[above, at end]{$2$}
          (0,1);
          \draw[red] (.8 ,-1)--node[below, at start]{$r_R$} (.8,1);
          \draw[red] (-1,0.8) .. controls (-1,0.8) and (-0.8,0.8) .. (-0.8,1);
          \draw (-1,0) .. controls (-1,0) and (-.45,0) ..node[above, at end]{$1$} (-.45,1);
          \draw (-1,0.6) .. controls (-1,0.6) and (0,0.6) ..node[above, at end]{$2$} (0,1);
          \draw (-1,-.8) .. controls (-1,-.5) and (.4,-.4) ..node[above, at end]{$3$} (0.45,1);
          \draw[dashed] (1,1)-- (1,-1);
          \draw (-1,-.6)-- (1,0); 
          \draw (-1,-.4)-- (1,.2); 
          \draw[red] (-1,-.2)-- (1,.4); 
          \draw (-1,.2)-- (1,.6); 
          \draw[red] (-1,.4)-- (1,.8); 
          }
   \end{array}-  \begin{array}{c}  
 \tikz[very thick,xscale=2,yscale=3]{
            \draw (-1,-1)-- (1,-1);
            \draw (-1,1)-- (1,1);
          \draw[dashed] (-1,-1)-- (-1,1);
           \draw[red] (-.8,-1) .. controls (-.8,-1) and (-.8,-.6) ..node[below, at start]{$r_L$} (1,-.2);
            \draw (-.48,-1) .. controls (-.48,-1) and (-.48,-0.7) ..node[below, at start]{$2$} (1,-.4);
            \draw (-0.16,-1) .. controls (-.16,-1) and (-.16,-0.8) ..node[below, at start]{$1$} 
          (1,-.6);
           \draw (.16,-1) .. controls (.16,-1) and (.16,-.9) ..node[below, at start]{$3$}
          (1,-.8);
           \draw (.48,-1).. controls (.48,-1) and (.48,-.6) .. node[below, at start]{$2$} (1,.2);
          \draw[red] (.8 ,-1)--node[below, at start]{$r_R$} (.8,1);
          \draw[red] (-1,0.8) .. controls (-1,0.8) and (-0.8,0.8) .. (-0.8,1);
          \draw (-1,0) .. controls (-1,0) and (-.45,0) ..node[above, at end]{$1$} (-.45,1);
          \draw (-1,0.6) .. controls (-1,0.6) and (0,0.6) ..node[above, at end]{$2$} (0,1);
          \draw (-1,-.4) .. controls (-1,-.4) and (0,-.4) ..node[above, at end]{$2$} (0,1);
          \draw (-1,-.8) .. controls (-1,-.6) and (.9,-.4) ..node[above, at end]{$3$} (0.45,1);
          \draw[dashed] (1,1)-- (1,-1);
          \draw (-1,-.6)-- (1,0); 
          \draw[red] (-1,-.2)-- (1,.4); 
          \draw (-1,.2)-- (1,.6); 
          \draw[red] (-1,.4)-- (1,.8); 
          }
   \end{array}\\
   =\coord_{13}^2\Zij{1}{24}{1}+(\bullet_1-\bullet_2+\bullet_3)\coord_{24}\coord_{13}\Zij{2}{24}{1}.
\end{multline*}
\Aiden{for the final version of the paper I'll type up an extra step for this calculation showing where the dots appear. I won't worry about this for the thesis version though as it's a little fiddly.}
Several important identities were used in the above calculation. The most important is the relation \eqref{triple-dumb} for $ijk=232$. Secondly, the bigon relations \eqref{black-bigon} and \eqref{cost} were applied to arrive at the final result.

Upon inspection, we immediately see that $\Zij{2}{13}{1}=\Zij{2}{24}{1}$ which implies that the vector bundle contains a line subbundle. Expressing these equations using generators with winding degree zero yields
\begin{align*}
    &&\coord_{24}\Zij{1}{13}{0}&=\coord_{13}\Zij{1}{24}{0}+(\bullet_1-\bullet_2+\bullet_3)\coord_{24}\Zij{2}{24}{0},& \coord_{13}\Zij{2}{13}{0}&=\coord_{24}\Zij{2}{24}{0}.&
\end{align*}
From these we can deduce the transition matrix is
\begin{align*}
    &&\gamma^{24}_{13}[\textcolor{red}{2}2132\textcolor{red}{2}]&=\left(\begin{array}{cc}
        \coord_{13}/\coord_{24} & \bullet_1-\bullet_2+\bullet_3  \\
         0 & \coord_{24}/\coord_{13}
    \end{array}\right)&\\
    &\implies& \gamma_{24}^{13}[\textcolor{red}{2}2132\textcolor{red}{2}]&=\left(\begin{array}{cc}
      \coord_{24}/\coord_{13}   & -\bullet_1+\bullet_2-\bullet_3 \\
        0 & \coord_{13}/\coord_{24} 
    \end{array}\right).&
\end{align*}
\begin{align*}
    &&\coord_{13}\Zij{1}{13}{0}&=\coord_{24}\Zij{1}{24}{0},&\coord_{24}\Zij{2}{13}{0}&=(\bullet_1-\bullet_2+\bullet_3)\coord_{24}\Zij{1}{24}{0}+\coord_{13}\Zij{2}{24}{0}.&
\end{align*}
Lemma \ref{indecomp-ext} then tells us that
\begin{lemma}\label{lem:331}
    $\tilting^{\color{red}2\color{black}2132\color{red}2}\cong \mathcal{H}\otimes\mathcal{L}.$
\end{lemma}

\subsubsection{Sheaf with idempotent \textcolor{red}{22}2312}
Applying algorithm \ref{algorithm13} and its variant to $\Bi=$\textcolor{red}{22}2312 yields the following set of generators:
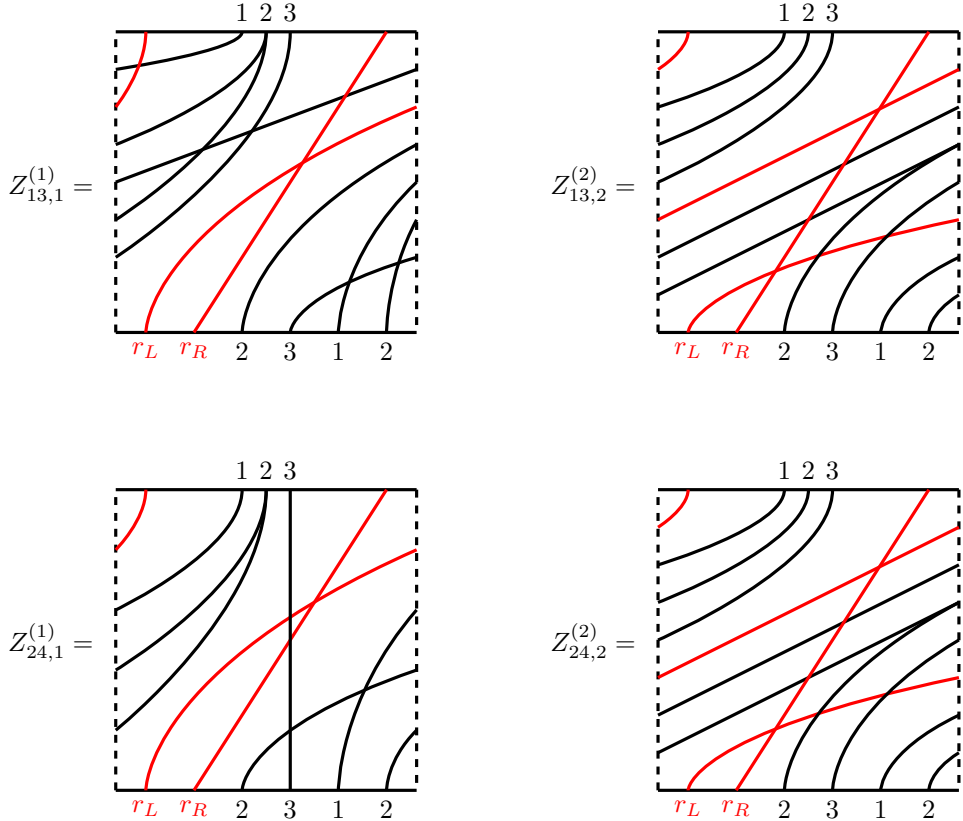
\begin{align*}
&&\Zij{1}{13}{1}&=\begin{array}{c}  
 \tikz[very thick,xscale=2,yscale=2]{
            \draw (-1,-1)-- (1,-1);
            \draw (-1,1)-- (1,1);
          \draw[dashed] (-1,-1)-- (-1,1);
          \draw[dashed] (1,1)-- (1,-1);
           \draw (-1,0.75).. controls (-1,0.75) and (-.16,0.875).. node[above, at end]{$1$} (-.16,1);
           \draw[red] (-1,0.5).. controls (-1,0.5) and (-0.8,0.75).. (-0.8,1);
           \draw (-1,0.25).. controls (-1,0.25) and (0,0.625).. node[above, at end]{$2$} (0,1);
           \draw (-1,0)-- (1,0.75);
           \draw (-1,-0.25).. controls (-1,-0.25) and (0,0.375).. (0,1);
           \draw (-1,-.5).. controls (-1,-.5) and (.16,.25).. node[above, at end]{$3$} (.16,1);
           \draw[red] (-.8,-1).. controls (-.8,-1) and (-.8,-.25).. node[below, at start]{$r_L$} (1,0.5);
           \draw[red] (-.48,-1)-- node[below, at start]{$r_R$} (.8,1);
           \draw (-.16,-1).. controls (-.16,-1) and (-.16,-.375).. node[below, at start]{$2$} (1,0.25);
           \draw (.16,-1).. controls (.16,-1) and (.16,-0.75).. node[below, at start]{$3$} (1,-.5);
           \draw (.48,-1).. controls (.48,-1) and (.48,-.5).. node[below, at start]{$1$} (1,0);
           \draw (.8,-1).. controls (.8,-1) and (.8, -.625).. node[below, at start]{$2$} (1,-.25);
          }
   \end{array}&
   &&\Zij{2}{13}{2}&=\begin{array}{c}  
 \tikz[very thick,xscale=2,yscale=2]{
            \draw (-1,-1)-- (1,-1);
            \draw (-1,1)-- (1,1);
          \draw[dashed] (-1,-1)-- (-1,1);
          \draw[dashed] (1,1)-- (1,-1);
          \draw[red] (-1,.75).. controls (-1,.75) and (-0.8,.875).. (-0.8,1);
          \draw (-1,.5).. controls (-1,.5) and (-.16,.75).. node[above, at end]{$1$} (-.16,1);
          \draw (-1,.25).. controls (-1,.25) and (0,.625).. node[above, at end]{$2$} (0,1);
          \draw (-1,0).. controls (-1,0) and (.16,0.5).. node[above, at end]{$3$} (0.16,1);
          \draw[red] (-1,-0.25)-- (1,0.75);
          \draw (-1,-0.5)-- (1,0.5);
          \draw (-1,-0.75)-- (1,0.25);
          \draw[red] (-0.8,-1).. controls (-0.8,-1) and (-0.8,-.625).. node[below, at start]{$r_L$} (1,-0.25);
          \draw[red] (-.48,-1)-- node[below, at start]{$r_R$} (0.8,1);
          \draw (-.16,-1).. controls (-.16,-1) and (-.16,-.375).. node[below, at start]{$2$} (1,0.25);
          \draw (0.16,-1).. controls (0.16,-1) and (0.16,-0.5).. node[below, at start]{$3$} (1,0);
          \draw (.48,-1).. controls (.48,-1) and (.48,-.75).. node[below, at start]{$1$} (1,-.5);
          \draw (.8,-1).. controls (.8,-1) and (.8,-.875).. node[below, at start]{$2$} (1,-.75);
          }
   \end{array}&
\end{align*}
\begin{align*}
&&\Zij{1}{24}{1}&=\begin{array}{c}  
 \tikz[very thick,xscale=2,yscale=2]{
            \draw (-1,-1)-- (1,-1);
            \draw (-1,1)-- (1,1);
          \draw[dashed] (-1,-1)-- (-1,1);
          \draw[dashed] (1,1)-- (1,-1);
          \draw[red] (-1,.6).. controls (-1,.6) and (-.8,.8).. (-.8,1);
          \draw (-1,.2).. controls (-1,.2) and (-.16,.6).. node[above, at end]{$1$} (-.16,1);
          \draw (-1,-.2).. controls (-1,-.2) and (0,.4).. node[above, at end]{$2$} (0,1);
          \draw (-1,-.6).. controls (-1,-.6) and (0, .2).. (0,1);
          \draw[red] (-.8,-1).. controls (-.8,-1) and (-.8, -.2).. node[below, at start]{$r_L$} (1,.6);
          \draw[red] (-.48, -1)-- node[below, at start]{$r_R$} (.8,1);
          \draw (-.16,-1).. controls (-.16,-1) and (-.16, -.6).. node[below, at start]{$2$} (1,-.2);
          \draw (.16,-1)-- node[below, at start]{$3$}node[above, at end]{$3$} (.16,1);
          \draw (.48,-1).. controls (.48,-1) and (.48, -.4).. node[below, at start]{$1$} (1,.2);
          \draw (.8,-1).. controls (.8,-1) and (.8,-.8).. node[below, at start]{$2$} (1,-.6);
          }
   \end{array}&
   &&\Zij{2}{24}{2}&=\begin{array}{c}  
 \tikz[very thick,xscale=2,yscale=2]{
            \draw (-1,-1)-- (1,-1);
            \draw (-1,1)-- (1,1);
          \draw[dashed] (-1,-1)-- (-1,1);
          \draw[dashed] (1,1)-- (1,-1);
          \draw[red] (-1,.75).. controls (-1,.75) and (-.8,.875).. (-.8,1);
          \draw (-1,.5).. controls (-1,.5) and (-.16,.75).. node[above, at end]{$1$} (-.16,1);
          \draw (-1,.25).. controls (-1,.25) and (0,.625).. node[above, at end]{$2$} (0,1);
          \draw (-1,0).. controls (-1,0) and (.16, .5).. node[above, at end]{$3$} (.16,1);
          \draw[red] (-1,-.25)-- (1,.75);
          \draw (-1,-.5)-- (1,.5);
          \draw (-1,-.75)-- (1,.25);
          \draw[red] (-.8,-1).. controls (-.8,-1) and (-.8, -.625).. node[below, at start]{$r_L$} (1,-.25);
          \draw[red] (-.48,-1)-- node[below, at start]{$r_R$} (.8,1);
          \draw (-.16,-1).. controls (-.16,-1) and (-.16, -.375).. node[below, at start]{$2$} (1,.25);
          \draw (.16,-1).. controls (.16,-1) and (.16,-.5).. node[below, at start]{$3$} (1,0);
          \draw (.48,-1).. controls (.48,-1) and (.48,-.75).. node[below, at start]{$1$} (1,-.5);
          \draw (.8,-1).. controls (.8,-1) and (.8,-.875).. node[below, at start]{$2$} (1,-.75);
          }
   \end{array}&
\end{align*}
To compute the transition matrix, we proceed similarly to before:
\begin{multline*}
    \coord_{24}\Zij{1}{13}{1}=\\
\begin{array}{c}  
 \tikz[very thick,xscale=2,yscale=3]{
            \draw (-1,-1)-- (1,-1);
            \draw (-1,1)-- (1,1);
          \draw[dashed] (-1,-1)-- (-1,1);
          \draw[dashed] (1,1)-- (1,-1);
           \draw (-1,0.33).. controls (-1,0.33) and (-.4,0.67).. node[above, at end]{$1$} (-.4,1);
           \draw[red] (-1,0.78).. controls (-1,0.78) and (-0.8,0.89).. (-0.8,1);
           \draw (-1,-.7).. controls (-1,-.7) and (0,0.15).. node[above, at end]{$2$} (0,1);
           \draw (-1,.56).. controls (-1,.56) and (0,0.78).. (0,1); 
           \draw (-1,-.56).. controls (-1,-.56) and (.4,.22).. node[above, at end]{$3$} (.4,1);
           \draw[red] (-.8,-1).. controls (-.8,-1) and (-.8,-.44).. node[below, at start]{$r_L$} (1,0.11);
           \draw[red] (-.48,-1)-- node[below, at start]{$r_R$} (.8,1);
           \draw (-.16,-1).. controls (-.16,-1) and (-.16,-.56).. node[below, at start]{$2$} (1,-.11);
           \draw (.16,-1).. controls (.16,-1) and (.16,-0.78).. node[below, at start]{$3$} (1,-.56);
           \draw (.48,-1).. controls (.48,-1) and (.48,-.67).. node[below, at start]{$1$} (1,-.33);
           \draw (.8,-1).. controls (.8,-1) and (.8, -.89).. node[below, at start]{$2$} (1,-.7);
            \draw (-1,-.33)-- (1,0.33); 
            \draw (-1,-.11)-- (1,0.56); 
            \draw[red] (-1,.11)-- (1,0.78); 
          }
   \end{array}=\begin{array}{c}  
 \tikz[very thick,xscale=2,yscale=3]{
            \draw (-1,-1)-- (1,-1);
            \draw (-1,1)-- (1,1);
          \draw[dashed] (-1,-1)-- (-1,1);
          \draw[dashed] (1,1)-- (1,-1);
           \draw (-1,0.33).. controls (-1,0.33) and (-.4,0.67).. node[above, at end]{$1$} (-.4,1);
           \draw[red] (-1,0.78).. controls (-1,0.78) and (-0.8,0.89).. (-0.8,1);
           \draw (-1,-.7).. controls (-1,-.7) and (0,0.15).. node[above, at end]{$2$} (0,1);
           \draw (-1,.56).. controls (-1,.56) and (0,0.78).. (0,1); 
           \draw (-1,-.56).. controls (-1,.22) and (.4,.22).. node[above, at end]{$3$} (.4,1);
           \draw[red] (-.8,-1).. controls (-.8,-1) and (-.8,-.44).. node[below, at start]{$r_L$} (1,0.11);
           \draw[red] (-.48,-1)-- node[below, at start]{$r_R$} (.8,1);
           \draw (-.16,-1).. controls (-.16,-1) and (-.16,-.56).. node[below, at start]{$2$} (1,-.11);
           \draw (.16,-1).. controls (.16,-1) and (.16,-0.78).. node[below, at start]{$3$} (1,-.56);
           \draw (.48,-1).. controls (.48,-1) and (.48,-.67).. node[below, at start]{$1$} (1,-.33);
           \draw (.8,-1).. controls (.8,-1) and (.8, -.89).. node[below, at start]{$2$} (1,-.7);
            \draw (-1,-.33)-- (1,0.33); 
            \draw (-1,-.11)-- (1,0.56); 
            \draw[red] (-1,.11)-- (1,0.78); 
          }
   \end{array}-\begin{array}{c}  
 \tikz[very thick,xscale=2,yscale=3]{
            \draw (-1,-1)-- (1,-1);
            \draw (-1,1)-- (1,1);
          \draw[dashed] (-1,-1)-- (-1,1);
          \draw[dashed] (1,1)-- (1,-1);
           \draw (-1,0.33).. controls (-1,0.33) and (-.4,0.67).. node[above, at end]{$1$} (-.4,1);
           \draw[red] (-1,0.78).. controls (-1,0.78) and (-0.8,0.89).. (-0.8,1);
           \draw (-1,-.11).. controls (-1,-.11) and (0,0.445).. node[above, at end]{$2$} (0,1);
           \draw (-1,.56).. controls (-1,.56) and (0,0.78).. (0,1); 
           \draw (-1,-.56).. controls (-1,-.56) and (.4,.22).. node[above, at end]{$3$} (.4,1);
           \draw[red] (-.8,-1).. controls (-.8,-1) and (-.8,-.44).. node[below, at start]{$r_L$} (1,0.11);
           \draw[red] (-.48,-1)-- node[below, at start]{$r_R$} (.8,1);
           \draw (-.16,-1).. controls (-.16,-1) and (-.16,-.56).. node[below, at start]{$2$} (1,-.11);
           \draw (.16,-1).. controls (.16,-1) and (.16,-0.78).. node[below, at start]{$3$} (1,-.56);
           \draw (.48,-1).. controls (.48,-1) and (.48,-.67).. node[below, at start]{$1$} (1,-.33);
           \draw (.8,-1).. controls (.8,-1) and (.8, -.89).. node[below, at start]{$2$} (1,-.7);
            \draw (-1,-.33)-- (1,0.33); 
            \draw (-1,-.7).. controls(-1,-.7) and (0,-.25).. (1,.56);
            \draw[red] (-1,.11)-- (1,0.78); 
          }
   \end{array}\\
   =\coord_{13}\Zij{1}{24}{1}+(\bullet_1-\bullet_2+\bullet_3)\Zij{2}{24}{2}.
\end{multline*}
Similarly to the previous case, in the above we have applied the relation \eqref{triple-dumb} with $ijk=232$ as well as the bigon relations \eqref{black-bigon} and \eqref{cost}.

Also similar to the previous case, we note that $\Zij{2}{13}{2}=\Zij{2}{24}{2}$ and so we again have a line subbundle. Rewriting these equations using a basis comprised of winding degree 0 generators yields:
\begin{align*}
    &&\coord_{13}\Zij{1}{13}{0}&=\coord_{13}\Zij{1}{24}{0}+(\bullet_1-\bullet_2+\bullet_3)\coord_{24}\Zij{2}{24}{0},&\coord_{13}^2\Zij{2}{13}{0}&=\coord_{24}^2\Zij{2}{24}{0}.&
\end{align*}
The transition matrix of this vector bundle on these patches is hence given by:
\begin{align*}
    &&\gamma_{13}^{24}[\textcolor{red}{22}2312]&=\left(\begin{array}{cc}
        1 & (\bullet_1-\bullet_2+\bullet_3)\coord_{24}/\coord_{13}  \\
        0 & \coord_{24}^2/\coord_{13}^2
    \end{array}\right)&\\
    &\implies &\gamma_{24}^{13}[\textcolor{red}{22}2312]&=\left(\begin{array}{cc}
        1 & -(\bullet_1-\bullet_2+\bullet_3)\coord_{13}/\coord_{24}  \\
        0 & \coord_{13}^2/\coord_{24}^2
    \end{array}\right).&
\end{align*}
Lemma \ref{indecomp-ext} then tells us that
\begin{lemma}\label{lem:332}
    $\tilting^{\color{red}22\color{black}2312}\cong \mathcal{H}$.
\end{lemma}
\subsubsection{Sheaf with idempotent \textcolor{red}{2}23\textcolor{red}{2}21}
Applying algorithm \ref{algorithm13} and its variant to $\Bi=$\textcolor{red}{2}23\textcolor{red}{2}21 yields the following set of generators:
\begin{align*}
&&\Zij{1}{13}{1}&=\begin{array}{c}  
 \tikz[very thick,xscale=2,yscale=2]{
            \draw (-1,-1)-- (1,-1);
            \draw (-1,1)-- (1,1);
            \draw[dashed] (1,1)-- (1,-1);
          \draw[dashed] (-1,-1)-- (-1,1);
           \draw[red] (-.8,-1).. controls (-.8,-1) and (-.8,-.25).. node[below, at start]{$r_L$} (1,.5) ;
            \draw (-.48,-1) .. controls (-.48,-1) and (-.48,-0.375) ..node[below, at start]{$2$} (1,0.25);
            \draw (-0.16,-1) .. controls (-.16,-1) and (-.16,-.625) ..node[below, at start]{$3$} (1,-.25);
           \draw[red] (.48,-1)--node[below, at start]{$r_R$} 
          (.48,1);
            \draw (.64 ,-1).. controls (.64,-1) and (.64,-.625).. node[below, at start]{$2$} (1,.25);
          \draw (.8 ,-1).. controls (.8,-1) and (.8,-.5).. node[below, at start]{$1$} (1,0);
          \draw (-1,0.75).. controls (-1,0.75) and (-.16,.75).. node[above, at end]{$1$} (-.16,1);
          \draw[red] (-1,.5).. controls (-1,.5) and (-.8,.5).. (-.8,1);
          \draw (-1,.25).. controls (-1,.25) and (0,.25).. node[above, at end]{$2$} (0,1);
          \draw (-1,0)-- (1,.75);
          \draw (-1,-.25).. controls (-1,-.25) and (.16,-.25).. node[above, at end]{$3$} (.16,1);
          }
   \end{array}&
    &&\Zij{2}{13}{1}&=\begin{array}{c}  
 \tikz[very thick,xscale=2,yscale=2]{
            \draw (-1,-1)-- (1,-1);
            \draw (-1,1)-- (1,1);
            \draw[dashed] (1,1)-- (1,-1);
          \draw[dashed] (-1,-1)-- (-1,1);
           \draw[red] (-.8,-1).. controls (-.8,-1) and (-.8,-.4).. node[below, at start]{$r_L$} (1,.2);
            \draw (-.48,-1) .. controls (-.48,-1) and (0,-1) ..node[below, at start]{$2$}node[above, at end]{$2$} (0,1);
            \draw (-0.16,-1) .. controls (-.16,-1) and (0,-1) ..node[below, at start]{$3$}node[above, at end]{$3$} (.16,1);
           \draw[red] (.48,-1)--node[below, at start]{$r_R$} 
          (.48,1);
            \draw (.64 ,-1).. controls (.64,-1) and (.64,-.6).. node[below, at start]{$2$} (1,-.2);
          \draw (.8 ,-1).. controls (.8,-1) and (.8,-.8).. node[below, at start]{$1$} (1,-.6);
          \draw (-1,0.6).. controls (-1,0.6) and (-.16,.6).. node[above, at end]{$1$} (-.16,1);
          \draw[red] (-1,.2).. controls (-1,.2) and (-.8,.2).. (-.8,1);
          \draw (-1,-.2).. controls (-1,-.2) and (0,-.2).. node[above, at end]{$2$} (0,1);
          \draw (-1,-.6)-- (1,.6);
          }
   \end{array}&
\end{align*}

\begin{align*}
     &&\Zij{1}{24}{1}&=\begin{array}{c}  
 \tikz[very thick,xscale=2,yscale=2]{
            \draw (-1,-1)-- (1,-1);
            \draw (-1,1)-- (1,1);
            \draw[dashed] (1,1)-- (1,-1);
          \draw[dashed] (-1,-1)-- (-1,1);
           \draw[red] (-.8,-1).. controls (-.8,-1) and (-.8,-.33).. node[below, at start]{$r_L$} (1,.33) ;
            \draw (-.48,-1) .. controls (-.48,-1) and (-.48,-0.5) ..node[below, at start]{$2$} (1,0);
            \draw (-0.16,-1) .. controls (-.16,-1) and (0,-1) ..node[below, at start]{$3$}node[above, at end]{$3$} (.16,1);
           \draw[red] (.48,-1)--node[below, at start]{$r_R$} 
          (.48,1);
            \draw (.64 ,-1).. controls (.64,-1) and (.64,-.5).. node[below, at start]{$2$} (1,0);
          \draw (.8 ,-1).. controls (.8,-1) and (.8,-.66).. node[below, at start]{$1$} (1,-.33);
          \draw (-1,0.66).. controls (-1,0.66) and (-.16,.66).. node[above, at end]{$1$} (-.16,1);
          \draw[red] (-1,.33).. controls (-1,.33) and (-.8,.33).. (-.8,1);
          \draw (-1,0).. controls (-1,0) and (0,0).. node[above, at end]{$2$} (0,1);
          \draw (-1,-.33)-- (1,.66);
          }
   \end{array}&
   &&\Zij{2}{24}{1}&=\begin{array}{c}  
 \tikz[very thick,xscale=2,yscale=2]{
            \draw (-1,-1)-- (1,-1);
            \draw (-1,1)-- (1,1);
            \draw[dashed] (1,1)-- (1,-1);
          \draw[dashed] (-1,-1)-- (-1,1);
           \draw[red] (-.8,-1).. controls (-.8,-1) and (-.8,-.25).. node[below, at start]{$r_L$} (1,.5);
            \draw (-.48,-1) .. controls (-.48,-1) and (0,-1) ..node[below, at start]{$2$}node[above, at end]{$2$} (0,1);
            \draw (-0.16,-1) .. controls (-.16,-1) and (0,-1) ..node[below, at start]{$3$}node[above, at end]{$3$} (.16,1);
           \draw[red] (.48,-1)--node[below, at start]{$r_R$} 
          (.48,1);
            \draw (.64 ,-1).. controls (.64,-1) and (.64,-.75).. node[below, at start]{$2$} (1,-.5);
          \draw (.8 ,-1).. controls (.8,-1) and (.8,-.5).. node[below, at start]{$1$} (1,0);
          \draw[red] (-1,.5).. controls (-1,.5) and (-.8,.5).. (-.8,1);
         \draw (-1,0).. controls (-1,0) and (-.16,0).. node[above, at end]{$1$} (-.16,1);
          \draw (-1,-.5).. controls (-1,-.5) and (0,-.5).. node[above, at end]{$2$} (0,1);
          }
   \end{array}&
\end{align*}
To compute the transition matrix $\gamma_{24}^{13}$[\textcolor{red}{2}23\textcolor{red}{2}21], we perform the following manipulations:
\begin{multline*}
    \coord_{24}\Zij{1}{13}{1}=\\
    \begin{array}{c}  
 \tikz[very thick,xscale=2,yscale=2]{
            \draw (-1,-1)-- (1,-1);
            \draw (-1,1)-- (1,1);
            \draw[dashed] (1,1)-- (1,-1);
          \draw[dashed] (-1,-1)-- (-1,1);
           \draw[red] (-.8,-1).. controls (-.8,-1) and (-.8,-.5).. node[below, at start]{$r_L$} (1,0) ;
            \draw (-.48,-1) .. controls (-.48,-1) and (-.48,-0.625) ..node[below, at start]{$2$} (1,-0.25);
            \draw (-0.16,-1) .. controls (-.16,-1) and (-.16,-.875) ..node[below, at start]{$3$} (1,-.75);
           \draw[red] (.48,-1)--node[below, at start]{$r_R$} (.48,1);
            \draw (.64 ,-1).. controls (.64,-1) and (.64,-.625).. node[below, at start]{$2$} (1,-.25);
          \draw (.8 ,-1).. controls (.8,-1) and (.8,-.75).. node[below, at start]{$1$} (1,-.5);
          \draw (-1,0.25).. controls (-1,0.25) and (-.16,.25).. node[above, at end]{$1$} (-.16,1);
          \draw[red] (-1,.75).. controls (-1,.75) and (-.8,.75).. (-.8,1);
          \draw (-1,-.25).. controls (-1,-.25) and (0,0).. node[above, at end]{$2$} (0,1);
          \draw (-1,-.5)-- (1,.25); 
          \draw (-1,-.25)-- (1,.5); 
          \draw[red] (-1,0)-- (1,.75); 
          \draw (-1,-.75).. controls (-1,-.75) and (.1,-.6).. node[above, at end]{$3$} (.16,1);
           \draw (-1,.5).. controls (-1,.5) and (0,.5).. (0,1); 
          }
   \end{array}=\begin{array}{c}  
 \tikz[very thick,xscale=2,yscale=2]{
            \draw (-1,-1)-- (1,-1);
            \draw (-1,1)-- (1,1);
            \draw[dashed] (1,1)-- (1,-1);
          \draw[dashed] (-1,-1)-- (-1,1);
           \draw[red] (-.8,-1).. controls (-.8,-1) and (-.8,-.44).. node[below, at start]{$r_L$} (1,0.11) ;
            \draw (-.48,-1) .. controls (-.48,-1) and (-.48,-0.56) ..node[below, at start]{$2$} (1,-0.11);
            \draw (-0.16,-1) .. controls (-.16,-1) and (-.16,-.67) ..node[below, at start]{$3$} (1,-.33);
           \draw[red] (.48,-1)--node[below, at start]{$r_R$} (.48,1);
            \draw (.64 ,-1).. controls (.64,-1) and (.64,-.78).. node[below, at start]{$2$} (1,-.56);
          \draw (.8 ,-1).. controls (.8,-1) and (.8,-.89).. node[below, at start]{$1$} (1,-.7);
          \draw (-1,0.35).. controls (-1,0.35) and (-.16,.35).. node[above, at end]{$1$} (-.16,1);
          \draw[red] (-1,.77).. controls (-1,.77) and (-.8,.77).. (-.8,1);
          \draw (-1,-.55).. controls (-1,-.55) and (-.25,-.55).. node[above, at end]{$2$} (0,1);
          \draw (-1,-.7)-- (1,.35); 
          \draw (-1,-.11)-- (1,.55); 
          \draw[red] (-1,.11)-- (1,.77); 
          \draw (-1,-.33).. controls (-1,-.33) and (-.35,0).. node[above, at end]{$3$} (.16,1);
         \draw (-1,.5).. controls (-1,.5) and (-.5,.5).. (0,1); 
          }
   \end{array}-\begin{array}{c}  
 \tikz[very thick,xscale=2,yscale=2]{
            \draw (-1,-1)-- (1,-1);
            \draw (-1,1)-- (1,1);
            \draw[dashed] (1,1)-- (1,-1);
          \draw[dashed] (-1,-1)-- (-1,1);
           \draw[red] (-.8,-1).. controls (-.8,-1) and (-.8,-.44).. node[below, at start]{$r_L$} (1,0.11) ;
            \draw (-.48,-1) .. controls (-.48,-1) and (-.48,-0.56) ..node[below, at start]{$2$} (1,-0.11);
            \draw (-0.16,-1) .. controls (-.16,-1) and (-.16,-.67) ..node[below, at start]{$3$} (1,-.33);
           \draw[red] (.48,-1)--node[below, at start]{$r_R$} (.48,1);
            \draw (.64 ,-1).. controls (.64,-1) and (.64,-.78).. node[below, at start]{$2$} (1,-.56);
          \draw (.8 ,-1).. controls (.8,-1) and (.8,-.89).. node[below, at start]{$1$} (1,-.7);
          \draw (-1,0.35).. controls (-1,0.35) and (-.16,.35).. node[above, at end]{$1$} (-.16,1);
          \draw[red] (-1,.77).. controls (-1,.77) and (-.8,.77).. (-.8,1);
          \draw (-1,-.11).. controls (-1,-.11) and (-.5,-.11).. node[above, at end]{$2$} (0,1);
          \draw (-1,-.7)-- (1,.35); 
          \draw (-1,-.55)-- (1,.55); 
          \draw[red] (-1,.11)-- (1,.77); 
          \draw (-1,-.33).. controls (-1,-.25) and (-.42,-.25).. node[above, at end]{$3$} (.16,1);
           \draw (-1,.5).. controls (-1,.5) and (-.5,.5).. (0,1); 
          }
   \end{array}\\
   =\coord_{23}\Zij{1}{24}{1}-\coord_{12}\Zij{2}{24}{1}.
\end{multline*}
The most important identity used in this calculation is relation \eqref{triple-dumb} with $ijk=232$. Similarly, we have
\begin{multline*}
    \coord_{24}\Zij{2}{13}{1}=\\
    \begin{array}{c}  
 \tikz[very thick,xscale=2,yscale=2]{
            \draw (-1,-1)-- (1,-1);
            \draw (-1,1)-- (1,1);
            \draw[dashed] (1,1)-- (1,-1);
          \draw[dashed] (-1,-1)-- (-1,1);
           \draw[red] (-.8,-1).. controls (-.8,-1) and (-.8,-.57).. node[below, at start]{$r_L$} (1,-.14);
            \draw (-.48,-1) .. controls (-.48,-1) and (0,-1) ..node[below, at start]{$2$}node[above, at end]{$2$} (0,1);
            \draw (-0.16,-1) .. controls (-.16,-1) and (0,-1) ..node[below, at start]{$3$}node[above, at end]{$3$} (.16,1);
           \draw[red] (.48,-1)--node[below, at start]{$r_R$} 
          (.48,1);
            \draw (.64 ,-1).. controls (.64,-1) and (.64,-.71).. node[below, at start]{$2$} (1,-.43);
          \draw (.8 ,-1).. controls (.8,-1) and (.8,-.86).. node[below, at start]{$1$} (1,-.71);
          \draw (-1,0.14).. controls (-1,0.14) and (-.16,.14).. node[above, at end]{$1$} (-.16,1);
          \draw[red] (-1,.71).. controls (-1,.71) and (-.8,.71).. (-.8,1);
          \draw (-1,.43).. controls (-1,.43) and (0,.43).. node[above, at end]{$2$} (0,1);
          \draw (-1,-.71)-- (1,.14); 
          \draw (-1, -.43)-- (1,.43); 
          \draw[red] (-1, -.14)-- (1,.71); 
          }
   \end{array}=\begin{array}{c}  
 \tikz[very thick,xscale=2,yscale=2]{
            \draw (-1,-1)-- (1,-1);
            \draw (-1,1)-- (1,1);
            \draw[dashed] (1,1)-- (1,-1);
          \draw[dashed] (-1,-1)-- (-1,1);
           \draw[red] (-.8,-1).. controls (-.8,-1) and (-.8,-.57).. node[below, at start]{$r_L$} (1,-.14);
            \draw (-.48,-1) .. controls (-.48,-1) and (0,-1) ..node[below, at start]{$2$}node[above, at end]{$2$} (0,1);
            \draw (-0.16,-1) .. controls (-.16,-1) and (0,-1) ..node[below, at start]{$3$}node[above, at end]{$3$} (.16,1);
           \draw[red] (.48,-1)--node[below, at start]{$r_R$} 
          (.48,1);
            \draw (.64 ,-1).. controls (.64,-1) and (.64,-.86).. node[below, at start]{$2$} (1,-.71);
          \draw (.8 ,-1).. controls (.8,-1) and (.8,-.71).. node[below, at start]{$1$} (1,-.43);
          \draw (-1,0.43).. controls (-1,0.43) and (-.16,.43).. node[above, at end]{$1$} (-.16,1);
          \draw[red] (-1,.71).. controls (-1,.71) and (-.8,.71).. (-.8,1);
          \draw (-1,.14).. controls (-1,.14) and (0,.14).. node[above, at end]{$2$} (0,1);
          \draw (-1,-.71)-- (1,.14); 
          \draw (-1, -.43)-- (1,.43); 
          \draw[red] (-1, -.14)-- (1,.71); 
          }
   \end{array}+\begin{array}{c}  
 \tikz[very thick,xscale=2,yscale=2]{
            \draw (-1,-1)-- (1,-1);
            \draw (-1,1)-- (1,1);
            \draw[dashed] (1,1)-- (1,-1);
          \draw[dashed] (-1,-1)-- (-1,1);
           \draw[red] (-.8,-1).. controls (-.8,-1) and (-.8,-.57).. node[below, at start]{$r_L$} (1,-.14);
            \draw (-1,-.43) .. controls (-1,-.43) and (0,-.43).. node[above, at end]{$2$} (0,1);
             \draw (-.48,-1) .. controls (-.48,-1) and (-.48,-.2) ..node[below, at start]{$2$} (1,.14);
            \draw (-0.16,-1) .. controls (-.16,-1) and (0,-1) ..node[below, at start]{$3$}node[above, at end]{$3$} (.16,1);
           \draw[red] (.48,-1)--node[below, at start]{$r_R$} 
          (.48,1);
            \draw (.64 ,-1).. controls (.64,-1) and (.64,-.71).. node[below, at start]{$2$} (1,-.43);
          \draw (.8 ,-1).. controls (.8,-1) and (.8,-.86).. node[below, at start]{$1$} (1,-.71);
          \draw (-1,0.43).. controls (-1,0.43) and (-.16,.43).. node[above, at end]{$1$} (-.16,1);
          \draw[red] (-1,.71).. controls (-1,.71) and (-.8,.71).. (-.8,1);
          \draw (-1,.14).. controls (-1,.14) and (0,.14).. node[above, at end]{$2$} (0,1);
          \draw (-1, -.71)-- (1,.43); 
          \draw[red] (-1, -.14)-- (1,.71); 
          }
   \end{array}\\
   =\coord_{34}\Zij{1}{24}{1}+\coord_{14}\Zij{2}{24}{1}.
\end{multline*}
where the most notable step was applying relation \eqref{triple-dumb} with $ijk=212$. Together, these two equations produce the transition matrix:
\begin{align*}
    &&\gamma_{13}^{24}[\textcolor{red}{2}23\textcolor{red}{2}21]&=\left(\begin{array}{cc}
        \coord_{23}/\coord_{13} & -\coord_{12}/\coord_{13}  \\
        \coord_{34}/\coord_{13} & \coord_{14}/\coord_{13} 
    \end{array}\right).&
\end{align*}
Lemma \ref{transition-lemma-2} then tells us that
\begin{lemma}\label{lem:333}
$\tilting^{\color{red}2\color{black}23\color{red}2\color{black}21}\cong\taut^\perp.$
\end{lemma}
\subsubsection{Sheaf with idempotent \textcolor{red}{2}21\textcolor{red}{2}23}
Applying the algorithms to the case where  $\Bi$=\textcolor{red}{2}21\textcolor{red}{2}23 produces the following set of generators:
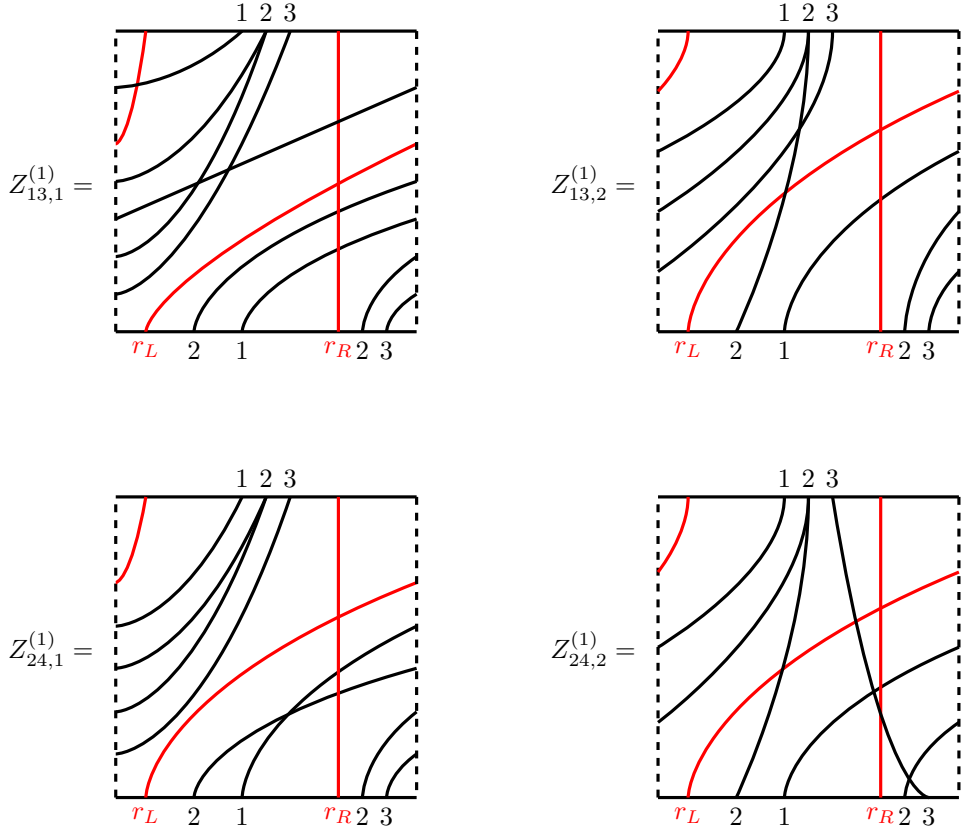
\begin{align*}
&&\Zij{1}{13}{1}&=\begin{array}{c}  
 \tikz[very thick,xscale=2,yscale=2]{
            \draw (-1,-1)-- (1,-1);
            \draw (-1,1)-- (1,1);
          \draw[dashed] (-1,-1)-- (-1,1);
           \draw[red] (-.8,-1) .. controls (-.8,-1) and (-.8,-.625) .. node[below, at start]{$r_L$} (1,0.25) ;
            \draw (-.48,-1) .. controls (-.48,-1) and (-.48,-.5) ..node[below, at start]{$2$} (1,0);
            \draw (-0.16,-1) .. controls (-.16,-1) and (-.16,-.625) ..node[below, at start]{$1$} (1,-.25);
           \draw[red] (.48,-1)--node[below, at start]{$r_R$} 
          (.48,1);
            \draw (.64 ,-1).. controls (.64,-1) and (.64,-.75).. node[below, at start]{$2$} (1,-.5);
          \draw (.8 ,-1).. controls (.8,-1) and (.8,-.875).. node[below, at start]{$3$}(1,-.75);
          \draw[dashed] (1,1)-- (1,-1);
          \draw (-1,-0.75).. controls (-1,-.75) and (-.5,-.75).. node[above, at end]{$3$} (.16,1);
          \draw (-1,-.5).. controls (-1,-.5) and (-.5,-.5).. node[above, at end]{$2$} (0,1);
          \draw[red] (-1,.25).. controls (-1,.25) and (-.9,.25).. (-.8,1);
          \draw (-1,.625).. controls (-1,.625) and (-.58,.625).. node[above, at end]{$1$} (-.16,1);
          \draw (1,.625)-- (-1,-.25);
          \draw (-1,0).. controls (-1,0) and (-.5,0).. (0,1);
          }
   \end{array}&
   &&\Zij{2}{13}{1}&=\begin{array}{c}  
 \tikz[very thick,xscale=2,yscale=2]{
            \draw (-1,-1)-- (1,-1);
            \draw (-1,1)-- (1,1);
          \draw[dashed] (-1,-1)-- (-1,1);
           \draw[red] (-.8,-1) .. controls (-.8,-1) and (-.8,-.2) ..node[below, at start]{$r_L$} (1,.6);
            \draw (-.48,-1) .. controls (-.48,-1) and (0,0) ..node[below, at start]{$2$} (0,1);
            \draw (-0.16,-1) .. controls (-.16,-1) and (-.16,-0.4) ..node[below, at start]{$1$} 
          (1,0.2);
          \draw[red] (.48 ,-1)--node[below, at start]{$r_R$} (.48,1);
           \draw (.64,-1) .. controls (.64,-1) and (.64,-.6) ..node[below, at start]{$2$} (1,-.2);
           \draw (.8,-1).. controls (.8,-1) and (.8,-.8) .. node[below, at start]{$3$} (1,-.6);
          \draw[red] (-1,0.6) .. controls (-1,0.6) and (-0.8,0.8) .. (-0.8,1);
          \draw (-1,0.2) .. controls (-1,0.2) and (-.16,0.6) ..node[above, at end]{$1$} (-.16,1);
          \draw (-1,-.2) .. controls (-1,-.2) and (0,0.4) ..node[above, at end]{$2$} (0,1);
          \draw (-1,-.6) .. controls (-1,-.6) and (0.16,.2) ..node[above, at end]{$3$} (0.16,1);
          \draw[dashed] (1,1)-- (1,-1);
          }
   \end{array}&
\end{align*}

\begin{align*}
&&\Zij{1}{24}{1}&=\begin{array}{c}  
 \tikz[very thick,xscale=2,yscale=2]{
            \draw (-1,-1)-- (1,-1);
            \draw (-1,1)-- (1,1);
          \draw[dashed] (-1,-1)-- (-1,1);
           \draw[red] (-.8,-1).. controls (-.8,-1) and (-.8,-.285).. node[below, at start]{$r_L$} (1,.43);
            \draw (-.48,-1) .. controls (-.48,-1) and (-.48,-.55) ..node[below, at start]{$2$} (1,-.14);
            \draw (-0.16,-1) .. controls (-.16,-1) and (-.16,-.43) ..node[below, at start]{$1$} (1,.14);
           \draw[red] (.48,-1)--node[below, at start]{$r_R$} 
          (.48,1);
            \draw (.64 ,-1).. controls (.64,-1) and (.64,-.715).. node[below, at start]{$2$} (1,-.43);
          \draw (.8 ,-1).. controls (.8,-1) and (.8,-.855).. node[below, at start]{$3$} (1,-.71);
          \draw (-1,.14).. controls (-1,.14) and (-.58,.14)..node[above, at end]{$1$} (-.16,1);
          \draw[dashed] (1,1)-- (1,-1);
          \draw (-1,-.71).. controls (-1,-.71) and (-.42,-.71).. node[above,at end]{$3$}(.16,1);
          \draw (-1,-.43).. controls (-1,-.43) and (-.5,-.43).. node[above,at end]{$2$} (0,1);
          \draw (-1,-.14).. controls (-1,-.14) and (-.5,-.14).. (0,1);
          \draw[red] (-1,.43).. controls (-1,.43) and (-.9,.43).. (-.8,1);
          }
   \end{array}&
   &&\Zij{2}{24}{1}&=\begin{array}{c}  
 \tikz[very thick,xscale=2,yscale=2]{
            \draw (-1,-1)-- (1,-1);
            \draw (-1,1)-- (1,1);
          \draw[dashed] (-1,-1)-- (-1,1);
           \draw[red] (-.8,-1) .. controls (-.8,-1) and (-.8,-.25) ..node[below, at start]{$r_L$} (1,.5);
            \draw (-.48,-1) .. controls (-.48,-1) and (0,0) ..node[below, at start]{$2$} (0,1);
            \draw (-0.16,-1) .. controls (-.16,-1) and (-.16,-0.5) ..node[below, at start]{$1$} 
          (1,0);
          \draw[red] (.48 ,-1)--node[below, at start]{$r_R$} (.48,1);
           \draw (.64,-1) .. controls (.64,-1) and (.64,-.75) ..node[below, at start]{$2$} (1,-.5);
           \draw (.8,-1).. controls (.8,-1) and (.48,-1) .. node[below, at start]{$3$}node[above, at end]{$3$} (.16,1);
          \draw[red] (-1,0.5) .. controls (-1,0.5) and (-0.8,0.75) .. (-0.8,1);
          \draw (-1,0) .. controls (-1,0) and (-.16,0.5) ..node[above, at end]{$1$} (-.16,1);
          \draw (-1,-.5) .. controls (-1,-.5) and (0,0.25) ..node[above, at end]{$2$} (0,1);
          \draw[dashed] (1,1)-- (1,-1);
          }
   \end{array}&
\end{align*}
To compute the transition matrix $\gamma_{24}^{13}$[\textcolor{red}{2}21\textcolor{red}{2}23] we perform the following manipulations:
\begin{multline*}
    \coord_{24}\Zij{1}{13}{1}=\\
    \begin{array}{c}  
 \tikz[very thick,xscale=2,yscale=2]{
            \draw (-1,-1)-- (1,-1);
            \draw (-1,1)-- (1,1);
          \draw[dashed] (-1,-1)-- (-1,1);
           \draw[red] (-.8,-1) .. controls (-.8,-1) and (-.8,-.72) .. node[below, at start]{$r_L$} (1,0.11) ;
            \draw (-.48,-1) .. controls (-.48,-1) and (-.48,-.78) ..node[below, at start]{$2$} (1,-.11);
            \draw (-0.16,-1) .. controls (-.16,-1) and (-.16,-.83) ..node[below, at start]{$1$} (1,-.33);
           \draw[red] (.48,-1)--node[below, at start]{$r_R$} 
          (.48,1);
            \draw (.64 ,-1).. controls (.64,-1) and (.64,-.77).. node[below, at start]{$2$} (1,-.55);
          \draw (.8 ,-1).. controls (.8,-1) and (.8,-.88).. node[below, at start]{$3$}(1,-.78);
          \draw[dashed] (1,1)-- (1,-1);
          \draw (-1,-0.78).. controls (-1,-.78) and (-.5,-.78).. node[above, at end]{$3$} (.16,1);
          \draw (-1,-.56).. controls (-1,-.56) and (-.5,-.56).. node[above, at end]{$2$} (0,1);
          \draw[red] (-1,.78).. controls (-1,.78) and (-.9,.78).. (-.8,1);
          \draw (-1,.33).. controls (-1,.33) and (-.58,.33).. node[above, at end]{$1$} (-.16,1);
          \draw (-1,-.33)-- (1,.33);
          \draw (-1,-.11)-- (1,.56);
          \draw[red] (-1,0.11)-- (1,.78);
          \draw (-1,0.56).. controls (-1,0.56) and (-.5,0.56).. (0,1);
          }
   \end{array}
   =\begin{array}{c}  
 \tikz[very thick,xscale=2,yscale=2]{
            \draw (-1,-1)-- (1,-1);
            \draw (-1,1)-- (1,1);
          \draw[dashed] (-1,-1)-- (-1,1);
           \draw[red] (-.8,-1) .. controls (-.8,-1) and (-.8,-.72) .. node[below, at start]{$r_L$} (1,0.11) ;
            \draw (-.48,-1) .. controls (-.48,-1) and (-.48,-.78) ..node[below, at start]{$2$} (1,-.11);
            \draw (-0.16,-1) .. controls (-.16,-1) and (-.16,-.83) ..node[below, at start]{$1$} (1,-.33);
           \draw[red] (.48,-1)--node[below, at start]{$r_R$} 
          (.48,1);
            \draw (.64 ,-1).. controls (.64,-1) and (.64,-.77).. node[below, at start]{$2$} (1,-.55);
          \draw (.8 ,-1).. controls (.8,-1) and (.8,-.88).. node[below, at start]{$3$}(1,-.78);
          \draw[dashed] (1,1)-- (1,-1);
          \draw (-1,-0.78).. controls (-1,-.78) and (-.5,-.78).. node[above, at end]{$3$} (.16,1);
          \draw (-1,-.11).. controls (-1,-.11) and (-.5,-.11).. node[above, at end]{$2$} (0,1);
          \draw[red] (-1,.78).. controls (-1,.78) and (-.9,.78).. (-.8,1);
          \draw (-1,.33).. controls (-1,.33) and (-.58,.33).. node[above, at end]{$1$} (-.16,1);
          \draw (-1,-.33)-- (1,.56);
          \draw (-1,-.56)-- (1,.33);
          \draw[red] (-1,0.11)-- (1,.78);
          \draw (-1,0.56).. controls (-1,0.56) and (-.5,0.56).. (0,1);
          }
   \end{array}+\begin{array}{c}  
 \tikz[very thick,xscale=2,yscale=2]{
            \draw (-1,-1)-- (1,-1);
            \draw (-1,1)-- (1,1);
          \draw[dashed] (-1,-1)-- (-1,1);
           \draw[red] (-.8,-1) .. controls (-.8,-1) and (-.8,-.72) .. node[below, at start]{$r_L$} (1,0.11) ;
            \draw (-.48,-1) .. controls (-.48,-1) and (-.48,-.83) ..node[below, at start]{$2$} (1,-.33);
            \draw (-0.16,-1) .. controls (-.16,-1) and (-.16,-.78) ..node[below, at start]{$1$} (1,-.11);
           \draw[red] (.48,-1)--node[below, at start]{$r_R$} 
          (.48,1);
            \draw (.64 ,-1).. controls (.64,-1) and (.64,-.77).. node[below, at start]{$2$} (1,-.55);
          \draw (.8 ,-1).. controls (.8,-1) and (.8,-.88).. node[below, at start]{$3$}(1,-.78);
          \draw[dashed] (1,1)-- (1,-1);
          \draw (-1,-0.78).. controls (-1,-.78) and (-.5,-.78).. node[above, at end]{$3$} (.16,1);
          \draw (-1,-.56).. controls (-1,-.56) and (-.5,-.56).. node[above, at end]{$2$} (0,1);
          \draw[red] (-1,.78).. controls (-1,.78) and (-.9,.78).. (-.8,1);
          \draw (-1,.56).. controls (-1,.56) and (-.58,.56).. node[above, at end]{$1$} (-.16,1);
          \draw (-1,-.33)-- (1,.33);
          \draw (-1,-.11)-- (1,.56);
          \draw[red] (-1,0.11)-- (1,.78);
          \draw (-1,0.33).. controls (-1,0.33) and (-.5,0.33).. (0,1);
          }
   \end{array}\\
   =\coord_{14}\Zij{1}{24}{1}+\coord_{12}\Zij{2}{24}{1}.
\end{multline*}
The most notable relation used here was applying \eqref{triple-dumb} with $ijk=212$. Similarly, we have
\begin{multline*}
    \coord_{24}\Zij{2}{13}{1}=\\
    \begin{array}{c}  
 \tikz[very thick,xscale=2,yscale=2]{
            \draw (-1,-1)-- (1,-1);
            \draw (-1,1)-- (1,1);
          \draw[dashed] (-1,-1)-- (-1,1);
           \draw[red] (-.8,-1) .. controls (-.8,-1) and (-.8,-.42) ..node[below, at start]{$r_L$} (1,.14);
            \draw (-.48,-1) .. controls (-.48,-1) and (0,0) ..node[below, at start]{$2$} (0,1);
            \draw (-0.16,-1) .. controls (-.16,-1) and (-.16,-0.57) ..node[below, at start]{$1$} (1,-.14);
          \draw[red] (.48 ,-1)--node[below, at start]{$r_R$} (.48,1);
           \draw (.64,-1) .. controls (.64,-1) and (.64,-.71) ..node[below, at start]{$2$} (1,-.43);
           \draw (.8,-1).. controls (.8,-1) and (.8,-.855) .. node[below, at start]{$3$} (1,-.71);
          \draw[red] (-1,0.6) .. controls (-1,0.6) and (-0.8,0.8) .. (-0.8,1);
          \draw (-1,-.14) .. controls (-1,-.14) and (-.16,0.6) ..node[above, at end]{$1$} (-.16,1);
          \draw (-1,.43) .. controls (-1,.43) and (0,0.4) ..node[above, at end]{$2$} (0,1);
          \draw (-1,-.71) .. controls (-1,-.71) and (0.5,-.25) ..node[above, at end]{$3$} (0.16,1);
          \draw[dashed] (1,1)-- (1,-1);
          \draw (-1,-.43)-- (1,.43); 
          \draw[red] (-1, .14)-- (1,.71);
          }
   \end{array}
   =\begin{array}{c}  
 \tikz[very thick,xscale=2,yscale=2]{
            \draw (-1,-1)-- (1,-1);
            \draw (-1,1)-- (1,1);
          \draw[dashed] (-1,-1)-- (-1,1);
           \draw[red] (-.8,-1) .. controls (-.8,-1) and (-.8,-.42) ..node[below, at start]{$r_L$} (1,.14);
            \draw (-.48,-1) .. controls (-.48,-1) and (0,0) ..node[below, at start]{$2$} (0,1);
            \draw (-0.16,-1) .. controls (-.16,-1) and (-.16,-0.57) ..node[below, at start]{$1$} (1,-.14);
          \draw[red] (.48 ,-1)--node[below, at start]{$r_R$} (.48,1);
           \draw (.64,-1) .. controls (.64,-1) and (.64,-.71) ..node[below, at start]{$2$} (1,-.43);
           \draw (.8,-1).. controls (.8,-1) and (.8,-.855) .. node[below, at start]{$3$} (1,-.71);
          \draw[red] (-1,0.6) .. controls (-1,0.6) and (-0.8,0.8) .. (-0.8,1);
          \draw (-1,-.14) .. controls (-1,-.14) and (-.16,0.6) ..node[above, at end]{$1$} (-.16,1);
          \draw (-1,.43) .. controls (-1,.43) and (0,0.4) ..node[above, at end]{$2$} (0,1);
          \draw (-1,-.71) .. controls (-1,-.71) and (0.2,0) ..node[above, at end]{$3$} (0.16,1);
          \draw[dashed] (1,1)-- (1,-1);
          \draw (-1,-.43)-- (1,.43); 
          \draw[red] (-1, .14)-- (1,.71);
          }
   \end{array}-\begin{array}{c}  
 \tikz[very thick,xscale=2,yscale=2]{
            \draw (-1,-1)-- (1,-1);
            \draw (-1,1)-- (1,1);
          \draw[dashed] (-1,-1)-- (-1,1);
           \draw[red] (-.8,-1) .. controls (-.8,-1) and (-.8,-.42) ..node[below, at start]{$r_L$} (1,.14);
            \draw (-.48,-1) .. controls (-.48,-1) and (0,0) ..node[below, at start]{$2$} (0,1);
            \draw (-0.16,-1) .. controls (-.16,-1) and (-.16,-0.57) ..node[below, at start]{$1$} (1,-.14);
          \draw[red] (.48 ,-1)--node[below, at start]{$r_R$} (.48,1);
           \draw (.64,-1) .. controls (.64,-1) and (.64,-.855) ..node[below, at start]{$2$} (1,-.71);
           \draw (.8,-1).. controls (.8,-1) and (.8,-.71) .. node[below, at start]{$3$} (1,-.43);
          \draw[red] (-1,0.6) .. controls (-1,0.6) and (-0.8,0.8) .. (-0.8,1);
          \draw (-1,-.14) .. controls (-1,-.14) and (-.16,0.6) ..node[above, at end]{$1$} (-.16,1);
          \draw (-1,.43) .. controls (-1,.43) and (0,0.4) ..node[above, at end]{$2$} (0,1);
          \draw (-1,-.43) .. controls (-1,-.43) and (-.5,-.5) ..node[above, at end]{$3$} (0.16,1);
          \draw[dashed] (1,1)-- (1,-1);
          \draw (-1,-.71)-- (1,.43); 
          \draw[red] (-1, .14)-- (1,.71);
          }
   \end{array}\\
   =-\coord_{34}\Zij{1}{24}{1}+\coord_{23}\Zij{2}{24}{1}.
\end{multline*}
The most notable relation used here was \eqref{triple-dumb} with $ijk=232$.
These equations then give us the following transition matrix:
\begin{align*}
    &&\gamma_{13}^{24}[\textcolor{red}{2}21\textcolor{red}{2}23]&=\left(\begin{array}{cc}
        \coord_{14}/\coord_{13} & \coord_{12}/\coord_{13} \\
         -\coord_{34}/\coord_{13} & \coord_{23}/\coord_{13} 
    \end{array}\right).&
\end{align*}
Lemma \ref{transition-lemma-2} then tells us that:
\begin{lemma}\label{lem:334}
    $\tilting^{\textcolor{red}{2}21\textcolor{red}{2}23}\cong\taut^\perp$.
\end{lemma}
\subsection{Summary}
Combining Lemmata \ref{lem:line-isos}, \ref{lem:331}, \ref{lem:332}, \ref{lem:333}, and \ref{lem:334}, we summarize the results of the previous two sections in the following table:
\begin{center}
\begin{tabular}{ |c|c| } 
 \hline
 Idempotent  & Vector bundle \\ 
 \hline
\textcolor{red}{2}$12^{(2)}3$\textcolor{red}{2} & $\mathcal{O}$ \\ 
\textcolor{red}{2}\textcolor{red}{2}$12^{(2)}3$&  $\mathcal{L}^{-1}$ \\
\textcolor{red}{2}2132\textcolor{red}{2} & $\mathcal{H}\otimes\mathcal{L}$\\
\textcolor{red}{22}2312 & $\mathcal{H}$\\
\textcolor{red}{2}23\textcolor{red}{2}21 & $\taut^\perp$ \\
\textcolor{red}{2}21\textcolor{red}{2}23 & $\taut$ \\
 \hline
\end{tabular}
\end{center}
With these identifications established, we are now in a position to conclude our main theorem:
\begin{theorem}
The vector bundle
\begin{align*}
    \tilting&=\tilting^{\textcolor{red}{2}12^{(2)}3\textcolor{red}{2}}\oplus \tilting^{\textcolor{red}{2}\textcolor{red}{2}12^{(2)}3}\oplus \tilting^{\textcolor{red}{2}2132\textcolor{red}{2}}\oplus \tilting^{\textcolor{red}{22}2312 }\oplus \tilting^{\textcolor{red}{2}23\textcolor{red}{2}21}\oplus \tilting^{\textcolor{red}{2}21\textcolor{red}{2}23 }&\\
    &\cong \mathcal{O}\oplus\mathcal{L}^{-1}\oplus\mathcal{H}\oplus\mathcal{H}\otimes\mathcal{L}\oplus\taut\oplus\taut^\perp&
\end{align*}
is isomorphic to the tilting generator of \Cref{th:Z-tilting} on $D^b(\mathsf{Coh}(T^*X))$.
\end{theorem}

\section{Comparison to other constructions}

\subsection{Comparison to Tseu}
\label{tseu}

In \cite[Th. 1]{tseuCanonicalTilting2024}, Tseu constructs a tilting generator on in the case where $k=2$ and $n$ is arbitrary.  In the case $n=4$, this is shown to coincide with our example up to tensor product with a line bundle in \cite[Ex. 1]{tseuCanonicalTilting2024}.  
To match notation between our papers:
\newcommand{\cT}{\mathcal{T}}
\begin{align}
	\cL^{-1}\otimes \mathcal{E}_{(-1,-1)}& =\mathcal{O} =\tilting^{\textcolor{red}{2}12^{(2)}3\textcolor{red}{2}} & \cL^{-1}\otimes  \mathcal{E}_{1} &= \mathcal{H}\otimes \cL=\tilting^{\textcolor{red}{22}2312 }\\
	\cL^{-1}\otimes 	\mathcal{E}_{(0,-1)}& = \cT =\tilting^{\textcolor{red}{2}23\textcolor{red}{2}21}& \cL^{-1}\otimes \mathcal{E}_{0} &= \mathcal{T}^{\perp}=\tilting^{\textcolor{red}{2}21\textcolor{red}{2}23 }\\
	\cL^{-1}\otimes 	\mathcal{E}_{(0,0)} & = \cL^{-1}=\tilting^{\textcolor{red}{2}\textcolor{red}{2}12^{(2)}3} & \cL^{-1}\otimes \mathcal{E}_{-1}& = \mathcal{H}=\tilting^{\textcolor{red}{2}2132\textcolor{red}{2}}
\end{align}

It seems quite likely that our construction agrees with Tseu's for all values of $n$, leading us to conjecture:
\begin{conjecture}
	For all $n$, the tilting generator on $T^*\Gr(2,n)$ constructed in \cite[Th. 1]{tseuCanonicalTilting2024} is equiconstituted to $\tilting$ up to tensor product with a line bundle.  
\end{conjecture}
One connection that Tseu emphasizes strongly is the interplay of his bundle and the stratified Mukai flop denoted $\mathbb{T}$ there.  
We believe that there is a good explanation for this connection.  The two different symplectic resolutions $\fM,\fM'$ correspond to $\Proj$'s of the algebras $\mathbf{A}=\bigoplus_{m=0}^{\infty}e\mathring{T}^me$ and $\mathbf{A}'=\bigoplus_{m=0}^{\infty}e\mathring{T}^{-m}e$ where we reverse the direction of rotation of the first red strand. We can apply the localization construction on both resolutions to get tilting bundles with endomorphism ring $\mathring{R}$ on both $\fM$ and $\fM'$.  The Cautis--Kamnitzer--Licata equivalence (which is really only unique up to tensor with a line bundle) can be chosen to send one such tilting bundle to the other.  

The Grassmannian case is quite special from this perspective, since the endomorphism ring of the resulting tilting generator depends on the cyclic ordering of red points on the circle (which in terms of characteristic $p$ quantization is the choice of quantization parameter).  Of course, there is only one cyclic order of two points on a circle, so the tilting generator $\tilting$ is unique up to tensor product with a line bundle, whereas other Coulomb branches can have many different tilting bundles with different endomorphism algebras.  

\subsection{Comparison to Toda-Uehara}

Let us discuss in a bit more detail the comparison of our construction with that of \cite{todaTiltingGenerators2010}.  They construct a tilting generator by an inductive process.  Define a series of vector bundles $\mathcal{E}_k$ by the construction:
\begin{algorithm}\hfill
\begin{enumerate}
    \item Let $\mathcal{E}_0=\mathcal{O}_X$.
    \item Given $\mathcal{E}_{k-1}$ for $0<k<4$, consider $\Ext^1(\mathcal{E}_{k-1},\mathcal{L}^{-k}$; this is the only nonzero Ext group between this pair and represents the only obstruction to $\mathcal{E}_{k-1}\oplus \mathcal{L}^{-k}$ being a tilting bundle.

As a module over $\End(\mathcal{E}_{k-1})$, the group $\Ext^1(\mathcal{E}_{k-1},\mathcal{L}^{-k})$ is finitely generated and any set of $r$ generators gives an extension \[\mathcal{L}^{-k}\to \mathcal{N}_{k}\to \mathcal{E}_{k-1}^{\oplus r}\] such that $\Ext^1(\mathcal{E}_{k-1},\mathcal{N}_{k})=0$.  

Note that we can often find a smaller summand $\mathcal{N}_{k}'$ of $\mathcal{N}_{k}$ containing $\mathcal{L}^{-k}$; the complementary summand is already a summand of $\mathcal{E}_{k-1}^{\oplus r}$, so $\mathcal{E}_{k-1}\oplus \mathcal{N}_{k}$ and $\mathcal{E}_{k-1}\oplus \mathcal{N}_{k}'$ are equiconstituted.  In particular, this shows that if we have two generating sets, the resulting bundles $\mathcal{E}_{k-1}\oplus \mathcal{N}_{k}$ are equiconsituted. 

Let $\mathcal{E}_{k}=\mathcal{E}_{k-1}\oplus \mathcal{N}_{k}'$, which is unique up to equiconsitution.

\item If $k=4$, we perform a more complicated process to cancel all  $\Ext^i(\mathcal{E}_{3},\mathcal{L}^{-4})$ with $i>0$; this is necessary because this Ext-group can be non-zero for $i\in [1,4].$
    \end{enumerate}
\end{algorithm}
\begin{proposition}\label{prop:TU-comparison}
    We can make choices so that $\mathcal{E}_3\cong \mathcal{O}_X\oplus \mathcal{L}^{-1}\oplus \mathcal{H}\oplus \mathcal{H}\otimes \mathcal{L}^{-1}$.  That is, $\mathcal{E}_{3}$ is a summand of $\tilting$.
\end{proposition}
\begin{proof}
    By definition, we have $\mathcal{E}_0=\mathcal{O}_X$.  Furthermore, 
$\mathcal{E}_1=\mathcal{O}_X\oplus \mathcal{L}^{-1}$ since $\Ext^1(\mathcal{E}_{0},(\mathcal{L}^{-1})=0$.

On the other hand, by \eqref{eq:ext1}, we have $\Ext^1(\mathcal{E}_1,\mathcal{L}^{-2})=H^0(T^*X;\mathcal{O})$, and so $\mathcal{N}_2$ comes from the unique lowest-degree class, which is a generator of this module.  That is,  \[\mathcal{N}_2\cong \mathcal{H}\oplus \mathcal{L}^{-1}.\]  This is the first example where it is more convenient to take $\mathcal{N}_2'\cong \mathcal{H}$, so that we have 
\[\mathcal{E}_2\cong \mathcal{O}_X\oplus \mathcal{L}^{-1}\oplus \mathcal{H}.\]  Taking $\mathcal{N}_2$ instead would have just given us a second copy of $\mathcal{L}^{-1}$.

To construct $\mathcal{E}_3$, consider \begin{align*}
\Ext^1(\mathcal{E}_2,\mathcal{L}^{-3})&\cong \Ext^1(\mathcal{O},\mathcal{L}^{-3})\oplus \Ext^1(\mathcal{L}^{-1},\mathcal{L}^{-3})\oplus \Ext^1(\mathcal{H},\mathcal{L}^{-3})\\
    &\cong H^0(T^*X;\mathcal{L})\oplus H^0(T^*X;\mathcal{O})\oplus H^0(T^*X;\mathcal{L})
\end{align*} 
We have a long exact sequence \[\cdots \longrightarrow\Ext^1(\mathcal{O},\mathcal{L}^{-3}) \longrightarrow{ \Ext^1(\mathcal{H},\mathcal{L}^{-3})} \longrightarrow { \Ext^1(\mathcal{L}^{-2},\mathcal{L}^{-3})}=0.\]
In particular, the module $\Ext^1(\mathcal{E}_3,\mathcal{L}^{-3})$ is generated by the class in $\Ext^1(\mathcal{L}^{-1},\mathcal{L}^{-3})$ corresponding to $\mathcal{H}\otimes \mathcal{L}^{-1}$. All other classes in other summands come from this one using elements of $\Hom(\mathcal{L}^{-1},\mathcal{O})\cong \Hom(\mathcal{L}^{-1},\mathcal{H})\cong H^0(T^*X;\mathcal{L})$.  
\end{proof}

\section*{Acknowledgements}

Thanks to Wei Tseu for discussing his work with us while it was still in progress and for helpful references.  Many thanks to Joel Kamnitzer, Travis Schedler, Alex Weekes, Oded Yacobi, and Baorui Zhou for useful discussions on
these topics.

Supported by the NSERC through a Discovery Grant and by MITACS through a Globalink Research Award. This research was supported in part by Perimeter Institute for Theoretical Physics. Research at Perimeter Institute is supported in part by the Government of Canada through the Department of Innovation, Science and Economic Development Canada and by the Province of Ontario through the Ministry of Colleges and Universities.

{\renewcommand{\markboth}[2]{}
\printbibliography}
\end{document}